\documentclass[12pt,a4paper,reqno,oneside]{amsart}
 
\usepackage[utf8]{inputenc}
\usepackage[T1]{fontenc}
\usepackage{lmodern}
\usepackage{amsmath}
\usepackage{amssymb}
\usepackage{amsthm}
\allowdisplaybreaks
\usepackage{geometry}
\usepackage{soul}
\usepackage{graphicx}
\usepackage{bbm}
\usepackage{float}
\usepackage{enumerate}
\usepackage{scrhack}
\usepackage{hyperref}
\usepackage{cleveref}

\newtheoremstyle{theorem}	
   {15pt}       							
   {15pt}      							
   {\itshape}  							
   {\parindent} 							
   {\bfseries} 							
   {}         									
   {.7em}      							
   {}          								

\newtheoremstyle{definition}	
  {15pt}       							
  {15pt}      								
  {} 		 									
  {\parindent} 							
  {\bfseries} 							
  {}         									
  {.7em}      								
  {}          									
  
\newtheoremstyle{remark}		
  {15pt}       							
  {15pt}      								
  {} 											
  {\parindent} 							
  {\itshape} 								
  {.}         									
  {.7em}      								
  {}          									

\theoremstyle{theorem}
\newtheorem{theorem}{Theorem}[section]
\newtheorem{corollary}[theorem]{Corollary}
\newtheorem{lemma}[theorem]{Lemma}
\newtheorem{proposition}[theorem]{Proposition}

\theoremstyle{remark}
\newtheorem{remark}[theorem]{Remark}

\theoremstyle{definition}

\newtheorem{definition}[theorem]{Definition}

\newcommand{\Bb}{\mathbb{B}}
\newcommand{\Dbb}{\mathbb{D}}
\newcommand{\Ebb}{\mathbb{E}}
\newcommand{\Fbb}{\mathbb{F}}

\newcommand{\Nbb}{\mathbb{N}}
\newcommand{\Pbb}{\mathbb{P}}
\newcommand{\Qbb}{\mathbb{Q}}
\newcommand{\Rbb}{\mathbb{R}}

\newcommand{\Acal}{\mathcal{A}}

\newcommand{\Ccal}{\mathcal{C}}
\newcommand{\Dcal}{\mathcal{D}}
\newcommand{\Ecal}{\mathcal{E}}
\newcommand{\Fcal}{\mathcal{F}}
\newcommand{\Gcal}{\mathcal{G}}
\newcommand{\Hcal}{\mathcal{H}}

\newcommand{\Kcal}{\mathcal{K}}
\newcommand{\Lcal}{\mathcal{L}}

\newcommand{\Ncal}{\mathcal{N}}

\newcommand{\Scal}{\mathcal{S}}
\newcommand{\Wcal}{\mathcal{W}}
\newcommand{\Xcal}{\mathcal{X}}
\newcommand{\Ycal}{\mathcal{Y}}

\newcommand{\Afrak}{\mathfrak{A}}
\newcommand{\Bfrak}{\mathfrak{B}}

\newcommand{\Ifrak}{\mathfrak{I}}

\newcommand{\Kfrak}{\mathfrak{K}}
\newcommand{\Lfrak}{\mathfrak{L}}
\newcommand{\Nfrak}{\mathfrak{N}}

\DeclareMathOperator*{\VAR}{Var}
\DeclareMathOperator*{\COV}{Cov}
\newcommand{\Lip}{\text{Lip}}

\DeclareMathOperator*{\Perm}{perm}
\DeclareMathOperator*{\Id}{Id}

\newcommand{\btilde}{\widetilde{b}}
\newcommand{\Btilde}{\widetilde{B}}

\newcommand{\etilde}{\widetilde{e}}
\newcommand{\Gtilde}{\widetilde{G}}

\newcommand{\Xtilde}{\widetilde{X}}
\newcommand{\ytilde}{\widetilde{y}}

\newcommand{\pitilde}{\widetilde{\pi}}

\newcommand{\Bhat}{\widehat{B}}

\newcommand{\Ep}[2]{\Ebb\negthinspace\left[\left\vert #1 \right\vert^{#2} \right]^{\frac{1}{#2}}\negthinspace}
\newcommand{\EpO}[2]{\Ebb\negthinspace\left[\left\vert #1 \right\vert^{#2} \right]\negthinspace}
\newcommand{\Eabs}[1]{\Ebb\negthinspace\left[\left\vert #1 \right\vert \right]\negthinspace}
\newcommand{\EW}[1]{\Ebb\negthinspace\left[ #1 \right]\negthinspace}

\newcommand{\EH}[1]{\Ebb\negthinspace\left[\left\Vert #1 \right\Vert_{\Hcal}^2 \right]\negthinspace}
\newcommand{\EHp}[1]{\Ebb\negthinspace\left[\left\Vert #1 \right\Vert_{\Hcal}^2 \right]^{\frac{1}{2}}\negthinspace}
\newcommand{\ETp}[2]{\Ebb \negthinspace\left[ \int_0^T \left\Vert #1 \right\Vert_{\Hcal}^{#2} dt\right]^{\frac{1}{#2}}\negthinspace}

\newcommand{\ND}[1]{\left\Vert #1 \right\Vert_{\Dbb^{1,2}(\Hcal)}^2}
\newcommand{\NDO}[1]{\left\Vert #1 \right\Vert_{\Dbb^{1,2}(\Hcal)}}

\newcommand{\Var}[1]{\VAR \negthinspace\left( #1 \right)\negthinspace}
\newcommand{\Cov}[1]{\COV \negthinspace\left( #1 \right)\negthinspace}
\newcommand{\perm}[1]{\Perm \negthinspace \left( #1 \right) \negthinspace}

\usepackage{color}
\definecolor{darkgreen}{rgb}{0, .5, 0}
\definecolor{darkred}{rgb}{.5, 0, 0}

\newcommand\blfootnote[1]{%
  \begingroup
  \renewcommand\thefootnote{}\footnote{#1}%
  \addtocounter{footnote}{-1}%
  \endgroup
}

\begin{document}


\title[Restoration of Well-Posedness of Inf-dim Singular ODE's via Noise]{Restoration of Well-Posedness of Infinite-dimensional Singular ODE's via Noise}
\author[D. Ba\~nos]{David Ba\~nos}
\address{D. Ba\~nos: Department of Mathematics, University of Oslo, Moltke Moes vei 35, P.O. Box 1053 Blindern, 0316 Oslo, Norway.}
\email{davidru@math.uio.no}
\author[M.Bauer]{Martin Bauer}
\address{M. Bauer: Department of Mathematics, LMU, Theresienstr. 39, D-80333 Munich, Germany.}
\email{bauer@math.lmu.de}
\author[T. Meyer-Brandis]{Thilo Meyer-Brandis}
\address{T. Meyer-Brandis: Department of Mathematics, LMU, Theresienstr. 39, D-80333 Munich, Germany.}
\email{meyerbra@math.lmu.de} 
\author[F. Proske]{Frank Proske}
\address{F. Proske: Department of Mathematics, University of Oslo, Moltke Moes vei 35, P.O. Box 1053 Blindern, 0316 Oslo, Norway.}
\email{proske@math.uio.no}
\maketitle
\blfootnote{The research is partially supported by the FINEWSTOCH (NFR-ISP) project.}


\textbf{Abstract.} In this paper we aim at generalizing the results of A. K. Zvonkin \cite{Zvon74} and A. Y. Veretennikov \cite{Ver79} on the construction of unique strong solutions of stochastic differential equations with singular drift vector field and additive noise in the Euclidean space to the case of infinite-dimensional state spaces. The regularizing driving noise in our equation is chosen to be a locally non-H\"{o}lder continuous Hilbert space valued process of fractal nature, which does not allow for the use of classical construction techniques for strong solutions from PDE or semimartingale theory. Our approach, which does not resort to the Yamada-Watanabe principle for the verification of pathwise uniqueness of solutions, is based on Malliavin calculus. \\[0.5cm]

\textbf{Keywords.} Malliavin calculus; fractional Brownian motion; $L^2$-compactness criterion; strong solutions of SDEs; irregular drift coefficient

\section{Introduction}
	The main objective of this paper is the construction of (unique) strong solutions of infinite-dimensional stochastic differential equations (SDEs) with a singular drift and additive noise. In fact, we want to derive our results from the perspective of a rather recently established theory of stochastic regularization (see \cite{Flandoli} and the references therein) with respect to a new general method based on Malliavin calculus and another variational technique which can be applied to different types of SDEs and stochastic partial differential equations (SPDEs).

In order to explain the concept of stochastic regularization, let us consider the first-order ordinary differential equation (ODE)

\begin{equation}\label{ODE}
	\frac{d}{dt}X_{t}^{x}=b(t,X_{t}^{x}),\quad X_{0}=x\in \mathcal{H},\quad t \in [0,T]
\end{equation}
	for a vector field $b:[0,T]\times \mathcal{H}\rightarrow \mathcal{H}$, where $\mathcal{H}$ is a separable Hilbert space with norm $\left\Vert \cdot \right\Vert_\Hcal $.

Using Picard iteration, it is fairly straight forward to see that the ODE \eqref{ODE} has a unique (global) solution $(X_t^{x})_{ t\in [0, T]}$, if the driving vector field $b$ satisfies a linear growth and Lipschitz condition, that is 
\begin{equation*}
	\left\Vert b(t,x)\right\Vert_\Hcal \leq C_{1}(1+\left\Vert x\right\Vert_\Hcal )
\end{equation*}
	and
\begin{equation*}
	\left\Vert b(t,x)-b(t,y)\right\Vert_\Hcal \leq C_{2}\left\Vert x-y\right\Vert_\Hcal
\end{equation*}
	for all $t,x$ and $y$ with constants $C_{1},C_{2}<\infty $.

However, well-posedness in the sense of existence and uniqueness of solutions may fail, if the vector field $b$ lacks regularity, that is if e.g.~$b$ is not Lipschitz continuous. In this case, the ODE \eqref{ODE} may not even admit the existence of a solution in the case $\mathcal{H}=\mathbb{R}^{d}$.

On the other hand, the situation changes, if one integrates on both sides of the ODE (\ref{ODE}) and adds a "regularizing" noise to the right hand side of the resulting integral equation.

More precisely, if $\mathcal{H}=\mathbb{R}^{d}$, well-posedness of the ODE \eqref{ODE} can be restored via regularization by a Brownian (additive) noise, that is by a perturbation of the ODE \eqref{ODE} given by the SDE
\begin{equation}\label{SmallNoise}
	dX_{t}^{x}=b(t,X_{t}^{x})dt+\varepsilon dB_{t}, \quad t\in [0,T], \quad X_0^x = x \in \mathbb{R}^{d},
\end{equation}%
where $(B_t)_{t\in[0,T]}$ is a Brownian motion in $\mathbb{R}^{d}$ and $\varepsilon >0$.

If the vector field $b$ is merely bounded and measurable, it turns out that the SDE \eqref{SmallNoise} -- regardless how small $\varepsilon $ is -- possesses a unique (global) strong solution, that is a solution $(X_t^x)_{t\in[0,T]}$, which as a process is a measurable functional of the driving noise $(B_t)_{t\in[0,T]}$. This surprising and remarkable result was first obtained by A. K. Zvonkin \cite{Zvon74} in the one-dimensional case, whose proof, using PDE techniques, is based on a transformation ("Zvonkin-transformation"), that converts the SDE \eqref{SmallNoise} into a SDE without drift part. Subsequently, this result was generalized by A. Y. Veretennikov \cite{Ver79} to the multi-dimensional case. Much later, that is 35 years later, Zvonkin's and Veretennikov's results were extended by G. Da Prato, F. Flandoli, E. Priola and M. R\"{o}ckner \cite{DPFPR13} to the infinite-dimensional setting by using estimates of solutions of Kolmogorov's equation on Hilbert spaces. In fact, the latter authors study mild solutions $(X_t)_{t\in[0,T]}$ to the SDE
\begin{equation*}
	dX_{t}=AX_{t}dt+b(X_{t})dt+\sqrt{Q}dW_{t}, \quad t\in [0,T], \quad X_{0}=x\in \mathcal{H}\text{,}
\end{equation*}%
	where $(W_t)_{t\in[0,T]}$ is a cylindrical Brownian motion on $\mathcal{H}$, $A:D(A)\rightarrow \mathcal{H}$ a negative self-adjoint operator with compact resolvent, $Q:\mathcal{H}\rightarrow \mathcal{H}$ a non-negative definite self-adjoint bounded operator and $b:\mathcal{H}\rightarrow \mathcal{H}$. Here, the authors prove for $b\in L^{\infty }(\mathcal{H};\mathcal{H})$ under certain conditions on $A$ and $Q$ the existence of a unique mild solution, which is adapted to a completed filtration generated by $(W_t)_{t\in[0,T]}$. So restoration of well-posedness of the ODE \eqref{ODE} with a singular vector field is established via regularization by \emph{both} the cylindrical Brownian noise $(W_t)_{t\in[0,T]}$ and $A$, which cannot be chosen to be the zero operator.

Other works in this direction in the infinite-dimensional setting based on different methods are e.g. A. S. Sznitman \cite{Snitzman_TopicsInPropagationOfChaos}, A.\ Y. Pilipenko, M. V. Tantsyura \cite{PT} in connection with systems of McKean-Vlasov equations and G. Ritter, G. Leha \cite{LR} in the case of discontinuous drift vector fields of a rather specific form. We also refer to the references therein.

In this article, we aim at restoring well-posedness of singular ODE's by using a certain non-H\"{o}lder continuous additive noise of fractal nature. More specifically, we want to analyze solutions to the following type of SDE:

\begin{equation}\label{eq:MainSDE}
	X_{t}^{x}=x+\int_{0}^{t}b(t,X_{s}^{x})ds+\mathbb{B}_{t},\quad t \in[0,T],
\end{equation}
	where the $\mathcal{H}-$valued regularizing noise $(\mathbb{B}_{t})_{t\in[0,T]}$ is a stationary Gaussian process with locally non-H\"{o}lder continuous paths given by 
\begin{align*}
	\mathbb{B}_{t}=\sum_{k\geq 1}\lambda_k B_{t}^{H_k}e_k.
\end{align*}
	Here $\lbrace \lambda_k \rbrace_{k\geq 1}\subset \mathbb{R}$, $\lbrace e_k \rbrace_{k\geq 1}$ is an orthonormal basis of $\mathcal{H}$ and $\lbrace B_{\cdot }^{H_k} \rbrace_{k\geq 1}$ are independent one-dimensional fractional Brownian motions with Hurst parameters $H_k \in (0,\frac{1}{2})$, $k\geq 1$, such that 
\begin{equation*}
	H_k \searrow 0
\end{equation*}
	for $k \rightarrow \infty $.

Under certain (rather mild) growth conditions on the Fourier components $b_k$, $k\geq 1$, of the singular vector field $b:[0,T]\times \mathcal{H}\rightarrow \mathcal{H}$ (see \eqref{eq:conditionDrift} and \eqref{eq:KOperator}), which do not necessarily require that all $b_k$ are equal (compare e.g. to \cite{Snitzman_TopicsInPropagationOfChaos}), we show in this paper the existence of a unique (global) strong solution to the SDE \eqref{eq:MainSDE} driven by the non-Markovian process $(\mathbb{B}_{t})_{t\in[0,T]}$.

\bigskip Our approach for the construction of strong solutions to \eqref{eq:MainSDE} relies on Malliavin calculus (see e.g. D. Nualart \cite{Nualart_MalliavinCalculus}) and another variational technique, which involves the use of spatial regularity of local time of finite-dimensional approximations of $\mathbb{B}_{t}$. In contrast to the above mentioned works (and most of other related works in the literature), the method in this paper is not based on PDE, Markov or semimartingale techniques. Furthermore, our technique corresponds to a construction principle, which is diametrically opposed to the commonly used Yamada-Watanabe principle (see e.g.~\cite{YW71}): Using the Yamada-Watanabe principle, one combines the existence of a weak solution to a SDE with pathwise uniqueness to obtain strong uniqueness of solutions. So
\begin{equation*}
	\fbox{Weak existence}+\text{\fbox{Pathwise uniqueness}}\Rightarrow \text{\fbox{Strong uniqueness}.}
\end{equation*}
	This tool is in fact used by many authors in the literature. See e.g. the above mentioned authors or I. Gy\"{o}ngy, T. Martinez \cite{GyM01}, I. Gy\"{o}ngy, N. V. Krylov \cite{GyK96}, N. V. Krylov, M. R\"{o}ckner \cite{KR05} or S. Fang,\ T. S. Zhang \cite{FZ}, just to mention a few.

\bigskip However, using our approach, verification of the existence of a strong solution, which is unique in law, provides strong uniqueness:

\begin{equation*}
	\fbox{Strong existence}+\text{\fbox{Uniqueness in law}}\Rightarrow \text{\fbox{Strong uniqueness}.}
\end{equation*}
	See also H. J. Engelbert \cite{Engel} in the finite-dimensional Brownian case regarding the latter construction principle.

	In order to briefly explain our method in the case of time-homogeneous vector fields, we mention that we apply an infinite-dimensional generalization of a compactness criterion for square integrable Brownian functionals in $L^{2}(\Omega )$, which is originally due to G. Da Prato, P. Malliavin, and D. Nualart \cite{Nualart_MalliavinCalculus}, to a double-sequence of strong solutions $\lbrace(X_{t}^{d,\varepsilon })_{t\in[0,T]} \rbrace_{d\geq 1,\varepsilon >0}$ associated with the following SDE's
\begin{equation}\label{ApproxSDE}
	X_{t}^{d,\varepsilon }=x + \int_0^t b^{d,\varepsilon }(X_{s}^{d,\varepsilon })ds + \mathbb{B}_{t},\quad t\in [0,T].
\end{equation}
	Here $\lbrace b^{d,\varepsilon }\rbrace_{d\in \mathbb{N},\varepsilon >0}$ is an approximating double-sequence of vector fields of the singular drift $b$, which are smooth and live on $d-$dimensional subspaces of $\mathcal{H}$.

The application of the above mentioned compactness criterion (for each fixed $t$), however, requires certain (uniform) estimates with respect to the Malliavin derivative $D_{t}$ of $X_{t}^{d,\varepsilon }$ in the direction of a cylindrical Brownian motion. For this purpose, the Malliavin derivative $D_{\cdot }:\mathbb{D}^{1,2}(\mathcal{H})\longrightarrow L^{2}([0,T]\times \Omega )\otimes \mathcal{L}_{HS}(\mathcal{H},\mathcal{H})$ ($\mathbb{D}^{1,2}(\mathcal{H})$ is the space of $\mathcal{H}-$valued Malliavin differentiable random variables and $\mathcal{L}_{HS}(\mathcal{H},\mathcal{H})$ is the space of Hilbert-Schmidt operators from $\mathcal{H}$ to $\mathcal{H}$) in connection with a chain rule is applied to both sides of (\ref{ApproxSDE}) and one obtains the following linear equation:
\begin{equation}\label{DEq}
	D_{s}X_{t}^{d,\varepsilon }=\int_{s}^{t}\left( b^{d,\varepsilon }\right)^{\prime }(X_{u}^{d,\varepsilon })D_{s}X_{u}^{d,\varepsilon }du+\sum_{n\geq 1}\lambda _{n}K_{H_{n}}(t,s)\left\langle e_{n},\cdot \right\rangle _{\mathcal{H}}e_{n}, ~ s<t, 
\end{equation}
	where $\left( b^{d,\varepsilon }\right) ^{\prime }$ is the derivative of $b^{d,\varepsilon }$, $\left\langle \cdot ,\cdot \right\rangle _{\mathcal{H}}$ the inner product and $K_{H}$ a certain kernel function defined for Hurst parameters $H_n \in (0,\frac{1}{2})$.

	We remark here that this type of linearization based on a stochastic derivative $D_{t}$ actually corresponds to the Nash-Moser principle, which is used for the construction of solutions of (non-linear) PDE's by means of linearization of equations via classical derivatives. See e.g. J. Moser \cite{Moser}.

	In a next step we then can derive a representation of $D_{s}X_{t}^{d,\varepsilon }$ (under a Girsanov change of measure) in \eqref{DEq} which is not based on derivatives of $b^{d,\varepsilon }$ by using Picard iteration and the following variational argument:
\begin{align*}
	\int_{t<s_{1}<...<s_{n}<u}\kappa (s)D^{\alpha }f(\mathbb{B}_{s}^{d})ds &= \int_{\mathbb{R}^{dn}}D^{\alpha }f(z)L_{\kappa}^{n}(t,z)dz \notag \\
	&= (-1)^{\left\vert \alpha \right\vert }\int_{\mathbb{R}^{dn}}f(z)D^{\alpha }L_{\kappa }^{n}(t,z)dz,
\end{align*}
	where $\mathbb{B}_{s}^{d}:=(B_{s_{1}}^{H_{1}},...,B_{s_{1}}^{H_{d}},...,B_{s_{n}}^{H_{1}},...,B_{s_{n}}^{H_{d}})$ and $f:\mathbb{R}^{dn}\longrightarrow \mathbb{R}$ is a smooth function with compact support. Here $D^{\alpha }$ stands for a partial derivative of order $\left\vert \alpha \right\vert $ with respect a multi-index $\alpha $. Further, $L_{\kappa }^{n}(t,z)$ is a spatially differentiable local time of $\mathbb{B}_{\cdot }^{d}$ on a simplex scaled by non-negative integrable function $\kappa (s)=$ $\kappa _{1}(s)...\kappa _{n}(s)$.

	Then, using the latter we can verify the required estimates for the Malliavin derivative of the approximating solutions in connection with the above mentioned compactness criterion and we finally obtain (under some additional arguments) that for each fixed $t$
\begin{equation*}
	X_{t}^{d,\varepsilon }\longrightarrow X_{t}\text{ in }L^{2}(\Omega )
\end{equation*}
	for $\varepsilon \searrow 0,d\longrightarrow \infty $, where $(X_{t})_{t\in[0,T]}$ is the unique strong solution to \eqref{eq:MainSDE}.

\bigskip

	Finally, let us also mention a series of papers, from which our construction method gradually evolved: We refer to the works \cite{MenoukeuMeyerBrandisNilssenProskeZhang_VariationalApproachToTheConstructionOfStrongSolutions}, \cite{MeyerBrandisProske_OnTheExistenceOfStrongSolutionsOfLevyNoiseSDEs}, \cite{MeyerBrandisProske_ConstructionOfStrongSolutionsOfSDEs}, \cite{MohammedNilssenProske_Sobolev} in the case of finite-dimensional Brownian noise. See \cite{FNP.13} in the Hilbert space setting in connection with H\"{o}lder continuous drift vector fields. In the case of SDEs driven by L\'{e}vy processes we mention \cite{HaaPros.14}. Other results can be found in \cite{BNP}, \cite{ABP18} with respect to SDEs driven by fractional Brownian motion and related noise. See also \cite{BOPP.17} in the case of "skew fractional Brownian motion", \cite{BHP.17} with respect to singular delay equations and \cite{Bauer_StrongSolutionsOfMFSDEs} in the case of Brownian motion driven mean-field equations.

	We shall also point to the work of R. Catellier and M. Gubinelli \cite{CG}, who prove existence and \emph{path by path} uniqueness (in the sense of A. M. Davie \cite{Da07}) of strong solutions of fractional Brownian motion driven SDEs with respect to (distributional) drift vector fields belonging to the Besov-H\"{o}lder space $B_{\infty ,\infty }^{\alpha}$, $\alpha \in  \mathbb{R}$. The approach of the authors is based inter alia on the theorem of Arzela-Ascoli and a comparison principle based on an average translation operator. In the distributional case, that is $\alpha <0$, the drift part of the SDE is given by a generalized non-linear Young integral defined via the topology of $B_{\infty ,\infty }^{\alpha}$. See also D. Nualart, Y. Ouknine \cite{nualart2002regularization} in the one-dimensional case.

\bigskip

	The structure of our article is as follows: In Section 2 we introduce the mathematical framework of this paper. Further, in Section 3 we discuss some properties of the process $\mathbb{B}_{\cdot }$ and weak solutions of the SDE \eqref{eq:MainSDE}. Section 4 is devoted to the construction of unique strong solutions to the SDE \eqref{eq:MainSDE}. Finally, in Section 5 examples of singular vector fields for which strong solutions exist are given.

\subsection*{Notation}
	For the sake of readability we assume throughout the paper that $1 \leq T < \infty$ is a finite time horizon. We define $\Hcal$ to be an infinite-dimensional separable real-valued Hilbert space with scalar product $\langle \cdot, \cdot \rangle_{\Hcal}$ and orthonormal basis $\lbrace e_k \rbrace_{k\geq 1}$. Denote by $\Vert \cdot \Vert_{\Hcal}$ the induced norm on $\Hcal$ defined by $\Vert x \Vert_{\Hcal} := \langle x, x \rangle_{\Hcal}^{\frac{1}{2}}$, $x\in\Hcal$. For every $x\in \Hcal$ and $k\geq 1$ we denote by $x^{(k)} := \langle x, e_k \rangle_{\Hcal}$ the projection onto the subspace spanned by $e_k$, $k\geq 1$. Loosely speaking we are referring to the subspace spanned by $e_k$, $k\geq 1$, as the $k$-th dimension. In line with this notation we denote the projection of the SDE \eqref{eq:MainSDE} on the subspace spanned by $e_k$, $k\geq 1$, by $X^{(k)} := \langle X, e_k \rangle_{\Hcal}$. Moreover we can write the SDE \eqref{eq:MainSDE} as an infinite dimensional system of real-valued stochastic differential equations, namely
	\begin{align*}
		X_t^{(k)} = x^{(k)} + \int_0^t b_k(s, X_s) ds + \Bb_t^{(k)}, \quad t\in[0,T], \quad k\geq 1, 
	\end{align*}
	where $b_k$ and $\Bb^{(k)}$ are the projections on the subspace spanned by $e_k$, $k\geq 1$, of $b$ and $\Bb$, respectively. Note here that the function $b_k:[0,T] \times \Hcal \to \Rbb$ has still domain $[0,T]\times \Hcal$. Furthermore, we define the truncation operator $\pi_d$, $d \geq 1$, which maps an element $x\in \Hcal$ onto the first $d$ dimensions, by 
\begin{align}\label{eq:truncationOperator}
	\pi_d x := \sum_{k=1}^d x^{(k)} e_k.
\end{align}
	The truncated space $\pi_d \Hcal$ is denoted by $\Hcal_d$. We define the change of basis operator $\tau: \Hcal \to \ell^2$ by
	\begin{align}\label{eq:changeOfBasis}
		\tau x = \tau \sum_{k\geq 1} x^{(k)} e_k = \sum_{k\geq 1} x^{(k)} \etilde_k,
	\end{align}
	where $\lbrace \etilde_k \rbrace_{k\geq 1}$ is an orthonormal basis of $\ell^2$. It is easily seen that the operator $\tau$ is a bijection and we denote its inverse by $\tau^{-1}:\ell^2 \to \Hcal$.

\underline{Further frequently used notation:}
\begin{itemize}
	\item Let $(\Xcal, \Acal, \mu)$ denote a measurable space and $(\Ycal, \Vert \cdot \Vert_\Ycal)$ a normed space. Then $L^{2}(\Xcal;\Ycal)$ denotes the space of square integrable functions $X$ over $\Xcal$ taking values in $\Ycal$ and is endowed with the norm $$\Vert X \Vert_{L^{2}(\Xcal;\Ycal)}^2 = \int_\Xcal \Vert X(\omega) \Vert_\Ycal^2 \mu(d\omega).$$
	\item The space $L^2(\Omega, \Fcal)$ denotes the space of square integrable random variables on the sample space $\Omega$ measurable with respect to the $\sigma$-algebra $\Fcal$.
	\item We define $\Bb^x := x + \Bb$.
	\item For any vector $u$ we denote its transposed by $u^\top$.
	\item We denote by $\Id$ the identity operator.
	\item The Jacobian of a differentiable function is denoted by $\nabla$.
	\item For any multi-index $\alpha$ of length $d$ and any $d$-dimensional vector $u$ we define $u^\alpha := \prod_{i=1}^d u_i^{\alpha_i}$.
	\item For two mathematical expressions $E_1(\theta),E_2(\theta)$ depending on some parameter $\theta$ we write $E_1(\theta) \lesssim E_2(\theta)$, if there exists a constant $C>0$ not depending on $\theta$ such that $E_1(\theta) \leq C E_2(\theta)$.
	\item Let $A$ be some countable set. Then we denote by $\# A$ its cardinality.
\end{itemize}

\section{Preliminaries}\label{sec:prelim}
\subsection{Shuffles}\label{sec:shuffles}
	Let $m$ and $n$ be two integers. We denote by $\Scal(m,n)$ the set of \emph{shuffle permutations}, i.e. the set of permutations $\sigma: \{1, \dots, m+n\} \rightarrow \{1, \dots, m+n\}$ such that $\sigma(1) < \dots < \sigma(m)$ and $\sigma(m+1) < \dots < \sigma(m+n)$. Equivalently we denote for integers $k$ and $n$ by $\Scal(k; n)$ the set of shuffle permutations of $k$ sets of size $n$, i.e. the set of permutations $\sigma: \lbrace 1, \dots, k\cdot n \rbrace \to \lbrace 1, \dots, k\cdot n \rbrace$ such that $\sigma(m\cdot n + 1) < \dots < \sigma((m+1) \cdot n)$ for all $m = 0, \dots, k-1$.  Furthermore the $n$-dimensional simplex $\Delta^n$ of the interval $(s,t)$ is defined by
\begin{align*}
	\Delta_{s,t}^n := \{(u_1,\dots, u_n)\in [0,T]^n : \, s< u_1 < \cdots < u_n < t\}.
\end{align*}
	Note that the product of two simplices can be written as
\begin{align}\label{eq:ShuffleProduct}
	&\Delta_{s,t}^m \times \Delta_{s,t}^n = \bigcup_{\sigma \in \Scal(m,n)} \{(w_1, \dots, w_{m+n})\in [0,T]^{m+n} : \, w_{\sigma} \in \Delta_{s,t}^{m+n} \} \cup \Nfrak,
\end{align}
	where the set $\Nfrak$ has Lebesgue measure zero and $w_\sigma$ denotes the shuffled vector $(w_{\sigma(1)}, \dots, w_{\sigma(m+n)})$. For the sake of readability we denote throughout the paper the integral over the simplex $\Delta_{s,t}^n$ of the product of integrable functions $f_i:[0,T] \rightarrow \Rbb$, $i=1,\dots, n$, by
	\begin{align*}
		\int_{\Delta_{s,t}^n} \prod_{j=1}^n f_j(u_j) du := \int_s^t \int_{u_1}^t \cdots \int_{u_{n-1}}^t \prod_{j=1}^n f_j(u_j) du_n \cdots du_2 du_1.
	\end{align*}
	Due to \eqref{eq:ShuffleProduct}, we get for integrable functions $f_i:[0,T] \rightarrow \Rbb$, $i=1,\dots,m+n$, that
\begin{align}\label{shuffleIntegral}
	&\int_{\Delta_{s,t}^m} \prod_{j=1}^m f_j(u_j) du \int_{\Delta_{s,t}^n} \prod_{j=m+1}^{m+n} f_j(u_j) du = \sum_{\sigma\in \Scal(m,n)} \int_{\Delta_{s,t}^{m+n}} \prod_{j=1}^{m+n} f_{\sigma(j)} (w_j) dw. 
\end{align}
	For a proof of a more general result we refer the reader to \cite[Lemma 2.1]{BNP}.

\subsection{Fractional Calculus}
In the following we give some basic definitions and properties on fractional calculus. For more insights on the general theory we refer the reader to \cite{oldham1974fractional} and \cite{samko1993fractional}.

Let $a,b\in \Rbb$ with $a<b$, $f,g \in L^p([a,b])$ with $p\geq 1$ and $\alpha>0$. We define the \emph{left-} and \emph{right-sided Riemann-Liouville fractional integrals} by
$$I_{a^+}^\alpha f(x) = \frac{1}{\Gamma (\alpha)} \int_a^x (x-y)^{\alpha-1}f(y)dy,$$
and
$$I_{b^-}^\alpha g(x) = \frac{1}{\Gamma (\alpha)} \int_x^b (y-x)^{\alpha-1}g(y)dy,$$
for almost all $x\in [a,b]$. Here $\Gamma$ denotes the gamma function.

Furthermore, for any given integer $p\geq 1$, let $I_{a^+}^{\alpha} (L^p)$ and $I_{b^-}^{\alpha} (L^p)$ denote the images of $L^p([a,b])$ by the operator $I_{a^+}^\alpha$ and $I_{b^-}^\alpha$, respectively. If $0<\alpha<1$ as well as $f\in I_{a^+}^{\alpha} (L^p)$ and $g\in I_{b^-}^{\alpha} (L^p)$, we define the \emph{left-} and \emph{right-sided Riemann-Liouville fractional derivatives} by
\begin{align}\label{eq:leftSidedDerivative}
	D_{a^+}^{\alpha} f(x)= \frac{1}{\Gamma (1-\alpha)} \frac{d}{d x} \int_a^x \frac{f(y)}{(x-y)^{\alpha}}dy,
\end{align}
and
\begin{align}\label{eq:rightSidedDerivative}
	D_{b^-}^{\alpha} g(x)= \frac{1}{\Gamma (1-\alpha)} \frac{d}{d x} \int_x^b \frac{g(y)}{(y-x)^{\alpha}}dy,
\end{align}
respectively. The left- and right-sided derivatives of $f$ and $g$ defined in \eqref{eq:leftSidedDerivative} and \eqref{eq:rightSidedDerivative} admit moreover the representations
$$D_{a^+}^{\alpha} f(x)= \frac{1}{\Gamma (1-\alpha)} \left(\frac{f(x)}{(x-a)^\alpha}+\alpha\int_a^x \frac{f(x)-f(y)}{(x-y)^{\alpha+1}}dy\right),$$
and
$$D_{b^-}^{\alpha} g(x)= \frac{1}{\Gamma (1-\alpha)} \left(\frac{g(x)}{(b-x)^\alpha}+\alpha\int_x^b \frac{g(x)-g(y)}{(y-x)^{\alpha+1}}dy\right).$$
Last, we get by construction that similar to the fundamental theorem of calculus 
\begin{align}\label{eq:fundamentalThmFC1}
	I_{a^+}^\alpha (D_{a^+}^{\alpha} f) = f,
\end{align}
for all $f\in I_{a^+}^{\alpha} (L^p)$, and
\begin{align}\label{eq:fundamentalThmFC2}
	D_{a^+}^{\alpha}(I_{a^+}^\alpha  g) = g,
\end{align}
for all $g\in L^p([a,b])$. Equivalent results hold for $I_{b^-}^{\alpha}$ and $D_{b^-}^{\alpha}$.

\subsection{Fractional Brownian motion} 
	The one-dimensional \emph{fractional Brownian motion}, in short fBm, $B^H = \left(B_t^H\right)_{ t\in [0,T]}$ with Hurst parameter $H\in (0, \frac{1}{2})$ on a complete probability space $(\Omega,\Fcal,\Pbb)$ is defined as a centered Gaussian process with covariance function
$$R_H(t,s):=\EW{B_t^H B_s^H}=\frac{1}{2}\left(t^{2H} + s^{2H} - |t-s|^{2H} \right).$$
Note that $\EW{\left\vert B_t^H - B_s^H\right\vert^2}= |t-s|^{2H}$ and hence $B^H$ has stationary increments and almost surely H\"{o}lder continuous paths of order $H-\varepsilon$ for all $\varepsilon\in(0,H)$. However, the increments of $B^H$, $H\in (0,\frac{1}{2})$, are not independent and $B^H$ is not a semimartingale, see e.g. \cite[Proposition 5.1.1]{Nualart_MalliavinCalculus}.

Subsequently we give a brief outline of how a fractional Brownian motion can be constructed from a standard Brownian motion. For more details we refer the reader to \cite{Nualart_MalliavinCalculus}.


Recall the following result (see \cite[Proposition 5.1.3]{Nualart_MalliavinCalculus}) which gives the kernel of a fractional Brownian motion and an integral representation of $R_H(t,s)$ in the case of $H<\frac{1}{2}$.

\begin{proposition}\label{prop:kernel}
Let $H< \frac{1}{2}$. The kernel
\begin{small}
\begin{align}\label{eq:kernel}
	K_H(t,s) := c_H \left[\left( \frac{t}{s}\right)^{H- \frac{1}{2}} (t-s)^{H- \frac{1}{2}} + \left( \frac{1}{2}-H\right) s^{\frac{1}{2}-H} \int_s^t u^{H-\frac{3}{2}} (u-s)^{H-\frac{1}{2}} du\right],
\end{align}
\end{small}
where $c_H = \sqrt{\frac{2H}{(1-2H) \beta\left(1-2H , H+\frac{1}{2}\right)}}$ and $\beta$ is the beta function, satisfies
\begin{align}\label{eq:CovarianceFunctionR}
R_H(t,s) = \int_0^{t\wedge s} K_H(t,u)K_H(s,u)du.
\end{align}
\end{proposition}

Subsequently, we denote by $W$ a standard Brownian motion on the complete filtered probability space $(\Omega, \Fcal, \Fbb^W, \Pbb)$, where $\Fbb^W := ( \Fcal_t^W)_{t\in [0,T]}$ is the natural filtration of $W$ augmented by all $\Pbb$-null sets. Using the kernel given in \eqref{eq:kernel} it is well known that the fractional Brownian motion $B^H$ has a representation
\begin{align}\label{eq:fBmAsWienerIntegral}
	B_t^H = \int_0^t K_H(t,s) dW_s, ~H\in\left(0,\frac{1}{2}\right).
\end{align}
Note that due to representation \eqref{eq:fBmAsWienerIntegral} the natural filtration generated by $B^H$ is identical to $\Fbb^W$. Furthermore, equivalent to the case of a standard Brownian motion, it exists a version of Girsanov's theorem for fractional Brownian motion which is due to \cite[Theorem 4.9]{decreusefond1999stochastic}. In the following we state the version given in \cite[Theorem 3.1]{nualart2002regularization}. \par 
But first let us define the isomorphism $K_H$ from $L^2([0,T])$ onto $I_{0+}^{H+\frac{1}{2}}(L^2)$ (see \cite[Theorem 2.1]{decreusefond1999stochastic}) given by
\begin{align}\label{eq:kernelIsomorphism}
	(K_H \varphi)(s) = I_{0^+}^{2H} s^{\frac{1}{2}-H} I_{0^+}^{\frac{1}{2}-H}s^{H-\frac{1}{2}}  \varphi, \quad \varphi \in L^2([0,T]).
\end{align}
From \eqref{eq:kernelIsomorphism} and the properties of the Riemann-Liouville fractional integrals and derivatives \eqref{eq:fundamentalThmFC1} and \eqref{eq:fundamentalThmFC2}, the inverse of $K_H$ is given by
$$(K_H^{-1} \varphi)(s) = s^{\frac{1}{2}-H} D_{0^+}^{\frac{1}{2}-H} s^{H-\frac{1}{2}} D_{0^+}^{2H} \varphi(s), \quad \varphi \in I_{0+}^{H+\frac{1}{2}}(L^2).$$ It can be shown (see \cite{nualart2002regularization}) that if $\varphi$ is absolutely continuous
\begin{align}\label{eq:inverseKH}
(K_H^{-1} \varphi)(s) = s^{H-\frac{1}{2}} I_{0^+}^{\frac{1}{2}-H} s^{\frac{1}{2}-H}\varphi'(s),
\end{align}
where $\varphi'$ denotes the weak derivative of $\varphi$.

\begin{theorem}[Girsanov's theorem for fBm]\label{thm:Girsanov}
Let $u=\left(u_t \right)_{ t\in [0,T]}$ be a process with integrable trajectories and set $\widetilde{B}_t^H = B_t^H + \int_0^t u_s ds, ~ t\in [0,T].$
Assume that
\begin{itemize}
\item[(i)] $\int_0^{\cdot} u_s ds \in I_{0+}^{H+\frac{1}{2}} (L^2 ([0,T])$, $\Pbb$-a.s., and

\item[(ii)] $\EW{\Ecal_T}=1$, where
$$\Ecal_T := \exp\left\{-\int_0^T K_H^{-1}\left( \int_0^{\cdot} u_r dr\right)(s)dW_s - \frac{1}{2} \int_0^T K_H^{-1} \left( \int_0^{\cdot} u_r dr \right)^2(s)ds \right\}.$$
\end{itemize}
Then the shifted process $\widetilde{B}^H$ is an $\Fbb^W$-- fractional Brownian motion with Hurst parameter $H$ under the new probability measure $\widetilde{\Pbb}$ defined by $\frac{d\widetilde{\Pbb}}{d\Pbb}=\Ecal_T$.
\end{theorem}

\begin{remark}\label{rem:GirsanovInfty}
	\Cref{thm:Girsanov} can be extended to the multi- and infinite-dimensional cases, which will be considered in this paper primarily. Indeed, note first that the measure change in Girsanov's theorem acts dimension-wise. In particular, consider the two dimensional shifted process
	\begin{align*}
		X_t^{(1)} &= B_t^{H_1} + \int_0^t u_s^{(1)} ds, \\
		X_t^{(2)} &= B_t^{H_2} + \int_0^t u_s^{(2)} ds, ~t\in[0,T],
	\end{align*}
	where $B^{H_1}$ and $B^{H_2}$ are two fractional Brownian motions with Hurst parameters $H_1$ and $H_2$ generated by the independent standard Brownian motions $W^{(1)}$ and $W^{(2)}$, respectively, and $u^{(1)}$ and $u^{(2)}$ are two shifts fulfilling the conditions of \Cref{thm:Girsanov}. Then the measure change with respect to the stochastic exponential
	\begin{align*}
		\Ecal_T^{(1)} := \exp\left\{-\int_0^T K_{H_1}^{-1}\left( \int_0^{\cdot} u_r^{(1)} dr\right)(s)dW_s^{(1)} - \frac{1}{2} \int_0^T K_{H_1}^{-1} \left( \int_0^{\cdot} u_r^{(1)} dr \right)^2(s)ds \right\}
	\end{align*}
	yields the two dimensional process
	\begin{align*}
		X_t^{(1)} &= \Btilde_t^{H_1}, \\
		X_t^{(2)} &= B_t^{H_2} + \int_0^t u_s^{(2)} ds, ~t\in[0,T].
	\end{align*}
	Here, $\Btilde^{H_1}$ is a fractional Brownian motions with respect to the measure $\widetilde{\Pbb}$ defined by $\frac{d\widetilde{\Pbb}}{d\Pbb} = \Ecal_T^{(1)}$. Note that $B^{H_2}$ is still a fractional Brownian motion under $\widetilde{\Pbb}$, since $W^{(1)}$ and $W^{(2)}$ are independent. Applying Girsanov's theorem again with respect to the stochastic exponential 
	\begin{align*}
		\Ecal_T^{(2)} := \exp\left\{-\int_0^T K_{H_2}^{-1}\left( \int_0^{\cdot} u_r^{(2)} dr\right)(s)dW_s^{(2)} - \frac{1}{2} \int_0^T K_{H_2}^{-1} \left( \int_0^{\cdot} u_r^{(2)} dr \right)^2(s)ds \right\},
	\end{align*}
	yields the two dimensional process
	\begin{align*}
		X_t^{(1)} &= \Btilde_t^{H_1}, \\
		X_t^{(2)} &= \Btilde_t^{H_2}, ~t\in[0,T],
	\end{align*}
	where $\Btilde^{H_1}$ and $\Btilde^{H_2}$ are independent fractional Brownian motions with respect to the measure $\hat{\Pbb}$ defined by
	\begin{align*}
		\frac{d\hat{\Pbb}}{d\Pbb} = \frac{d\hat{\Pbb}}{d\widetilde{\Pbb}} \frac{d\widetilde{\Pbb}}{d\Pbb} = \Ecal_T^{(2)}\Ecal_T^{(1)}.
	\end{align*}
	Repeating iteratively yields the stochastic exponential -- if well-defined --
	\begin{align*}
		\Ecal_T := \prod_{k\geq 1} \Ecal_T^{(k)}
	\end{align*}
	acting on infinite dimensions.
\end{remark}

Finally, we give the property of strong local non-determinism of the fractional Brownian motion $B^H$ with Hurst parameter $H\in(0,\frac{1}{2})$ which was proven in \cite[Lemma 7.1]{pitt1978local}. This property will essentially help us to overcome the limitations of not having independent increments of the underlying noise.

\begin{lemma}\label{lem:localNonDeterminism}
	Let $B^H$ be a fractional Brownian motion with Hurst parameter $H \in (0,\frac{1}{2})$. Then there exists a constant $\Kfrak_H$ dependent merely on $H$ such that for every $t\in [0,T]$ and $0 < r \leq t$
	\begin{align*}
		\Var{B_{t}^{H}\left\vert B_{s}^{H}: \left\vert t-s\right\vert \geq r \right. } \geq \Kfrak_{H} r^{2H}.
	\end{align*}
\end{lemma}

\section{Cylindrical fractional Brownian motion and weak solutions}\label{sec:weakSol}
	We start this section by defining the driving noise $(\Bb_t)_{t\in[0,T]}$ in SDE \eqref{eq:MainSDE}. Let $\lbrace W^{(k)} \rbrace_{k\geq 1}$ be a sequence of independent one-dimensional standard Brownian motions on a joint complete probability space $(\Omega, \Fcal, \Pbb)$. We define the cylindrical Brownian motion $W$ taking values in $\Hcal$ by
	\begin{align*}
		W_t := \sum_{k\geq 1} W_t^{(k)} e_k, \quad t\in[0,T],
	\end{align*}
	and denote by $\Fbb^W := \left( \Fcal_t^W \right)_{t\in[0,T]}$ its natural filtration augmented by the $\Pbb$-null sets. Moreover, we define a sequence of Hurst parameters $H:= \lbrace H_k \rbrace_{k\geq 1} \subset \left( 0, \frac{1}{2} \right)$ with the following properties:
	\begin{enumerate}[(i)]
		\item $\sum_{k\geq 1} H_k < \frac{1}{6}$
		\item $\sup_{k\geq 1} H_k < \frac{1}{12}$ 
	\end{enumerate}
	\vspace{0.2cm}
	Using $H$ we construct the sequence of fractional Brownian motions $\lbrace B^{H_k} \rbrace_{k\geq 1}$ associated to $\lbrace W^{(k)} \rbrace_{k\geq 1}$ by
	\begin{align*}
		B_t^{H_k} := \int_0^t K_{H_k}(t,s)dW_s^{(k)}, \quad t\in[0,T], \quad k\geq 1,
	\end{align*}
	where the kernel $K_{H_k}(\cdot,\cdot)$ is defined as in \eqref{eq:kernel}. Note that the fractional Brownian motions $\lbrace B^{H_k} \rbrace_{k\geq 1}$ are independent by construction. Consequently, we define the cylindrical fractional Brownian motion $B^H$ with associated sequence of Hurst parameters $H$ by
	\begin{align}\label{eq:cylindricalFBm}
		B_t^H:= \sum_{k\geq 1} B_t^{H_k} e_k, \quad t\in[0,T].
	\end{align}
	Nevertheless, the cylindrical fractional Brownian motion $B^H$ is not in the space $L^2(\Omega; \Hcal)$. That is why we consider the operator $Q: \Hcal \to \Hcal$ defined by
	\begin{align*}
		Qx = \sum_{k\geq 1} \lambda_k^2 x^{(k)} e_k,
	\end{align*}
	for a given sequence of non-negative real numbers $\lambda := \lbrace \lambda_k \rbrace_{k\geq 1} \in \ell^2$ such that $\frac{\lambda}{\sqrt{H}} := \left\lbrace \frac{\lambda_k}{\sqrt{H_k}} \right\rbrace_{k\geq 1} \in \ell^1$. In particular, $Q$ is a self-adjoint operator and we have that the \emph{weighted cylindrical fractional Brownian motion}
	\begin{align}\label{eq:Bb}
		\Bb_t := \sqrt{Q}B_t^H = \sum_{k\geq 1} \lambda_k B_t^{H_k} e_k,
	\end{align}
	lies in $L^2(\Omega; \Hcal)$ for every $t\in[0,T]$. Due to the following lemma the stochastic process $(\Bb_t)_{t\in [0,T]}$ is continuous in time.

\begin{lemma}
		The stochastic process $(\Bb_t)_{t\in[0,T]}$ defined in \eqref{eq:Bb} has almost surely continuous sample paths on $[0,T]$.
\end{lemma}
	\begin{proof}
		Note first that due to \cite{borovkov2017bounds}[Theorem 1] for any fractional Brownian motion $B^H$ with Hurst parameter $H \in (0,\frac{1}{2})$ there exists a constant $C>0$ independent of $H$ such that
		\begin{align}\label{eq:boundMaximalfBm}
			\EW{\sup_{t\in[0,T]} \left\vert B_t^{H} \right\vert} \leq \frac{C}{\sqrt{H}}.
		\end{align}
		Using monotone convergence and \eqref{eq:boundMaximalfBm} we have that
		\begin{align*}
			\EW{\sup_{t\in[0,T]} \left\Vert \Bb_t \right\Vert_{\Hcal}} &\leq \EW{\sup_{t\in[0,T]} \sum_{k\geq 1} \vert \lambda_k \vert \left\vert B_t^{H_k} \right\vert} \leq \sum_{k\geq 1} \lambda_k \EW{\sup_{t\in[0,T]} \left\vert B_t^{H_k} \right\vert} \\
			&\leq \sum_{k\geq 1} \lambda_k \frac{C}{\sqrt{H_k}} < \infty.
		\end{align*}
		Thus, $(\sqrt{Q}B_t^H )_{t\in[0,T]}$ is almost surely finite and $\lbrace ( \pi_d \sqrt{Q}B_t^H )_{t\in[0,T]} \rbrace_{d\geq 1}$ is a Cauchy sequence in $L^1( \Omega; \Ccal([0,T];\Hcal) )$ which converges almost surely to $(\sqrt{Q}B_t^H )_{t\in[0,T]}$.
	\end{proof}

Before we come to the next result, let us recall the notion of a weak solution and uniqueness in law.

\begin{definition}\label{def:weakSolution}
	The sextuple $(\Omega, \Fcal, \Fbb, \Pbb, \Bb, X)$ is called a weak solution of stochastic differential equation \eqref{eq:MainSDE}, if
		\begin{enumerate}[(i)]
			\item $(\Omega, \Fcal, \Fbb, \Pbb)$ is a complete filtered probability space, where $\Fbb = \lbrace \Fcal_t \rbrace_{t\in [0,T]}$ satisfies the usual conditions of right-continuity and completeness,
			\item  $\Bb=(\Bb_t)_{t\in [0,T]}$ is a weighted cylindrical fractional $(\Fbb, \Pbb)$-Brownian motion as defined in \eqref{eq:Bb}, and
			\item $X = (X_t)_{t\in [0,T]}$  is a continuous, $\Fbb$-adapted, $\Hcal$-valued process satisfying $\Pbb$-a.s.
			\begin{align*}
				X_t = x+ \int_0^t b(s,X_s) ds + \Bb_t, \quad t\in[0,T].
			\end{align*}
		\end{enumerate}
\end{definition}

	\begin{remark}
		For notational simplicity we refer solely to the process $X$ as a weak solution (or later on as a strong solution) in the case of an unambiguous stochastic basis $(\Omega, \Fcal, \Fbb, \Pbb, B)$.
	\end{remark}

\begin{definition}\label{def:weakUniqueness}
	We say a weak solution $X^1$ with respect to the stochastic basis $(\Omega^1, \Fcal^1,\Fbb^1, \Pbb^1, \Bb^1)$ of the SDE \eqref{eq:MainSDE} is \emph{weakly unique} or \emph{unique in law}, if for any other weak solution $X^2$ of \eqref{eq:MainSDE} on a potential other stochastic basis $(\Omega^2, \Fcal^2,\Fbb^2, \Pbb^2, \Bb^2)$ it holds that 
	\begin{align*}
			\Pbb^1_{X^1} = \Pbb^2_{X^2},
	\end{align*}
		whenever $\Pbb^1_{X_0^1} = \Pbb^2_{X_0^2}$.
\end{definition}

\begin{proposition}\label{prop:weakSolution}
	Let $b:[0,T] \times \Hcal \to \Hcal$ be a measurable and bounded function with $\left\Vert b_k \right\Vert_{\infty} \leq C_k \lambda_k < \infty$ for every $k\geq 1$ where $C := \left\lbrace C_k \right\rbrace_{k\geq 1} \in \ell^1$. Then SDE \eqref{eq:MainSDE} has a weak solution $(X_t)_{t\in[0,T]}$ such that 
	\begin{align*}
		\EW{\sup_{t\in[0,T]} \left\Vert X_t \right\Vert_{\Hcal}^2} < \infty.
	\end{align*}
	Moreover, the solution is unique in law.
\end{proposition}
\begin{proof}
	Let $\lbrace W^{(k)} \rbrace_{k\geq 1}$ be a sequence of independent standard Brownian motions on the filtered probability space $(\Omega, \Fcal, \Fbb, \Qbb)$. Consider the cylindrical fractional Brownian motion $\Bhat^H$ generated by $\lbrace W^{(k)} \rbrace_{k\geq 1}$ as defined in \eqref{eq:cylindricalFBm} with associated sequence of Hurst parameters $H$. We define the stochastic exponential $\Ecal$ by
		\begin{align*}
		\begin{split}
			\Ecal_t := \exp &\left\lbrace \sum_{k\geq 1} \left( \int_0^t K_{H_k}^{-1} \left( \int_0^{\cdot} b_k \left(u,x + \sqrt{Q}\Bhat^H_u \right)\lambda_k^{-1} du \right)(s) dW_s^{(k)} \right. \right. \\
			&\quad \left.\left. - \frac{1}{2} \int_0^t K_{H_k}^{-1} \left( \int_0^{\cdot} b_k\left(u,x + \sqrt{Q}\Bhat^H_u \right)\lambda_k^{-1} du \right)^2 (s) ds \right) \right\rbrace.
		\end{split}
		\end{align*}
	In order to show that the stochastic exponential $\Ecal$ is well-defined we first have to verify that for every $k\geq 1$
	\begin{align*}
		\int_0^{\cdot} b_k \left(u,x + \sqrt{Q}\Bhat^H_u \right)\lambda_k^{-1} du \in I^{H_k+\frac{1}{2}}_{0+}\left(L^2([0,T])\right), ~ \Pbb-\text{a.s.}.
	\end{align*}
		Due to \eqref{eq:inverseKH} this property is fulfilled, if for all $k\geq 1$
		\begin{align*}
			\int_0^T \left( b_k \left(u,x + \sqrt{Q}\Bhat^H_u \right)\lambda_k^{-1}\right)^2 du < \infty,
		\end{align*}
		which holds since $\Vert b_k \Vert_{\infty} \leq C_k \lambda_k$. Furthermore, we can find a constant $C>0$ such that
		\begin{align*}
			\exp&\left\lbrace \frac{1}{2} \sum_{k \geq 1} \int_0^T K_{H_k}^{-1} \left( \int_0^{\cdot} b_k \left(u,x + \sqrt{Q}\Bhat^H_u \right)\lambda_k^{-1} du \right)^2 (s) ds \right\rbrace \\
			&\leq \exp \left\lbrace CT^2 \sum_{k\geq 1} C_k^2 \right\rbrace < \infty.
		\end{align*}
		Hence, by Novikov's criterion $\Ecal_t$ is a martingale, in particular $\EW{\Ecal_t}=1$ for all $t\in[0,T]$. Consequently, under the probability measure $\Pbb$, defined by $\frac{d\Pbb}{d\Qbb} := \Ecal_T$, the process $B_t^{H} := \Bhat_t^H - \int_0^t \sqrt{Q}^{-1} b\left(u,x + \sqrt{Q}\Bhat^H_u \right) du$, $t\in[0,T]$, is a cylindrical fractional Brownian motion due to \Cref{thm:Girsanov} and \Cref{rem:GirsanovInfty}. Therefore, ${(\Omega,\Fcal,\Fbb,\Pbb,\sqrt{Q}B^H, X)}$, where $X_t := x + \sqrt{Q} \Bhat_t^H$, is a weak solution of SDE \eqref{eq:MainSDE}. Since the probability measures $\Qbb \approx \Pbb$ are equivalent, the solution is unique in law.
\end{proof}

\section{Strong Solutions and Malliavin Derivative}\label{sec:strongSol}
	After establishing the existence of a weak solution, we investigate under which conditions SDE \eqref{eq:MainSDE} has a strong solution. Therefore, let us first recall the notion of a strong solution and moreover the notion of pathwise uniqueness.

\begin{definition}\label{def:strongSolution}
	A weak solution $(\Omega, \Fcal, \Fbb^\Bb, \Pbb, \Bb, X^x)$ of the stochastic differential equation \eqref{eq:MainSDE} is called \emph{strong solution}, if $\Fbb^\Bb$ is the filtration generated by the driving noise $\Bb$ and augmented with the $\Pbb$-null sets.
\end{definition}

\begin{definition}\label{def:pathwiseUniqueness}
	We say a weak solution $(\Omega, \Fcal, \Fbb, \Pbb, \Bb, X^1)$ of \eqref{eq:MainSDE} is \emph{pathwise unique}, if for any other weak solution $(\Omega, \Fcal, \Fbb, \Pbb, \Bb, X^2)$ on the same stochastic basis,
	\begin{align*}
		\Pbb\left( \omega\in \Omega: X^1_t(\omega) = X^2_t(\omega) ~\forall t\geq 0 \right) = 1.
	\end{align*}
\end{definition}

	The cause of this paper is to establish the existence of strong solutions of stochastic differential equation \eqref{eq:MainSDE} for \emph{singular} drift coefficients $b$. More precisely, we define the class $\Bfrak([0,T] \times \Hcal; \Hcal)$ of measurable functions $b:[0,T]\times \Hcal \to \Hcal$ for which there exist sequences $C \in \ell^1$ and $D \in \ell^1$ such that for every $k\geq 1$ \vspace{0.2cm}
\begin{align}\label{eq:conditionDrift}
\begin{split}
	\sup_{y\in \Hcal} \sup_{t\in [0,T]} \vert b_k( t, y) \vert &\leq C_k \lambda_k, \text{ and} \\ 
	\sup_{d\geq 1} \int_{\Rbb^d} \sup_{t\in [0,T]} \vert b_k\left( t, \sqrt{Q} \sqrt{\Kcal} \tau^{-1} y \right) \vert dy &\leq D_k \lambda_k,
\end{split}
\end{align}
	where $y=(y_1, \dots, y_d)$ and $\Kcal: \Hcal \to \Hcal$ is the defined by
	\begin{align}\label{eq:KOperator}
		\Kcal x = \sum_{k\geq 1} \Kfrak_{H_k} x^{(k)} e_k, ~x \in \Hcal,
	\end{align}
	for $\lbrace \Kfrak_{H_k} \rbrace_{k\geq 1}$ being the local non-determinism constant of $\lbrace B^{H_k} \rbrace_{k\geq 1}$ as given in \Cref{lem:localNonDeterminism}. \par
	In order to prove the existence of a strong solution for drift coefficients of class $\Bfrak([0,T] \times \Hcal; \Hcal)$ we proceed in the following way:
	\begin{enumerate}[\quad 1)]
		\item We define an approximating double-sequence $\lbrace b^{d,\varepsilon} \rbrace_{d\geq 1, \varepsilon>0}$ for drift coefficients of type \eqref{eq:conditionDrift} which merely act on $d$ dimensions and are sufficiently smooth
		\item For every $d\geq 1$ and $\varepsilon >0$, we prove that the SDE
		\begin{align}\label{eq:approximationSDE}
			X_t^{d,\varepsilon} = x + \int_0^t b^{d,\varepsilon}(s,X_s^{d,\varepsilon}) ds + \Bb_t, ~t\in [0,T],
		\end{align}
		has a unique strong solution which is Malliavin differentiable
		\item We show that the double-sequence of strong solutions $X_t^{d,\varepsilon}$ converges weakly to $\Ebb\left[ X_t \vert \Fcal_t^W \right]$, where $X_t$ is the unique weak solution of SDE \eqref{eq:MainSDE}
		\item Applying a compactness criterion based on Malliavin calculus, we prove that the double-sequence is relatively compact in $L^2(\Omega,\Fcal_t^W)$
		\item Last, we show that $X_t$ is adapted to the filtration $\Fbb^\Bb$ and thus is a strong solution of SDE \eqref{eq:MainSDE}
	\end{enumerate}

\subsection{Approximating double-sequence}
	Recall the truncation operator $\pi_d$, $d\geq 1$, defined in \eqref{eq:truncationOperator} and the change of basis operator $\tau$ defined in \eqref{eq:changeOfBasis}. We define the operator $\pitilde_d: \Hcal \to \Rbb^d$ as $\pitilde_d := \tau \circ \pi_d$. For every $k\geq 1$ let the function $\btilde^d:[0,T] \times \Rbb^d \to \Rbb^d$ be defined by
\begin{align}\label{eq:truncationR}
	\btilde^d (t,z) = \pitilde_d b\left( t, \tau^{-1} z \right).
\end{align}
Let $\varphi_\varepsilon$, $\varepsilon>0$, be a mollifier on $\Rbb^d$ such that for any locally integrable function $f: [0,T] \times \Rbb^d \to \Rbb^d$ and for every $t \in [0,T]$ the convolution $f(t,\cdot) \ast \varphi_\varepsilon$ is smooth and
\begin{align*}
	f(t, \cdot) \ast \varphi_\varepsilon \rightarrow f(t,\cdot), ~\varepsilon \to 0,
\end{align*}
almost everywhere with respect to the Lebesgue measure. Finally, we define for every $d\geq 1$ and $\varepsilon>0$ the double-sequence $b^{d,\varepsilon}:[0,T] \times \Hcal \to \Hcal$ by
\begin{align}\label{eq:approximationDrift}
	b^{d,\varepsilon}(t,y) := \tau^{-1} \left( \btilde^d (t,\pitilde_d y) * \varphi_\varepsilon(\pitilde_d y)\right).
\end{align}
Analogously to \eqref{eq:truncationR}, we define for $t\in [0,T]$ and $z \in \Rbb^d$
\begin{align}\label{eq:approximationDriftR}
	\btilde^{d,\varepsilon}(t, z) := \tau b^{d,\varepsilon}(t, \tau^{-1} z) = \btilde^d (t, z) * \varphi_\varepsilon(z).
\end{align}
Due to the definition of the mollifier $\varphi_\varepsilon$ we have that for every $d \geq 1$
\begin{align}\label{eq:convergenceMollifier}
	b^{d,\varepsilon}(t, \tau^{-1} z) = \tau^{-1} \left( \btilde^d (t, z) * \varphi_\varepsilon(z) \right) \xrightarrow[\varepsilon\to 0]{} \tau^{-1} \btilde^d(t,z) = b^d (t, \tau^{-1} z)
\end{align}
for almost every $(t, z) \in [0,T] \times \Rbb^d$ with respect to the Lebesgue measure. Thus, due to \eqref{eq:convergenceMollifier} and the canonical properties of the truncation operator we have that
\begin{align*}
	b^{d,\varepsilon}(t, y) \xrightarrow[\varepsilon \to 0]{} b^d (t,y) \xrightarrow[d\to \infty]{} b(t,y)
\end{align*}
pointwise in $[0,T] \times \Hcal$, where $b^d := \pi_d b$. Due to the assumptions on $b$ we further get for every $p\geq 2$ using dominated convergence that
\begin{align*}
	\lim_{d\to \infty} \lim_{\varepsilon \to 0} \ETp{b^{d,\varepsilon}(t,\Bb_t^x) - b(t,\Bb_t^x)}{p} = 0.
\end{align*}
Hence, we can speak of an approximating double-sequence $\lbrace b^{d,\varepsilon} \rbrace_{d\geq 1, \varepsilon>0}$ of the drift coefficient $b$. In line with the previously used notation we define
\begin{align*}
	b_k^{d,\varepsilon}(t,y) := \langle b^{d,\varepsilon}(t,y), e_k \rangle_\Hcal &= \langle \btilde^{d,\varepsilon}(t,\tau y), \etilde_k \rangle =: \btilde_k^{d,\varepsilon}(t,\tau y), \\
	b_k^d (t,y) := \langle b^d (t,y), e_k \rangle_\Hcal &= \langle \btilde^d (t,\tau y), \etilde_k \rangle =: \btilde_k^d (t,\tau y).
\end{align*}
Moreover, note that $b^{d,\varepsilon}, b^d \in \Bfrak([0,T] \times \Hcal; \Hcal)$.

\begin{remark}
	Note that we needed to truncate and shift the domain of the function $b$ to $\Rbb^d$ merely in order to apply mollification.
\end{remark}

\subsection{Malliavin differentiable strong solutions for regular drifts}
	In the following proposition we establish the existence of a unique strong solution for a class of drift coefficients which contains the approximating sequence $\lbrace b^{d,\varepsilon} \rbrace_{d\geq 1, \varepsilon>0}$. More specifically, we consider drift coefficients $b \in \Bfrak([0,T]\times \Hcal; \Hcal)$ such that for all $k\geq 1$ and all $t\in [0,T]$
\begin{align*}
	b_k(t,\cdot) \in \Lip_{L_k}(\Hcal;\Rbb),
\end{align*}
where $L \in \ell^2$. We denote the space of such functions by $\Lfrak([0,T]\times \Hcal;\Hcal)$.

\begin{proposition}\label{prop:strongSolutionRegular}
	Let $b \in \Lfrak([0,T]\times \Hcal;\Hcal)$. Then SDE \eqref{eq:MainSDE} has a pathwise unique strong solution.
\end{proposition}
\begin{proof}
	In order to prove the existence of a strong solution we use Picard iteration and proceed similar to the well-known case of finite dimensional SDEs. More precisely, we define inductively the sequence $Y^0 := x + \Bb$ and for all $n\geq 1$
	\begin{align}\label{eq:sequencePicard}
		Y_t^n = x + \int_0^t b\left(s,Y_s^{n-1}\right) ds + \Bb_t, ~ t\in[0,T].
	\end{align}
	We show next that $\lbrace Y^n \rbrace_{n\geq 0}$ is a Cauchy sequence in $L^2([0,T]\times \Omega)$. Indeed, due to monotone convergence we get for every $n \geq 1$ and $t\in[0,T]$
	\begin{align}\label{eq:PicardBound}
		\EHp{ Y_t^{n+1} - Y_t^n } &= \EHp{ \int_0^t b(s,Y_s^n) - b(s,Y_s^{n-1})  ds } \\
		&\leq \int_0^t \left( \sum_{k \geq 1} \EpO{ b_k(s,Y_s^n) - b_k(s,Y_s^{n-1})}{2} \right)^{\frac{1}{2}} ds \notag\\
		&\leq \Vert L \Vert_{\ell^2} \int_0^t \EHp{ Y_s^n - Y_s^{n-1} } ds, \notag
	\end{align}
	and 
	\begin{align*}
		\EHp{ Y_t^1 - Y_t^0 } &= \EHp{ \int_0^t b(s,x+\Bb_s) ds } \leq t \Vert C \lambda \Vert_{\ell^2}.
	\end{align*}
	By induction we obtain for every $n\geq 0$ a constant $A$ depending on $C$, $\lambda$ and $L$ such that
	\begin{align*}
		\EHp{Y_t^{n+1} - Y_t^n} \leq \frac{A^{n+1}}{(n+1)!} t^{n+1}.
	\end{align*}
	Hence, for every $m,n \geq 0$
	\begin{align*}
		\Vert Y^m - Y^n \Vert_{L^2([0,T]\times \Omega; \Hcal)} &\leq \sum_{k=n}^{m-1} \Vert Y^{k+1} - Y^k \Vert_{L^2([0,T]\times \Omega;\Hcal)} \\
		&= \sum_{k=n}^{m-1} \ETp{Y_t^{k+1} - Y_t^k}{2} \\
		&\leq \sum_{k=n}^{m-1} \frac{A^{k+1}}{(k+1)!} T^{k+\frac{3}{2}} =: B(n,m).
	\end{align*}
	Since $B(n,m)$ is bounded by $T^{\frac{1}{2}} e^{AT}$, the series converges and $$B(n,m) \xrightarrow[n,m\to\infty]{} 0.$$ Therefore $\lbrace Y^n \rbrace_{n\geq 0}$ is a Cauchy sequence in $L^2([0,T]\times \Omega;\Hcal)$. Define $$X_t := \lim_{n\to \infty} Y_t^n$$ as the $L^2([0,T]\times \Omega;\Hcal)$ limit of $\lbrace Y^n \rbrace_{n\geq 0}$. Then $X_t$ is $\Fcal^{\Bb}_t$ adapted for all $t\in[0,T]$ since this holds for all $Y_t^n$, $n\geq 0$. We prove that $X_t$ solves SDE \eqref{eq:MainSDE}: \par
	We have for all $n \geq 0$ and $t\in [0,T]$ that $$Y_t^{n+1} = x + \int_0^t b(s,Y_s^n)ds + \Bb_t.$$ Using the Lipschitz continuity of $b$, we get
	\begin{align*}
		\EHp{ \int_0^t b(s,Y_s^n) - b(s,X_s) ds } &\leq \int_0^t \left( \sum_{k\geq 1} \EpO{ b_k(s,Y_s^n) - b_k(s,X_s) }{2} \right)^{\frac{1}{2}} ds \\
		&\leq \Vert L \Vert_{\ell^2} \int_0^t \EHp{ Y_s^n - X_s } ds \xrightarrow[n\to \infty]{} 0.
	\end{align*}
	Hence, $(X_t)_{t\in[0,T]}$ is a strong solution of SDE \eqref{eq:MainSDE}. \par
	In order to show pathwise uniqueness, let $X$ and $Y$ be two strong solutions on the same stochastic basis $(\Omega, \Fcal, \Pbb, \Bb)$ with the same initial condition. Then for all $t\in [0,T]$ we get similar to \eqref{eq:PicardBound} that
	\begin{align*}
		\EHp{ X_t - Y_t } \leq \Vert L \Vert_{\ell^2}\int_0^t \EHp{ X_s - Y_s } ds.
	\end{align*}
	Using Grönwall's inequality yields that $\EH{ X_t - Y_t } = 0$ for all $t\in[0,T]$, and therefore $X_t = Y_t$ $\Pbb$-a.s. for all $t\in[0,T]$. But since $X$ and $Y$ are almost surely continuous we get $$\Pbb\left( \omega\in \Omega: X^1_t(\omega) = X^2_t(\omega) ~\forall t\geq 0 \right) = 1.$$
\end{proof}

Next we investigate under which conditions the unique strong solution is Malliavin differentiable. But let us start with a definition of Malliavin differentiability of a random variable in the space $\Hcal$.

\begin{definition}\label{def:MalliavinDerivative}
	Let $X$ be an $\Hcal$-valued square integrable functional of the cylindrical Brownian motion $(W_t)_{t\in[0,T]}$. We define the operator $D^m$, $m\geq 1$, such that
	\begin{align*}
		D^m X = \sum_{k\geq 1} D^m X^{(k)} e_k,
	\end{align*}
	as the Malliavin derivative in the direction of the $m$-th Brownian motion $W^{(m)}$. Here, $D^m X^{(k)}$, $m,k \geq 1$, is the (standard) Malliavin derivative with respect to the Brownian motion $W^{(m)}$ of the square integrable random variable $X^{(k)}$ taking values in $\Rbb$. We say a random variable $X$ with values in $\Hcal$ is in the space $\Dbb^{1,2}(\Hcal)$ of Malliavin differentiable functions in $L^2(\Omega)$ if and only if
	\begin{align*}
		\ND{X} := \sum_{m\geq 1} \int_0^T \EH{D_s^m X} ds < \infty.
	\end{align*}
	Moreover, a stochastic process $(X_t)_{t\in[0,T]}$ with values in $\Hcal$ is said to be in the space $\Dbb^{1,2}([0,T] \times \Hcal)$ if and only if for every $t\in [0,T]$
	\begin{align*}
		\ND{X_t} := \sum_{m \geq 1} \int_0^T \EH{D_s^m X_t} ds < \infty.
	\end{align*}
\end{definition}

By means of \Cref{def:MalliavinDerivative} we extend the well-known chain rule in Malliavin Calculus, cf. \cite[Proposition 1.2.4]{Nualart_MalliavinCalculus}, to Malliavin differentiable random variables taking values in $\Hcal$. But first we define the class $\Lcal_0(\Hcal)$ of Lipschitz continuous functions on $\Hcal$ with vanishing Lipschitz constants. \par
	We say a function $f:\Hcal \to \Hcal$ is in the space $\Lcal_0(\Hcal)$ if there exist sequences of constants $L, M \in \ell^2$ such that for all $k \geq 1$ and $x,y \in \Hcal$
	\begin{align}\label{eq:vanishingLipschitz}
		\vert \langle f(x) - f(y), e_k \rangle_\Hcal \vert \leq L_k \sum_{i\geq 1} M_i \vert \langle x - y, e_i \rangle_\Hcal \vert.
	\end{align}

\begin{lemma}\label{lem:chainRule}
	Let $f\in \Lcal_0(\Hcal)$ with associated Lipschitz sequences $L, M \in \ell^2$ and $Y \in \Dbb^{1,2}(\Hcal)$. Then, $f(Y) \in \Dbb^{1,2}(\Hcal)$ and there exists a double-sequence $\lbrace G_i^{(k)} \rbrace_{k,i \geq 1}$ of random variables with $G_i^{(k)} \leq L_k \cdot M_i$ $\Pbb$-a.s.  for all $k,i\geq 1$ such that for every $m\geq 1$
	\begin{align}\label{eq:ChainRule}
		D^m f(Y) = \sum_{k\geq 1} \sum_{i\geq 1} G_i^{(k)}  D^m \langle Y, e_i \rangle_{\Hcal}  e_k.
	\end{align}
	Moreover,
	\begin{align*}
		\NDO{f(Y)} \leq \Vert L \Vert_{\ell^2} \cdot \Vert M \Vert_{\ell^2} \cdot \NDO{Y}.
	\end{align*}
\end{lemma}
\begin{proof}
	First, consider the case $f: \Rbb^d \to \Rbb^d$ for some $d\geq 1$, where $Y$ is taking values in $\Rbb^d$. Using the chain rule, see \cite[Proposition 1.2.4]{Nualart_MalliavinCalculus}, and the notion of Malliavin Differentiability in \Cref{def:MalliavinDerivative}, there exists a double-sequence $\lbrace G_i^{(k)} \rbrace_{1 \leq k,i \leq d}$ of random variables with $G_i^{(k)} \leq L_k \cdot M_i$ $\Pbb$-a.s.  for all $1 \leq k,i \leq d$ such that for every $m\geq 1$
	\begin{align}\label{eq:ChainRuleReal}
		D^m f(Y) = \sum_{k=1}^d D^m f_k(Y) \etilde_k = \sum_{k=1}^d \sum_{i=1}^d G_i^{(k)} D^m \langle Y, \etilde_i \rangle \etilde_k.
	\end{align}
	Recall the change of basis operator $\tau: \Hcal \to \ell^2$ defined in \eqref{eq:changeOfBasis}. Let now $f: \Hcal_d \to \Hcal_d$, where $Y$ is taking values in $\Hcal_d$. Define $g: \Rbb^d \to \Rbb^d$ by $g := \tau \circ f \circ \tau^{-1}$. Then $g$ is Lipschitz continuous in the sense of \eqref{eq:vanishingLipschitz} with associated Lipschitz sequences $L, M \in \ell^2$ and due to equality \eqref{eq:ChainRuleReal} we get the identity
	\begin{align*}
		\tau D^m f(Y) &= \tau \sum_{k=1}^d D^m f_k(Y) e_k = \sum_{k=1}^d D^m g_k( \tau Y) \etilde_k \\
		&= \sum_{k=1}^d \sum_{i=1}^d G_i^{(k)} D^m \langle \tau Y, \etilde_i \rangle \etilde_k = \sum_{k=1}^d \sum_{i=1}^d G_i^{(k)}  D^m \langle Y, e_i \rangle_\Hcal \etilde_k \\
		&= \tau \sum_{k=1}^d \sum_{i=1}^d G_i^{(k)}  D^m \langle Y, e_i \rangle_{\Hcal}  e_k.
	\end{align*}
	Thus, equation \eqref{eq:ChainRule} holds for $f: \Hcal_d \to \Hcal_d$. Let finally $f: \Hcal \to \Hcal$, where $Y$ is taking values in $\Hcal$. Recall the truncation operator $\pi_d: \Hcal \to \Hcal_d$ defined in \eqref{eq:truncationOperator}. Since $f$ is Lipschitz continuous, $f(\pi_d Y)$ converges to $f(Y)$ in $L^2(\Omega)$. Furthermore, we have for every $d\geq 1$ that
	\begin{align}\label{eq:uniformBoundMalliavin}
		&\ND{\pi_d f(\pi_d Y)} = \sum_{m\geq 1} \int_0^T \EH{D_s^m(\pi_d f(\pi_d Y))} ds \\
		&\quad = \sum_{m \geq 1} \sum_{k=1}^d \int_0^T \EW{\left\vert \sum_{i=1}^d G_i^{d,(k)}  D_s^m \langle Y, e_i \rangle_{\Hcal}\right\vert^2} ds \notag\\
		&\leq \Vert L \Vert_{\ell^2}^2 \sum_{m \geq 1} \int_0^T \EpO{\sum_{i=1}^d M_i D_s^m \langle Y, e_i \rangle_{\Hcal}}{2} ds \notag \\
		&\quad \leq \Vert L \Vert_{\ell^2}^2 \cdot \Vert M \Vert_{\ell^2}^2 \sum_{m \geq 1} \int_0^T \EH{ D_s^m Y} ds  = \Vert L \Vert_{\ell^2}^2 \cdot \Vert M \Vert_{\ell^2}^2 \cdot \ND{Y} < \infty. \notag
	\end{align}
	Note that the double-sequence $\lbrace G_i^{d,(k)} \rbrace_{i\geq 1, k\geq 1}$ depends on $d\geq 1$. Nevertheless, $\NDO{\pi_d f(\pi_d Y)}$ is uniformly bounded in $d\geq 1$. Thus, due to \cite[Lemma 1.2.3]{Nualart_MalliavinCalculus} and dominated convergence we have $f(Y) \in \Dbb^{1,2}(\Hcal)$ and $D^m(\pi_d f(\pi_d Y))$ converges weakly to $D^m f(Y)$ for every $m\geq 1$. Moreover, the sequence $\lbrace G_i^{d,(k)} \rbrace_{d\geq 1}$ is bounded by $L_k \cdot M_i$ for every $k,i\geq 1$. Hence, for every $k,i \geq 1$ there exists a subsequence $\lbrace G_i^{d_n, (k)} \rbrace_{n\geq 1}$ which converges weakly to some random variable $\Gtilde_i^{(k)}$ which is bounded by $L_k \cdot M_i$. Summarizing we get that in $L^2([0,T] \times \Omega; \Hcal)$
	\begin{align*}
		D^m f(Y) &= \lim_{n \to \infty} \pi_{d_n} D^m f(\pi_{d_n} Y) = \lim_{n \to \infty} \sum_{k=1}^{d_n} \sum_{i=1}^{d_n} G_i^{d_n, (k)}  D^m \langle Y, e_i \rangle_{\Hcal}  e_k \\
		&= \sum_{k \geq 1} \sum_{i \geq 1} \Gtilde_i^{(k)}  D^m \langle Y, e_i \rangle_{\Hcal}  e_k,
	\end{align*}
	where the last equality holds due to \eqref{eq:uniformBoundMalliavin} and dominated convergence.
\end{proof}

\bigskip

Define the class $\Lfrak_0([0,T] \times \Hcal; \Hcal)$ by
\begin{align*}
	&\Lfrak_0([0,T]\times \Hcal; \Hcal) = \\
	&\quad \left\lbrace f \in \Bfrak([0,T]\times \Hcal; \Hcal): f(t,\cdot) \in \Lcal_0(\Hcal) \text{ uniformly in } t\in [0,T] \right\rbrace,
\end{align*}
and note that $f(t,\cdot) \in \Lcal_0(\Hcal)$ uniformly in $t\in [0,T]$ implies $f_k(t,\cdot) \in \Lip_{L_k}(\Hcal; \Rbb)$, $k\geq 1$, uniformly in $t\in [0,T]$ for some sequence $L \in \ell^2$. Thus, $\Lfrak_0([0,T] \times \Hcal; \Hcal) \subset \Lfrak([0,T] \times \Hcal; \Hcal)$.

\begin{proposition}\label{prop:MalliavinRegular}
	Let $b \in \Lfrak_0([0,T]\times \Hcal; \Hcal)$. Then the unique strong solution $(X_t)_{t\in[0,T]}$ of \eqref{eq:MainSDE} is Malliavin differentiable.
\end{proposition}
\begin{proof}
	Recall the Picard iteration defined in \eqref{eq:sequencePicard}
	\begin{align}\label{eq:PicardMalliavinProof}
		Y_t^n = x + \int_0^t b\left(s,Y_s^{n-1}\right) ds + \Bb_t, ~ t\in[0,T], ~ n\geq 1,
	\end{align}
	and $Y^0 = x + \Bb$. We denote the $k$-th dimension of the infinite dimensional system \eqref{eq:PicardMalliavinProof} by $Y^{n,(k)} := \langle Y^n, e_k \rangle_\Hcal$.\par 
	Using the Picard iteration \eqref{eq:PicardMalliavinProof}, we show that for every step $n\geq 0$ the process $Y^n$ is Malliavin differentiable. We prove this using induction. For $n=0$ we have that for all $t\in[0,T]$ using \eqref{eq:CovarianceFunctionR}
	\begin{align*}
		\ND{Y_t^0} &= \sum_{m\geq 1} \int_0^T \EH{D_s^m Y_t^0} ds \notag\\
		&= \sum_{m \geq 1} \int_0^T \EH{\sum_{k\geq 1} \lambda_k D_s^m B_t^{H_k} e_k} ds \notag \\
		&= \sum_{m\geq 1} \int_0^T \EH{ \lambda_m D_s^m B_t^{H_m} e_m} ds \notag\\
		&= \sum_{m\geq 1} \int_0^T \lambda_m^2 K^2_{H_m}(t,s) ds \notag \\
		&= \sum_{m\geq 1} \lambda_m^2 R_{H_m}(t,t) = \sum_{m\geq 1} \lambda_m^2 t^{2H_m} < \infty.
	\end{align*}
	 Now suppose that $\NDO{Y_t^n} < \infty$ for $n\geq 0$. Due to \Cref{lem:chainRule} $b(t,Y_t^n)$ is in $\Dbb^{1,2}(\Hcal)$ and we have for every $t \in [0,T]$ that
	\begin{align*}
		\NDO{b(t,Y_t^n)} \leq \Vert L \Vert_{\ell^2} \cdot \Vert M \Vert_{\ell^2} \cdot \NDO{Y_t^n} < \infty,
	\end{align*}
	for some $L,M \in \ell^2$ independent of $n\geq 0$. Moreover, $\int_0^T b(r,Y_r^n) dr$ is Malliavin differentiable admitting for all $0 \leq s \leq T$ the representation $$D_s^m \left( \int_0^T b(r,Y_r^n) dr \right) = \int_s^T D_s^m b(r, Y_r^n) dr.$$ Thus, we get for $Y^{n+1}$ that
	\begin{align*}
		&\NDO{Y_t^{n+1}} = \NDO{ \left( \int_0^T b(s,Y_s^n) ds + Y_t^0 \right)} \\
		&\quad \leq \int_0^T \NDO{ b(s,Y_s^n) } ds + \NDO{Y_t^0} \\
		&\quad \leq \Vert L \Vert_{\ell^2} \cdot \Vert M \Vert_{\ell^2} \cdot \int_0^T \NDO{Y_s^n} ds + \NDO{Y_t^0} < \infty.
	\end{align*}
	Hence, $Y^{n+1}$ is Malliavin differentiable in the sense of \Cref{def:MalliavinDerivative}. Moreover, we can find a positive constant $A$ depending on $L, M, \lambda$ and $T$ such that
	\begin{align*}
		\NDO{Y_t^n} \leq \sum_{k=0}^n \frac{A^{k+1}}{k!} t^k \leq A\cdot e^{At}.
	\end{align*}
	Consequently, $\ND{Y_t^n}$ is uniformly bounded in $n\geq 0$ and therefore, since $Y^n \rightarrow X$ in $L^2([0,T]\times \Omega)$ and the Malliavin derivative is a closable operator, also $X$ is Malliavin differentiable in the sense of \Cref{def:MalliavinDerivative}.
\end{proof}

Let us finally put the previous results together and show that SDE \eqref{eq:approximationSDE} has a unique Malliavin differentiable strong solution.

\begin{corollary}\label{cor:representationMalliavinRegular}
	Let $b^{d, \varepsilon}: [0,T] \times \Hcal \to \Hcal$ be defined as in \eqref{eq:approximationDrift}. Then, SDE \eqref{eq:approximationSDE} has a unique strong solution $\left( X_t^{d,\varepsilon} \right)_{t\in[0,T]}$ which is Malliavin differentiable. Furthermore, the Malliavin derivative $D_s^m X_t^{d,\varepsilon}$ has for $0 \leq s < t \leq T$ a.s. the representation
	\begin{align}\label{eq:MalliavinDerivativeRepresentationRegular}
		D_s^m X_t^{d,\varepsilon} &= \lambda_m K_{H_m}(t,s) e_m \\
		&\qquad + \lambda_m \sum_{n\geq 1} \int_{\Delta_{s,t}^n} K_{H_m}(u_1,s) \sum_{\eta_0, \dots \eta_{n-1}=1}^d \left( \prod_{j=1}^n \partial_{\eta_j} \btilde_{\eta_{j-1}}^{d,\varepsilon}\left( u_j, \tau X_{u_j}^{d,\varepsilon} \right) \right) e_{\eta_0} du, \notag
	\end{align}
	where $\eta_n = m$ and $\btilde^{d,\varepsilon}: [0,T] \times \Rbb^d \to \Rbb^d$ is defined as in \eqref{eq:approximationDriftR}.
\end{corollary}
\begin{proof}
	If the drift function $b^{d,\varepsilon}$ is in the class $\Lfrak_0([0,T]\times \Hcal, \Hcal)$, then SDE \eqref{eq:approximationSDE} has a unique Malliavin differentiable strong solution by \Cref{prop:strongSolutionRegular} and \Cref{prop:MalliavinRegular}. Thus we merely need to show that $b^{d,\varepsilon}(t,\cdot) \in \Lcal_0(\Hcal)$ uniformly in $t\in[0,T]$. Let $t\in [0,T]$ and $y,z \in \Hcal$. Then, using the triangular inequality and the mean-value theorem we get for all $1 \leq k \leq d$ that
	\begin{align*}
		&\left\vert \left\langle b^{d,\varepsilon}(t,y) - b^{d,\varepsilon}(t,z), e_k \right\rangle_\Hcal \right\vert = \left\vert b_k^{d,\varepsilon}(t,y) - b_k^{d,\varepsilon}(t,z) \right\vert = \left\vert \btilde_k^{d,\varepsilon}(t,\tau^{-1} y) - \btilde_k^{d,\varepsilon}(t,\tau^{-1} z) \right\vert \\
		&\quad \leq \sum_{i=1}^d \left\vert \btilde_k^{d,\varepsilon}\left(t, \sum_{j=1}^{i-1} z_j \etilde_j + \sum_{j=i}^d y_j \etilde_j \right) - \btilde_k^{d,\varepsilon}\left(t,\sum_{j=1}^i z_j \etilde_j + \sum_{j=i+1}^d y_j \etilde_j  \right) \right\vert \\
		&\quad \leq \sum_{i=1}^d \sup_{\xi \in \Rbb^d} \vert \partial_i \btilde_k^{d,\varepsilon}(t,\xi) \vert \vert y_i - z_i \vert = \sum_{i=1}^d \sup_{\xi \in \Rbb^d} \vert \partial_i \btilde_k^{d,\varepsilon}(t,\xi)\vert \vert \langle y - z, e_i \rangle \vert.
	\end{align*}
	Note that we can find sequences $\lbrace L_k \rbrace_{1\leq k \leq d}$ and $\lbrace M_i \rbrace_{1\leq i \leq d}$ such that for all $1\leq k,i \leq d$ we have $\sup_{\xi \in \Rbb^d} \vert \partial_i \btilde_k^{d,\varepsilon}(t,\xi) \vert \leq L_k \cdot M_i$. Hence, $b^{d,\varepsilon} \in \Lfrak_0([0,T]\times \Hcal; \Hcal)$. \par
	It is left to show that representation \eqref{eq:MalliavinDerivativeRepresentationRegular} holds. First note that due to the definition of the Malliavin derivative of a random variable $Y$ with values in $\Hcal$, see \Cref{def:MalliavinDerivative}, we have that $D^m(\tau Y) = \tau D^m Y$, for all $m\geq 1$. Consequently, we get for $0\leq s < t \leq T$ using \Cref{lem:chainRule} that the Malliavin derivative $D_s^m X_t^{d,\varepsilon}$ can be written as
	\begin{align*}
		D_s^m X_t^{d, \varepsilon} &= \tau^{-1} D_s^m \Xtilde_t^{d,\varepsilon} = \int_s^t \nabla \btilde^{d,\varepsilon}\left(u, \Xtilde_u^{d,\varepsilon}\right) D_s^m X_u^{d,\varepsilon} du  + D_s^m \Bb_t.
	\end{align*}
	Iterating this step yields
	\begin{align*}
		D_s^m X_t^{d, \varepsilon} &= \sum_{n\geq 1} \int_{\Delta_{s,t}^n} \left( \prod_{j=1}^n \nabla \btilde^{d,\varepsilon}\left(u_j, \Xtilde_{u_j}^{d,\varepsilon} \right) \right) \lambda_m K_{H_m}(u_1,s) e_m du + \lambda_m K_{H_m}(t,s) e_m.
	\end{align*}
	Further note that
	\begin{align*}
		\nabla \btilde^{d,\varepsilon}\left(u_j, \Xtilde_{u_j}^{d,\varepsilon}\right) &= \nabla \left( \sum_{k=1}^d \btilde_k^{d,\varepsilon} \left( u_j, \Xtilde_{u_j}^{d,\varepsilon} \right) e_k \right) = \sum_{l=1}^d \sum_{k=1}^d \partial_l \btilde_k^{d,\varepsilon}\left( u_j, \Xtilde_{u_j}^{d,\varepsilon} \right) e_k e_l^\top.
	\end{align*}
	Thus, we get for every $n\geq 1$
	\begin{align}\label{eq:representationNabla}
		&\prod_{j=1}^n \nabla \btilde^{d,\varepsilon}\left(u_j, \Xtilde_{u_j}^{d,\varepsilon}\right) = \sum_{l=1}^d \sum_{k=1}^d \left( \sum_{\eta_1, \dots \eta_{n-1}=1}^d \prod_{j=1}^{n} \partial_{\eta_j} \btilde_{\eta_{j-1}}^{d,\varepsilon}\left( u_j, \Xtilde_{u_j}^{d,\varepsilon} \right) \right)e_k e_l^\top,
	\end{align}
	where $\eta_0 = k$ and $\eta_n = l$ and consequently, representation \eqref{eq:MalliavinDerivativeRepresentationRegular} holds.
\end{proof}

\subsection{Weak convergence}
	In this step we show that the sequence of unique strong solutions $\lbrace X^{d,\varepsilon} \rbrace_{d\geq 1, \varepsilon>0}$ of the approximating SDEs \eqref{eq:approximationSDE} converge weakly to the weak solution of \eqref{eq:MainSDE} where $b\in \Bfrak([0,T]\times \Hcal;\Hcal)$.

\begin{lemma}\label{lem:weakConvergence}
	Let $b \in \Bfrak([0,T]\times \Hcal; \Hcal)$. Furthermore, let $(X_t)_{t\in [0,T]}$ be the weak solution of \eqref{eq:MainSDE}. Consider the approximating sequence of strong solutions $\lbrace (X_t^{d,\varepsilon})_{t\in[0,T]} \rbrace_{d\geq 1, \varepsilon>0}$ of SDEs \eqref{eq:approximationSDE}, where $b^{d,\varepsilon}:[0,T] \times \Hcal \to \Hcal$ is defined as in \eqref{eq:approximationDrift}. Then, for every $t\in[0,T]$ and for any bounded continuous function $\phi: \Hcal \to \Rbb$
	\begin{align*}
		\phi(X_t^{d,\varepsilon}) \xrightarrow[d\to \infty, \varepsilon \to 0]{} \EW{\phi(X_t) \big\vert \Fcal_t^W},
	\end{align*}
	weakly in $L^2(\Omega,\Fcal_t^W)$.
\end{lemma}
\begin{proof}
	Using the Wiener transform
	\begin{align*}
		\Wcal(Z)(f) := \EW{Z\Ecal\left(\int_0^T \langle f(s), dW_s \rangle_\Hcal \right)},
	\end{align*}
	of some random variable $Z\in L^2(\Omega,\Fcal_T^W)$ in $f \in L^2([0,T];\Hcal)$, it suffices to show for any arbitrary $f\in L^2([0,T];\Hcal)$ that
	\begin{align*}
		\Wcal(\phi(X_t^{d,\varepsilon}))(f) \xrightarrow[d\to \infty, \varepsilon \to 0]{} \Wcal\left(\EW{\phi(X_t) \big\vert \Fcal_t^W}\right)(f).
	\end{align*}
	So, let $f \in L^2([0,T];\Hcal)$ be arbitrary, then by using Girsanov's theorem we get
	\begin{align*}
		&\left\vert \Wcal(\phi(X_t^{d,\varepsilon}))(f) - \Wcal\left(\EW{\phi(X_t) \big\vert \Fcal_t^W}\right)(f) \right\vert \\
		&\quad = \left\vert \EW{\phi(\Bb_t^x) \Ecal\left( \int_0^T \left\langle f(s) + \left( \sum_{k=1}^d K_{H_k}^{-1} \left(\int_0^\cdot b_k^{d,\varepsilon}(u,\Bb_u^x) \lambda_k^{-1} du \right)(s) e_k \right), dW_s\right\rangle_\Hcal \right)} \right.\\
		&\qquad - \left. \EW{\phi(\Bb_t^x) \Ecal\left( \int_0^T \left\langle f(s) + \left( \sum_{k\geq 1} K_{H_k}^{-1} \left(\int_0^\cdot b_k(u,\Bb_u^x) \lambda_k^{-1} du \right)(s) e_k \right), dW_s \right\rangle_\Hcal \right)} \right\vert \\
		&\quad \lesssim \Ebb \left[ \left\vert \Ecal\left( \int_0^T \left\langle f(s) + \left( \sum_{k=1}^d K_{H_k}^{-1} \left(\int_0^\cdot b_k^{d,\varepsilon}(u,\Bb_u^x) \lambda_k^{-1} du \right)(s) e_k \right), dW_s \right\rangle_\Hcal \right) \right. \right.\\
		&\qquad - \left. \left. \Ecal\left( \int_0^T \left\langle f(s) + \left( \sum_{k\geq 1} K_{H_k}^{-1} \left(\int_0^\cdot b_k(u,\Bb_u^x) \lambda_k^{-1} du \right)(s) e_k \right), dW_s \right\rangle_\Hcal \right) \right\vert \right].
	\end{align*}
	Using the inequality 
	\begin{align*}
		\left\vert e^x - e^y \right\vert \leq \left\vert x - y \right\vert \left( e^x + e^y \right) \quad \forall x,y \in \Rbb,
	\end{align*}
	we get
	\begin{align*}
		&\left\vert \Wcal(\phi(X_t^{d,\varepsilon}))(f) - \Wcal\left(\EW{\phi(X_t) \big\vert \Fcal_t^W}\right)(f) \right\vert \\
		&\quad \lesssim \Ebb \left[ \left\vert \int_0^T \left\langle f(s) + \left( \sum_{k=1}^d K_{H_k}^{-1} \left(\int_0^\cdot b_k^{d,\varepsilon}(u,\Bb_u^x) \lambda_k^{-1} du \right)(s) e_k \right), dW_s \right\rangle_\Hcal \right. \right.\\
		&\qquad - \left. \left. \int_0^T \left\langle f(s) + \left( \sum_{k\geq 1} K_{H_k}^{-1} \left(\int_0^\cdot b_k(u,\Bb_u^x) \lambda_k^{-1} du \right)(s) e_k \right), dW_s \right\rangle_\Hcal \right\vert \right] \\
		&\qquad + \Ebb \left[ \left\vert \int_0^T \left\langle \left( f(s) + \left( \sum_{k=1}^d K_{H_k}^{-1} \left(\int_0^\cdot b_k^{d,\varepsilon}(u,\Bb_u^x) \lambda_k^{-1} du \right)(s) e_k \right)\right)^2, ds \right\rangle_\Hcal \right. \right.\\
		&\qquad - \left. \left. \int_0^T \left\langle \left( f(s) + \left( \sum_{k\geq 1} K_{H_k}^{-1} \left(\int_0^\cdot b_k(u,\Bb_u^x) \lambda_k^{-1} du \right)(s) e_k \right)\right)^2, ds \right\rangle_\Hcal \right\vert \right] \\
		&\quad \leq \Ebb \left[ \left\vert \sum_{k=1}^d \int_0^T K_{H_k}^{-1} \left(\int_0^\cdot b_k^{d,\varepsilon}(u,\Bb_u^x) \lambda_k^{-1} - b_k(u,\Bb_u^x) \lambda_k^{-1} du \right)(s) dW_s^{(k)} \right. \right.\\
		&\qquad - \left. \left. \sum_{k \geq d+1} \int_0^T K_{H_k}^{-1} \left(\int_0^\cdot b_k(u,\Bb_u^x) \lambda_k^{-1} du \right)(s) dW_s^{(k)} \right\vert \right] \\
		&\qquad + A_{d,\varepsilon}(f),
	\end{align*}
	where
	\begin{align*}
		A_{d,\varepsilon}(f) &:= \Ebb \left[ \left\vert \int_0^T \left\langle \left( f(s) + \left( \sum_{k=1}^d K_{H_k}^{-1} \left(\int_0^\cdot b_k^{d,\varepsilon}(u,\Bb_u^x) \lambda_k^{-1} du \right)(s) e_k \right)\right)^2, ds \right\rangle_\Hcal \right. \right.\\
		&\quad - \left. \left. \int_0^T \left\langle \left( f(s) + \left( \sum_{k\geq 1} K_{H_k}^{-1} \left(\int_0^\cdot b_k(u,\Bb_u^x) \lambda_k^{-1} du \right)(s) e_k \right)\right)^2, ds \right\rangle_\Hcal \right\vert \right].
	\end{align*}
	For every $k\geq 1$, we get with representation \eqref{eq:inverseKH} that
	\begin{align*}
		&\Kcal_{H_k}^{-1}(d,\varepsilon,s)  := K_{H_k}^{-1} \left(\int_0^\cdot b_k^{d,\varepsilon}(u,\Bb_u^x) \lambda_k^{-1} - b_k(u,\Bb_u^x) \lambda_k^{-1} du \right)(s) \\
		&\quad = s^{H_k-\frac{1}{2}} I_{0+}^{\frac{1}{2}-H_k} s^{\frac{1}{2}-H_k} \left( b_k^{d,\varepsilon}(s,\Bb_s^x) - b_k(s,\Bb_s^x) \right) \lambda_k^{-1} \\
		&\quad = \frac{\lambda_k^{-1}}{\Gamma\left( \frac{1}{2}-H_k \right)} \int_0^s \left(\frac{u}{s} \right)^{\frac{1}{2}-H_k} (s-u)^{-\frac{1}{2}-H_k} \left( b_k^{d,\varepsilon}(u,\Bb_u^x) - b_k(u,\Bb_u^x) \right) du,
	\end{align*}
	which is bounded by
	\begin{align*}
		\left\vert \Kcal_{H_k}^{-1}(d,\varepsilon,s)  \right\vert &\leq 2 \frac{C_k}{\Gamma\left( \frac{1}{2}-H_k \right)} \int_0^s \left(\frac{u}{s} \right)^{\frac{1}{2}-H_k} (s-u)^{-\frac{1}{2}-H_k} du \\
		&= 2 \frac{C_k}{\Gamma\left( \frac{1}{2}-H_k \right)} s^{\frac{1}{2}-H_k} \beta\left( \frac{3}{2} - H_k, \frac{1}{2}-H_k \right) \lesssim C_k.
	\end{align*}
	Consequently, we get for every $d\geq 1$ using the Burkholder-Davis-Gundy inequality that
	\begin{align*}
		\Eabs{\sum_{k=1}^d \int_0^T  \Kcal_{H_k}^{-1}(d,\varepsilon,s) dW_s^{(k)}} &\leq \sum_{k=1}^d \Ebb \left[ \int_0^T  \left\vert \Kcal_{H_k}^{-1}(d,\varepsilon,s) \right\vert^2 ds \right]^\frac{1}{2} \lesssim \sum_{k\geq 1} C_k < \infty.
	\end{align*}
	Hence, by dominated convergence
	\begin{align*}
		\lim_{d\to \infty} \lim_{\varepsilon \to 0} \Eabs{\sum_{k=1}^d \int_0^T  \Kcal_{H_k}^{-1}(d,\varepsilon,s) dW_s^{(k)}} = 0.
	\end{align*}
	Equivalently, we have
	\begin{align*}
		&\Eabs{\int_0^T \sum_{k\geq d+1} K_{H_k}^{-1} \left(\int_0^\cdot b_k(u,\Bb_u^x) \lambda_k^{-1} du \right)(s) dW_s^{(k)}} \lesssim \sum_{k\geq 1} C_k < \infty.
	\end{align*}
	Thus, again by dominated convergence
	\begin{align*}
		\lim_{d\to \infty} \lim_{\varepsilon \to 0} \Eabs{\int_0^T \sum_{k\geq d+1} K_{H_k}^{-1} \left(\int_0^\cdot b_k(u,\Bb_u^x) \lambda_k^{-1} du \right)(s) dW_s^{(k)}} = 0.
	\end{align*}
	Similarly, one can show that $A_{d,\varepsilon}(f)$ vanishes for every $f\in L^2([0,T];\Hcal)$ as $\varepsilon \to 0$ and $d\to \infty$. Consequently, $\phi(X_t^{d,\varepsilon}) \xrightarrow[d\to \infty, \varepsilon \to 0]{} \EW{\phi(X_t) \big\vert \Fcal_t^W}$ weakly in $L^2(\Omega,\Fcal_t^W)$.
\end{proof}

\subsection{Application of the compactness criterion}

\begin{theorem}\label{thm:compactnessApplication}
	The double-sequence $\left\lbrace X_t^{d,\varepsilon} \right\rbrace_{d\geq 1, \varepsilon>0}$ of strong solutions of SDE \eqref{eq:approximationSDE} is relatively compact in $L^2(\Omega,\Fcal_t^W)$.
\end{theorem}
\begin{proof}
	We are aiming at applying the compactness criterion given in \Cref{thm:compactCriterion}. Therefore, let $0<\alpha_m<\beta_m<\frac{1}{2}$ and $\gamma_m>0$ for all $m\geq 1$ and define the sequence $\mu_{s,m}=2^{-i\alpha_m}\gamma_m$, if $s=2^{i}+j$, $i\geq 0$, $0\leq j\leq 2^{i},$ $m\geq 1$ where  $\mu_{s,m}\longrightarrow 0$ for $s,m \longrightarrow \infty $. We have to check that there exists a uniform constant $C$ such that for all $\lbrace X_t^{d,\varepsilon} \rbrace_{d\geq 1, \varepsilon>0}$
\begin{equation}\label{eq:ConvergenceThmL2}
	\left\Vert X_t^{d,\varepsilon}  \right\Vert_{L^{2}(\Omega;\Hcal)}\leq C,
\end{equation}
\begin{align*}
	\sum_{m\geq 1}\gamma_m^{-2} \left\Vert D^m X_t^{d,\varepsilon}  \right\Vert_{L^{2}(\Omega;L^{2}([0,T];\Hcal))}^{2} \leq C,
\end{align*}
and
\begin{equation}\label{eq:ConvergenceThmMalliavinDifference}
	\sum_{m\geq 1} \frac{1}{(1-2^{-2(\beta_m-\alpha_m)})\gamma_m^{2}} \int_{0}^T \int_{0}^T \frac{\left\Vert D_s^m X_t^{d,\varepsilon} - D_u^m X_t^{d,\varepsilon}  \right\Vert_{L^{2}(\Omega;\Hcal)}^{2}}{\left\vert s-u\right\vert^{1+2\beta_m}} ds du\leq C.
\end{equation}
	Note first that \eqref{eq:ConvergenceThmL2} is fulfilled due to the uniform boundedness of $\lbrace b^{d,\varepsilon} \rbrace_{d\geq 1, \varepsilon>0}$ and the definition of the process $\left( \Bb_t \right)_{t\in[0,T]}$, see \eqref{eq:Bb}. \par
	Next we show uniform boundedness of \eqref{eq:ConvergenceThmMalliavinDifference}. Note first that under the assumption $u\leq s$ we have
	\begin{align*}
		D_s^m X_t^{d,\varepsilon} &- D_u^m X_t^{d,\varepsilon} = \lambda_m \left( K_{H_m}(t,s) - K_{H_m}(t,u) \right) e_m \\
		&\quad + \int_s^t \nabla \btilde^{d,\varepsilon}(v, \Xtilde_v^{d,\varepsilon}) D_s^m X_v^{d,\varepsilon} dv - \int_u^t \nabla \btilde^{d,\varepsilon}(v, \Xtilde_v^{d,\varepsilon}) D_u^m X_v^{d,\varepsilon} dv \\
		&= \lambda_m \left( K_{H_m}(t,s) - K_{H_m}(t,u) \right) e_m - \int_u^s \nabla \btilde^{d,\varepsilon}(v, \Xtilde_v^{d,\varepsilon}) D_u^m X_v^{d,\varepsilon} dv \\
		&\quad + \int_s^t \nabla \btilde^{d,\varepsilon}(v, \Xtilde_v^{d,\varepsilon}) \left(D_s^m X_v^{d,\varepsilon} - D_u^m X_v^{d,\varepsilon}\right) dv \\
		&= \lambda_m \left( K_{H_m}(t,s) - K_{H_m}(t,u) \right) e_m - D_u^m X_s^{d,\varepsilon} + \lambda_m K_{H_m}(s,u) e_m \\
		&\quad + \int_s^t \nabla \btilde^{d,\varepsilon}(v, \Xtilde_v^{d,\varepsilon}) \left(D_s^m X_v^{d,\varepsilon} - D_u^m X_v^{d,\varepsilon}\right) dv.
	\end{align*}
	Using iteration we obtain the representation
	\begin{align*}
		D_s^m X_t^{d,\varepsilon} &- D_u^m X_t^{d,\varepsilon} = \lambda_m \left( K_{H_m}(t,s) - K_{H_m}(t,u) \right) e_m \\
		&\quad + \lambda_m \sum_{n\geq 1} \int_{\Delta_{s,t}^n} \prod_{j=1}^n \nabla \btilde^{d,\varepsilon}(v_j, \Xtilde_{v_j}^{d,\varepsilon}) \left( K_{H_m}(v_1,s) - K_{H_m}(v_1,u) \right) e_m dv \\
		&\quad + \left( \Id + \sum_{n\geq 1} \int_{\Delta_{s,t}^n} \prod_{j=1}^n \nabla \btilde^{d,\varepsilon}(v_j, \Xtilde_{v_j}^{d,\varepsilon}) dv \right) \left(\lambda_m K_{H_m}(s,u) e_m - D_u^m X_s^{d,\varepsilon}\right),
	\end{align*}
	where by \Cref{cor:representationMalliavinRegular}
	\begin{align*}
		&\left(\lambda_m K_{H_m}(s,u) e_m - D_u^m X_s^{d,\varepsilon}\right) = \\
		&\qquad - \lambda_m \sum_{n\geq 1} \int_{\Delta_{u,s}^n}  K_{H_m}(v_1,u) \sum_{\eta_0, \dots, \eta_{n-1}=1}^d \prod_{j=1}^n \partial_{\eta_j} \btilde_{\eta_{j-1}}^{d,\varepsilon}(v_j, \Xtilde_{v_j}^{d,\varepsilon}) e_{\eta_0} dv.
	\end{align*}
	Consequently, we get due to \eqref{eq:representationNabla} that
	\begin{align*}
		D_s^m X_t^{d,\varepsilon} &- D_u^m X_t^{d,\varepsilon} = \lambda_m \left( \Ifrak_1 + \Ifrak_2 + \Ifrak_3 \right),
	\end{align*}
	where
	\begin{align*}
		\Ifrak_1 &:= \left( K_{H_m}(t,s) - K_{H_m}(t,u) \right) e_m, \\
		\Ifrak_2 &:= \sum_{n\geq 1} \int_{\Delta_{s,t}^n} \left( K_{H_m}(v_1,s) - K_{H_m}(v_1,u) \right) \sum_{\eta_0, \dots, \eta_{n-1}=1}^d \prod_{j=1}^n \partial_{\eta_j} \btilde_{\eta_{j-1}}^{d,\varepsilon}(v_j, \Xtilde_{v_j}^{d,\varepsilon}) e_{\eta_0} dv, \\
		\Ifrak_3 &:= - \left( \Id + \sum_{n\geq 1} \int_{\Delta_{s,t}^n} \sum_{\eta_0, \dots, \eta_{n-1}=1}^d \prod_{j=1}^n \partial_{\eta_j} \btilde_{\eta_{j-1}}^{d,\varepsilon}(v_j, \Xtilde_{v_j}^{d,\varepsilon})dv \right) \\
		&\qquad \times \sum_{n\geq 1} \int_{\Delta_{u,s}^n}  K_{H_m}(v_1,u) \sum_{\eta_0, \dots, \eta_{n-1}=1}^d \prod_{j=1}^n \partial_{\eta_j} \btilde_{\eta_{j-1}}^{d,\varepsilon}(v_j, \Xtilde_{v_j}^{d,\varepsilon}) e_{\eta_0} dv.
	\end{align*}	
	In the following we consider each $\Ifrak_i$, $i=1,2,3$, separately starting with the first. Due to \Cref{lem:doubleint} there exists $\beta_1 \in \left( 0 , \frac{1}{2} \right)$ and a constant $K_1>0$ such that
	\begin{align*}
		\int_{0}^{t} \int_{0}^{t} \frac{\left\Vert \Ifrak_1 \right\Vert_{L^{2}(\Omega;\Hcal)}^{2}}{\left\vert s-u\right\vert^{1+2\beta_1}} ds du = \int_{0}^{t} \int_{0}^{t} \frac{\vert K_{H_m}(t,s) - K_{H_m}(t,u)\vert}{\left\vert s-u\right\vert^{1+2\beta_1}} ds du \leq K_1 < \infty.
	\end{align*}
	Consider now $\Ifrak_2$. Define the density $\Ecal_t^d$ by
	\begin{align*}
		\Ecal_t^d := \exp &\left\lbrace \sum_{k=1}^d \left( \int_0^t K_{H_k}^{-1} \left( \int_0^{\cdot} b_k^{d,\varepsilon} \left(u, X_u^{d,\varepsilon} \right)\lambda_k^{-1} du \right)(s) dW_s^{(k)} \right. \right. \\
			&\quad \left.\left. - \frac{1}{2} \int_0^t K_{H_k}^{-1} \left( \int_0^{\cdot} b_k^{d,\varepsilon} \left(u, X_u^{d,\varepsilon} \right)\lambda_k^{-1} du \right)^2 (s) ds \right) \right\rbrace.
	\end{align*}		
	Then applying Girsanov's \cref{thm:Girsanov}, monotone convergence and noting that $\sup_{d\geq 1} \sup_{t\in[0,T]]} \Vert \Ecal_t^d \Vert_{L^4(\Omega)}< \infty$ yields
	\begin{small}
	\begin{align*}
		&\left\Vert \Ifrak_2 \right\Vert_{L^{2}(\Omega;\Hcal)}^{2} \\
		&~ \leq \sum_{n\geq 1} \sum_{\eta_0, \dots \eta_{n-1}=1}^d \left\Vert \Ecal_t^d \int_{\Delta_{s,t}^n} \left( K_{H_m}(v_1,s) - K_{H_m}(v_1,u) \right)  \prod_{j=1}^n \partial_{\eta_j} \btilde_{\eta_{j-1}}^{d,\varepsilon} \left( v_j, \tau \Bb_{v_j}^x \right) dv \right\Vert_{L^2(\Omega)}^2 \\
		&~ \lesssim \sum_{n\geq 1} \sum_{\eta_0, \dots \eta_{n-1}=1}^d \left\Vert \int_{\Delta_{s,t}^n} \left( K_{H_m}(v_1,s) - K_{H_m}(v_1,u) \right) \prod_{j=1}^n \partial_{\eta_j} \btilde_{\eta_{j-1}}^{d,\varepsilon} \left( v_j, \tau \Bb_{v_j}^x \right) dv \right\Vert_{L^4(\Omega)}^2.
	\end{align*}
	\end{small}
	Using equation \eqref{shuffleIntegral} yields that
	\begin{align*}
		\left\vert \Afrak_2 \right\vert^2 := \left\vert \int_{\Delta_{s,t}^n} \left( K_{H_m}(v_1,s) - K_{H_m}(v_1,u) \right) \prod_{j=1}^n \partial_{\eta_j} \btilde_{\eta_{j-1}}^{d,\varepsilon} \left( v_j, \tau \Bb_{v_j}^x \right) dv \right\vert^2
	\end{align*}
	can be written as
	\begin{footnotesize}
	\begin{align*}
		\left\vert \Afrak_2 \right\vert^2 = \sum_{\sigma \in \Scal(n,n)} \int_{\Delta_{s,t}^{2n}} \left( \prod_{j=1}^{2n} g_{[\sigma(j)]}\left(v_j, \tau \Bb_{v_j}^x \right) \right) \left( \prod_{i=0}^1 \left( K_{H_m}(v_{(in+1)},s) - K_{H_m}(v_{(in+1)},u) \right) \right) dv
	\end{align*}
	\end{footnotesize}
		where for $j=1,\dots, n$
	\begin{align*}
		g_j \left(\cdot, \tau \Bb_{\cdot}^x \right) = \partial_{\eta_j} \btilde_{\eta_{j-1}}^{d,\varepsilon} \left( \cdot , \tau \Bb_{\cdot}^x \right)
	\end{align*}
	Repeating the application of \eqref{shuffleIntegral} yields
	\begin{footnotesize}
	\begin{align*}
		\left\vert \Afrak_2 \right\vert^4 = \sum_{\sigma \in \Scal(4; n)} \int_{\Delta_{s,t}^{4n}} \left( \prod_{j=1}^{4n} g_{[\sigma(j)]}\left(v_j, \tau \Bb_{v_j}^x \right) \right) \left( \prod_{i=0}^3 \left( K_{H_m}(v_{(in+1)},s) - K_{H_m}(v_{(in+1)},u) \right) \right) dv.
	\end{align*}
	\end{footnotesize}
	Defining $f_j^{d,\varepsilon}(t,\ytilde) := \btilde_{\eta_{j-1}}^{d,\varepsilon} \left( t , \sqrt{Q} \sqrt{\Kcal} \ytilde \right)$ permits the use of \Cref{prop:mainestimate1} with $\sum_{j=1}^{4n} \varepsilon_{j} = 4$, $\vert \alpha_j \vert = 1$ for all $1 \leq j \leq 4n$ and thus $\vert \alpha \vert = 4n$. Consequently, we get using the assumptions on $H$ and $b$ that
	\begin{align*}
		\EW{\vert \Afrak_2 \vert^4} & = \left\Vert \int_{\Delta_{s,t}^n} \left( K_{H_m}(v_1,s) - K_{H_m}(v_1,u) \right) \prod_{j=1}^n \partial_{\eta_j} b_{\eta_{j-1}}^{d,\varepsilon}\left( v_j, \tau \Bb_{v_j}^x \right) dv \right\Vert_{L^4(\Omega)}^4 \\
		&\leq \# \Scal(4;n) \frac{K_{d,H}^{4n} \cdot T^{\frac{\vert \alpha \vert}{12}}}{\sqrt{2\pi}^{4dn}} \left( C_{H_m,T} \left( \frac{s - u}{s u}\right)^{\gamma_m} s^{(H_m-\frac{1}{2} - \gamma_m)}\right)^{\sum_{j=1}^{4n} \varepsilon_j} \\
		&\quad \times \prod_{j=1}^{n} \left\Vert \btilde_{\eta_{j-1}}^{d,\varepsilon}(\cdot , \sqrt{Q} \sqrt{\Kcal}z_{j})\right\Vert _{L^{1}(\mathbb{R}^{d};L^{\infty }([0,T]))}^4 \\
		&\quad \times \frac{\left(\prod_{k=1}^{d} \left(2\left\vert \alpha ^{(k)}\right\vert\right)!\right)^{\frac{1}{4}}(t-s )^{-\sum_{k=1}^{d} H_k \left(4n+2\left\vert \alpha^{(k)}\right\vert \right)+ \left(H_m-\frac{1}{2}-\gamma_m \right)\sum_{j=1}^{4n} \varepsilon _{j}+4n}}{\Gamma(8n -\sum_{k=1}^{d}H_k(8n+4\left\vert \alpha^{(k)}\right\vert )+2(H_m - \frac{1}{2}-\gamma_m)\sum_{j=1}^{4n} \varepsilon _{j})^{\frac{1}{2}}} \\
		&\leq 2^{8n} \frac{K_{d,H}^{4n} \cdot T^{\frac{n}{3}}}{\sqrt{2\pi}^{4dn}} C_{H_m,T}^4 \prod_{j=1}^n D_{\eta_{j-1}}^4 \lambda_{\eta_{j-1}}^4 \\
	&\quad \times \left( \frac{s - u}{su}\right)^{4 \gamma_m} s^{4 \left(H_m-\frac{1}{2} - \gamma_m \right)} (t - s )^{4\left(H_m-\frac{1}{2}-\gamma_m \right)} T^{4n} S_{n},
	\end{align*}
	where
	\begin{align*}
	S_{n} = \sup_{\eta} \frac{\left(\prod_{k=1}^{d} \left(2\left\vert \alpha^{(k)}\right\vert\right)!\right)^{\frac{1}{4}}}{\Gamma\left(8n -\sum_{k=1}^{d} H_k \left(8n+4\left\vert \alpha^{(k)}\right\vert \right) + 8\left(H_m - \frac{1}{2}-\gamma_m \right) \right)^{\frac{1}{2}}}.
	\end{align*}
	For $n\geq 1$ we have due to the assumptions on $H$ that
	\begin{align*}
		A_n &:= 8n -\sum_{k=1}^{d} H_k \left(8n+4\left\vert \alpha^{(k)}\right\vert \right) + 8\left(H_m - \frac{1}{2}-\gamma_m \right) \\
		&\quad \geq 8n - 8n \Vert H \Vert_{\ell^1} - 16n \sup_{k\geq 1} \vert H_k \vert - 4 > \frac{16}{3} n - 4 > 0.
	\end{align*}
	Thus, we have for $n$ sufficiently large that
	\begin{align*}
		\Gamma(A_n) \geq \Gamma\left(\frac{16}{3}n-4\right) \sim \Gamma\left(\frac{16}{3}n + 1 \right) \left(\frac{16}{3}n\right)^{-4},
	\end{align*}
	and therefore by the approximations in \Cref{rem:boundFaculty}
	\begin{align*}
		S_{n} &\leq \frac{\left(\prod_{k=1}^{d} \left(2\left\vert \alpha^{(k)}\right\vert\right)!\right)^{\frac{1}{4}}}{\Gamma\left(8n -\sum_{k=1}^{d} H_k \left(8n+4\left\vert \alpha^{(k)}\right\vert \right) + 8\left(H_m - \frac{1}{2}-\gamma_m \right) \right)^{\frac{1}{2}}} \\
		&\sim \frac{\left(2\pi \right)^{\frac{d}{8}} e^{\frac{n}{2}} ( (10n)!)^{\frac{1}{4}} \left(\frac{16}{3}n\right)^{2}}{\left( 20 \pi n \right)^{\frac{1}{8}}\Gamma\left(\frac{16}{3}n  + 1 \right)^{\frac{1}{2}}} \\
		&\leq C^n \frac{\left(2\pi \right)^{\frac{d}{8}} ( (10n)!)^{\frac{1}{4}} n^{\frac{15}{8}}}{\Gamma\left(\frac{16}{3}n + 1\right)^{\frac{1}{2}}},
	\end{align*}
	where $C>0$ is a constant which may in the following vary from line to line. Using Stirling's formula we have moreover that
	\begin{align*}
		\frac{(10n)!}{\Gamma\left( \frac{16}{3} n + 1\right)^2} &\leq \frac{e^{\frac{1}{120n}} \sqrt{20 \pi n} \left( \frac{10n}{e} \right)^{10n}}{\frac{32}{3} \pi n \left( \frac{\frac{16}{3} n}{e} \right)^{\frac{32}{3}n}} \leq \frac{C^n}{\sqrt{\frac{4}{3}n}} \left(\frac{2}{3} n\right)^{-\frac{2}{3}n} \leq \frac{C^n}{\Gamma\left(\frac{2}{3}n +1\right)}.
	\end{align*}
	Consequently, we have for $S_n$ that
	\begin{align*}
		S_{n} &\sim C^n \left(2\pi \right)^{\frac{d}{8}} n^{\frac{15}{8}} \left(\frac{1}{\Gamma\left(\frac{2}{3}n +1\right)} \right)^{\frac{1}{4}}.
	\end{align*}
	Furthermore, using \Cref{lem:prodSumInterchange} we have for every $n\geq 1$ that
	\begin{align*}
		\sum_{\eta_0, \dots \eta_{n-1}=1}^d \prod_{j=1}^{n} D_{\eta_{j-1}}^4 \lambda_{\eta_{j-1}}^4 = \left( \sum_{k=1}^d D_k^4 \lambda_k^4 \right)^n.
	\end{align*}
	Moreover, due to the assumptions on $H$ there exists a finite constant $K>0$ which is independent of $d$ and $H$ such that $K_{d,H} \leq K$, cf. \eqref{eq:constantKdh}. Consequently, there exists a constant $C>0$ independent of $d$, $\varepsilon$ and $n$ such that for $n$ sufficiently large
	\begin{align*}
		\Dcal_n^2 &:= \sum_{\eta_0, \dots \eta_{n-1}=1}^d 2^{8n} \frac{K_{d,H}^{4n} \cdot T^{\frac{n}{3}}}{\sqrt{2\pi}^{4dn}} \left(\prod_{j=1}^n D_{\eta_{j-1}}^4 \lambda_{\eta_{j-1}}^4\right) T^{4n} S_{n}\\
		&\sim \left(\frac{ n^{\frac{15}{2}} C^n}{\Gamma\left(\frac{2}{3}n +1\right)} \right)^{\frac{1}{4}}
	\end{align*}
	and thus due to the comparison test
	\begin{align*}
		\sum_{n\geq 1} \Dcal_n < \infty.
	\end{align*}
	Hence, there exists a constant $C_2>0$ independent of $d$ and $\varepsilon$ such that
	\begin{align*}
		\left\Vert \Ifrak_2 \right\Vert_{L^{2}(\Omega;\Hcal)}^{2} &\leq C_2 C_{H_m,T}^4 \left( \frac{s - u}{su}\right)^{2 \gamma_m} s^{2(H_m-\frac{1}{2} - \gamma_m)} (t - s )^{2\left(H_m-\frac{1}{2}-\gamma_m \right)},
	\end{align*}
	and thus we can find a $\beta_2 \in \left(0, \frac{1}{2} \right)$ sufficiently small such that
	\begin{align*}
		\int_{0}^{t} \int_{0}^{t} \frac{\left\Vert \Ifrak_2 \right\Vert_{L^{2}(\Omega;\Hcal)}^{2}}{\left\vert s-u\right\vert^{1+2\beta_2}} ds du  \lesssim C_{H_m,T}^4 < \infty.
	\end{align*}
	Equivalently, we can show for $\Ifrak_3$ that there exists a $\beta_3 \in \left(0,\frac{1}{2} \right)$ such that
	\begin{align*}
		\int_{0}^{t} \int_{0}^{t} \frac{\left\Vert \Ifrak_3 \right\Vert_{L^{2}(\Omega;\Hcal)}^{2}}{\left\vert s-u\right\vert^{1+2\beta_2}} ds du  \lesssim C_{H_m,T}^4 < \infty,
	\end{align*}
	where $C_{H_m,T} = C \cdot c_{H_m}$ due to \Cref{lem:iterativeInt}. Here, $c_{H_m}$ is the constant in \eqref{eq:kernel}. Thus, we can find a constant $\tilde{C}>0$ independent of $H_m$ such that $\sup_{H \in (0,\frac{1}{6})} C_{H,T} \leq C<\infty$. Finally, we get with $\beta_m := \min \lbrace \beta_1, \beta_2, \beta_3 \rbrace$ that we can find $\gamma_m$, $m\geq 1$, such that
	\begin{align*}
		&\sum_{m\geq 1} \frac{1}{(1-2^{-2(\beta_m-\alpha_m)})\gamma_m^{2}} \int_{0}^{t} \int_{0}^{t} \frac{\left\Vert D_s^m X_t^{d,\varepsilon} - D_u^m X_t^{d,\varepsilon}  \right\Vert_{L^{2}(\Omega;\Hcal)}^{2}}{\left\vert s-u\right\vert^{1+2\beta_m}} ds du \\
		&\quad \leq \sum_{m\geq 1} \frac{1}{(1-2^{-2(\beta_m-\alpha_m)})\gamma_m^{2}} \int_{0}^{t} \int_{0}^{t} \frac{\lambda_m^2 \sum_{l=1}^3 \left\Vert \Ifrak_l  \right\Vert_{L^{2}(\Omega;\Hcal)}^{2}}{\left\vert s-u\right\vert^{1+2\beta_m}} ds du \\
		&\quad \lesssim \sum_{m\geq 1} \frac{\lambda_m^2 \tilde{C}^4}{(1-2^{-2(\beta_m-\alpha_m)})\gamma_m^{2}} < \infty,
	\end{align*}
	uniformly in $d\geq 1$ and $\varepsilon>0$. Similarly, we can show that
	\begin{align}\label{eq:MalliavinEstimate}
		\sum_{m\geq 1}\gamma_m^{-2} \left\Vert D^m X_t^{d,\varepsilon}  \right\Vert_{L^{2}(\Omega;L^{2}([0,1];\Hcal))}^{2} < \infty
	\end{align}
	uniformly in $d\geq 1$ and $\varepsilon>0$ and consequently the compactness criterion \Cref{thm:compactCriterion} yields the result.
\end{proof}

\subsection{$\Fbb^\Bb$ adaptedness and strong solution}

Finally, we can state and prove the main statement of this paper

\begin{theorem}
	Let $b \in \Bfrak([0,T] \times \Hcal; \Hcal)$. Then SDE \eqref{eq:MainSDE} has a unique Malliavin differentiable strong solution.
\end{theorem}
\begin{proof}
	Let $(X_t)_{t\in[0,T]}$ be a weak solution of SDE \eqref{eq:MainSDE} which is unique in law due to \Cref{prop:weakSolution}.  Due to \Cref{lem:weakConvergence} we know that for every bounded globally Lipschitz continuous function $\phi: \Hcal \to \Rbb$
	\begin{align*}
		\phi(X_t^{d,\varepsilon}) \xrightarrow[\varepsilon \to 0, ~d \to \infty]{} \EW{\phi(X_t) \vert \Fcal_t^W}
	\end{align*}
	weakly in $L^2(\Omega,\Fcal_t^W)$. Furthermore, by \Cref{thm:compactnessApplication} there exist subsequences $\lbrace d_k \rbrace_{k\geq 1}$ and $\lbrace \varepsilon_n \rbrace_{n\geq 1}$ such that
	\begin{align*}
		\phi(X_t^{d_k,\varepsilon_n}) \xrightarrow[n \to \infty,~ d \to \infty]{} \phi\left(\EW{X_t \vert \Fcal_t^W}\right)
	\end{align*}
	strongly in $L^2(\Omega,\Fcal_t^W)$. Uniqueness of the limit yields that $X_t$ is $\Fcal_t^W$--measurable for all $t\in [0,T]$. Since $\Fbb^W = \Fbb^\Bb$, we get that $(X_t)_{t\in[0,T]}$ is a unique strong solution of SDE \eqref{eq:MainSDE}. Malliavin differentiability follows by \eqref{eq:MalliavinEstimate} and noting that the estimate holds also for $\gamma_m \equiv 1$.
\end{proof}

\section{Example}

	In this section we give an example of a drift function $b \in \Bfrak([0,T]\times \Hcal; \Hcal)$ to show that the class does not merely contain the null function.\par 
	Let $f_k \in L^1(\ell^2; L^\infty([0,T];\ell^2))$, $k\geq 1$, i.e. for all $k\geq 1$ we have for all $z \in \ell^2$
	\begin{align}\label{eq:assumptionsF}
		&\sup_{t\in[0,T]} \vert f_k(t,z) \vert \leq C_k^f < \infty
		&\sup_{d\geq 1} \int_{\Rbb^d} \sup_{t\in[0,T]} \vert f_k(t,z) \vert dz \leq D_k^f < \infty,
	\end{align}
	such that $C^f, D^f \in \ell^1$ and define for every $k\geq 1$ an operator $A_k: \Hcal \to \Hcal$ which is invertible on $A_k \Hcal$ such that for all $k\geq 1$
	\begin{align*}
		\det \left(A_k^{-1} \sqrt{Q}^{-1} \sqrt{\Kcal}^{-1}\right) \leq \Dcal_k^A< \infty,
	\end{align*}
	where $\Dcal^A \in \ell^1$. Then, we define
	\begin{align*}
		b_k(t, y) := f_k(t,\tau^{-1} A_k y).
	\end{align*}
	This yields
	\begin{align*}
		&\sup_{t\in[0,T]} \vert b_k(t,y) \vert = \sup_{t\in[0,T]} \vert f_k(t,\tau^{-1} A_k y) \vert \leq C_k^f, \\
		&\int_\Hcal \sup_{t\in[0,T]} \left\vert b_k \left(t, \sqrt{Q} \sqrt{\Kcal} y\right) \right\vert dy = \int_\Hcal \sup_{t\in[0,T]} \left\vert f_k \left(t, \tau^{-1} A_k \sqrt{Q} \sqrt{\Kcal} y \right) \right\vert dy \\
		&\quad = \int_{\tau^{-1} A_k \Hcal} \sup_{t\in[0,T]} \vert f_k(t,z) \vert \det\left(A_k^{-1} \sqrt{Q}^{-1} \sqrt{\Kcal}^{-1}\right) dz \leq D_k^f \Dcal_k^A.
	\end{align*}
	Due to the definition $C^f \in \ell^1$ and $D^f \cdot \Dcal^A \in \ell^1$ and thus $b \in \Bfrak([0,T] \times \Hcal; \Hcal)$. \par
	A possible choice for $f$ is
	\begin{align*}
		f_k(t,z) = C_k^f \cdot e^{- t} \cdot e^{- D_k^f \frac{\vert z \vert}{2}} \left(a \mathbbm{1}_{\lbrace z\in A \rbrace}+ b \mathbbm{1}_{\lbrace z\in A^c \rbrace} \right),
	\end{align*}
	where $a,b \in \Rbb$ and $A \subset \Hcal$, which obviously fulfills the assumptions \eqref{eq:assumptionsF}. The operator $A_k$, $k\geq 1$, can for example be chosen such that there exists a finite subset $N_k \subset \Nbb$ such that for all $k\geq 1$
	\begin{align*}
		\prod_{n \in N_k} \lambda_k^{-1} \sqrt{\Kfrak_{H_k}}^{-1} \leq C.
	\end{align*}
	and we have for every $x\in \Hcal$
	\begin{align*}
		A_k x = \Dcal_k^A \sum_{n \in N_k} x^{(n)} e_n.
	\end{align*}
	Then $A_k$ is invertible on $A_k \Hcal$ for every $k\geq 1$ and
	\begin{align*}
		\det\left(A_k^{-1} \sqrt{Q}^{-1} \sqrt{\Kcal}^{-1}\right) = \Dcal_k^A \prod_{n \in N_k} \lambda_k^{-1} \sqrt{\Kfrak_{H_k}}^{-1} \leq C \Dcal_k^A.
	\end{align*}

\appendix
\section{Compactness Criterion}
	The following result which is originally due to \cite{daprato1992compact} in the finite dimensional case and which can be e.g. found in \cite{bogachev2010differentiable}, provides a compactness criterion of square integrable cylindrical Wiener processes on a Hilbert space.

\begin{theorem}\label{thm:generalCompactCriterion}
	Let $(B_{t})_{t\in[0,T]}$ be a cylindrical Wiener process on a separable Hilbert space $\Hcal$ with respect to a complete probability space $(\Omega ,\mathcal{F},\mu )$, where $\mathcal{F}$ is generated by $(B_{t})_{t\in[0,T]}$. Further, let $\mathcal{L}_{HS}(\Hcal,\mathbb{R})$ be the space of Hilbert-Schmidt operators from $\Hcal$ to $\mathbb{R}$ and let $D: \mathbb{D}^{1,2}\longrightarrow L^{2}(\Omega ;L^{2}([0,T])\otimes \mathcal{L}_{HS}(\Hcal,\mathbb{R}))$ be the Malliavin derivative in the direction of $(B_{t})_{t\in[0,T]}$, where $\mathbb{D}^{1,2}$ is the space of Malliavin differentiable random variables in $L^{2}(\Omega ).\newline$
Suppose that $C$ is a self-adjoint compact operator on $L^{2}([0,T])\otimes \mathcal{L}_{HS}(\Hcal,\mathbb{R})$ with dense image. Then for any $c>0$ the set
\begin{equation*}
\mathcal{G}=\left\{ G\in \mathbb{D}^{1,2}:\left\Vert G\right\Vert_{L^{2}(\Omega )}+\left\Vert C^{-1}DG\right\Vert _{L^{2}(\Omega;L^{2}([0,T])\otimes \mathcal{L}_{HS}(\Hcal,\mathbb{R}))}\leq c\right\} 
\end{equation*}
is relatively compact in $L^{2}(\Omega)$.
\end{theorem}

In this paper we aim at using a special case of the previous theorem, which is more suitable for explicit estimations. To this end we need the
following auxiliary result from \cite{daprato1992compact}.

\begin{lemma}\label{lem:compactCriterion} 
	Denote by $v_{s},s\geq 0$, with $v_{0}=1$ the Haar basis of $L^{2}([0,1])$. Define for any $0<\alpha <\frac{1}{2}$ the operator $A_{\alpha }$ on $L^{2}([0,1])$ by
\begin{equation*}
A_{\alpha }v_{s}=2^{i\alpha }v_{s},\quad \text{ if }s=2^{i}+j,\quad i\geq 0,\quad 0\leq j\leq 2^{i},
\end{equation*}%
and 
\begin{equation*}
A_{\alpha }1=1.
\end{equation*}%
Then for $\alpha <\beta <\frac{1}{2}$ we have that
\begin{align*}
\left\Vert A_{\alpha }f\right\Vert _{L^{2}([0,1])}^{2} \leq 2(\left\Vert f\right\Vert _{L^{2}([0,1])}^{2}+\frac{1}{1-2^{-2(\beta -\alpha )}}\int_{0}^{1}\int_{0}^{1}\frac{\left\vert f(t)-f(u)\right\vert^{2}}{\left\vert t-u\right\vert ^{1+2\beta }}dtdu).
\end{align*}
\end{lemma}

\begin{theorem}\label{thm:compactCriterion}
	Let $D^{k}$ be the Malliavin derivative in the direction of the $k$-th component of $(B_{t})_{t\in[0,T]}$. In addition, let $0<\alpha_{k}<\beta _{k}<\frac{1}{2}$ and $\gamma _{k}>0$ for all $k\geq 1$. Define the sequence $\mu_{s,k}=2^{-i\alpha_{k}}\gamma _{k}$, if $s=2^{i}+j$, $i\geq 0$, $0\leq j\leq 2^{i},$ $k\geq 1$. Assume that $\mu_{s,k}\longrightarrow 0$ for $s,k\longrightarrow \infty $. Let $c>0$ and $\mathcal{G}$ the collection of all $G\in \mathbb{D}^{1,2}$ such that
\begin{equation*}
\left\Vert G\right\Vert _{L^{2}(\Omega )}\leq c,
\end{equation*}
\begin{equation*}
\sum_{k\geq 1}\gamma _{k}^{-2}\left\Vert D^{k}G\right\Vert _{L^{2}(\Omega;L^{2}([0,1]))}^{2}\leq c,
\end{equation*}
and
\begin{equation*}
\sum_{k\geq 1}\frac{1}{(1-2^{-2(\beta_{k}-\alpha_{k})})\gamma_{k}^{2}}\int_{0}^{1}\int_{0}^{1}\frac{\left\Vert D_{t}^{k}G-D_{u}^{k}G\right\Vert_{L^{2}(\Omega)}^{2}}{\left\vert t-u\right\vert ^{1+2\beta _{k}}}dtdu\leq c.
\end{equation*}
Then $\mathcal{G}$ is relatively compact in $L^{2}(\Omega)$.
\end{theorem}
\begin{proof}
	As before denote by $v_{s},s\geq 0$, with $v_{0}=1$ the Haar basis of $L^{2}([0,1])$ and by $e_{k}^{\ast }= \langle e_k, \cdot \rangle_{H}, k\geq 1$, an orthonormal basis of $\mathcal{L}_{HS}(\Hcal,\mathbb{R})$, where $e_k, k \geq 0$, is an orthonormal basis of $\Hcal$. Define a self-adjoint compact operator $C$ on $L^{2}([0,1])\otimes \mathcal{L}_{HS}(\Hcal,\mathbb{R})$ with dense image by
\begin{equation*}
C(v_{s}\otimes e_{k}^{\ast })=\mu_{s,k}v_{s}\otimes e_{k}^{\ast },\quad s\geq 0,\quad k\geq 1.
\end{equation*}
Then it follows for $G\in \mathbb{D}^{1,2}$ from \Cref{lem:compactCriterion} that 
\begin{align*}
&\left\Vert C^{-1}DG\right\Vert _{L^{2}(\Omega ;L^{2}([0,1])\otimes \mathcal{L}_{HS}(\Hcal,\mathbb{R}))}^{2} \\
&\quad =\sum_{k\geq 1}\sum_{s\geq 0}\mu _{s,k}^{-2}E[\left\langle DG,v_{s}\otimes e_{k}^{\ast }\right\rangle _{L^{2}([0,1])\otimes \mathcal{L}_{HS}(\Hcal,\mathbb{R}))}^{2}] \\
&\quad =\sum_{k\geq 1}\gamma _{k}^{-2}\left\Vert A_{\alpha _{k}}D^{k}G\right\Vert_{L^{2}(\Omega ;L^{2}([0,1]))}^{2} \\
&\quad \leq 2\sum_{k\geq 1}\gamma _{k}^{-2}\left\Vert D^{k}G\right\Vert_{L^{2}(\Omega ;L^{2}([0,1]))}^{2} \\
&\qquad+2\sum_{k\geq 1}\frac{1}{(1-2^{-2(\beta _{k}-\alpha _{k})})\gamma _{k}^{2}} \int_{0}^{1}\int_{0}^{1}\frac{\left\Vert D_{t}^{k}G-D_{u}^{k}G\right\Vert_{L^{2}(\Omega )}^{2}}{\left\vert t-u\right\vert ^{1+2\beta _{k}}}dtdu \\
&\quad \leq M
\end{align*}
for a constant $M<\infty $. So using \Cref{thm:generalCompactCriterion} we obtain the result.
\end{proof}

\section{Integration by parts formula}
		In this section we derive an integration by parts formula similar to \cite{BNP} which is used in the proof of \Cref{thm:compactnessApplication} to verify the conditions of the compactness criterion \Cref{thm:compactCriterion}. Before stating the integration by parts formula, we start by giving some definitions and notations frequently used during the course of this section. \par
		
	Let $n$ be a given integer. We consider the function $f:[0,T]^n\times (\Rbb^{d})^n \rightarrow \Rbb$ of the form 
\begin{align}\label{f}
	f(s,z)=\prod_{j=1}^n f_j (s_j, z_j), \quad s = (s_1, \dots, s_n) \in [0,T]^n, \quad z = (z_1, \dots ,z_n) \in (\Rbb^{d})^n,
\end{align}
where $f_j : [0,T] \times \Rbb^{d} \rightarrow \Rbb$, $j=1,\dots, n$, are compactly supported smooth functions. Further, we deal with the function $\varkappa: [0,T]^n \rightarrow \Rbb$ which is of the form
\begin{align}\label{kappa}
	\varkappa (s) = \prod_{j=1}^n \varkappa_j (s_j), \quad s \in [0,T]^n,
\end{align}
with integrable factors $\varkappa_j : [0,T] \rightarrow \Rbb$, $j=1,\dots, n$.

Let $\alpha_j$ be a multi-index and $D^{\alpha_j}$ its corresponding differential operator. For $\alpha := (\alpha_1, \dots, \alpha_n) \in \Nbb_{0}^{d\times n}$ we define the norm $|\alpha | = \sum_{j=1}^n \sum_{k=1}^d \alpha_j^{(k)}$ and write 
\begin{align*}
	 D^{\alpha} f(s,z) = \prod_{j=1}^n D^{\alpha_j} f_j (s_j, z_j).
\end{align*}

Let $k$ be an arbitrary integer. Given $(s,z) = (s_1, \dots, s_{kn} ,z_1, \dots, z_n) \in [0,T]^{kn} \times (\Rbb^{d})^n$ and a shuffle permutation $\sigma \in \Scal(n,n)$ we define the shuffled functions
\begin{equation*}
	f_{\sigma}(s,z) := \prod_{j=1}^{kn} f_{[\sigma (j)]} (s_j, z_{[\sigma (j)]})
\end{equation*}
and 
\begin{equation*}
	\varkappa_{\sigma }(s) := \prod_{j=1}^{kn} \varkappa_{\lbrack \sigma(j)]}(s_j),
\end{equation*}
where $[j]$ is equal to $(j-in)$ if $(in+1)\leq j \leq (i+1)n$, $i=0, \dots, (k-1)$. For a multi-index $\alpha$, we define
\begin{align}\label{eq:PsiF}
	\Psi_{\alpha}^{f} (\theta, t, z, H, d) := \left( \prod_{k=1}^{d} \sqrt{(2\left\vert \alpha^{(k)}\right\vert)!} \right) \sum_{\sigma \in \Scal(n,n)} \int_{\Delta _{\theta, t}^{2n}} \left\vert f_{\sigma}(s,z)\right\vert \vert \Delta s \vert^{-H \left(1+\alpha_{[\sigma(\Delta)]}\right)} ds,
\end{align}
and
\begin{align}\label{eq:PsiKappa}
	\Psi_{\alpha}^{\varkappa}(\theta, t, H, d) := \left( \prod_{k=1}^{d} \sqrt{(2\left\vert \alpha^{(k)}\right\vert)!} \right) \sum_{\sigma \in \Scal(n,n)} \int_{\Delta_{\theta, t}^{2n}} \left\vert \varkappa_{\sigma}(s)\right\vert \vert \Delta s \vert^{-H \left(1+\alpha_{[\sigma(\Delta)]} \right)} ds,
\end{align}
where for any $a,b \in \Rbb$
\begin{small}
\begin{align*}
	&\vert \Delta s \vert^{H_k \left(a+b\cdot \alpha_{[\sigma(\Delta)]}^{(k)} \right)} :=  \vert s_1 \vert^{H_k \left(a + b\left(\alpha_{[\sigma(1)]}^{(k)} + \alpha_{[\sigma(2n)]}^{(k)}\right) \right)} \prod_{j=2}^{2n} \left\vert s_{j}-s_{j-1}\right\vert ^{H_k \left( a + b\left(\alpha _{\lbrack \sigma(j)]}^{(k)} + \alpha_{\lbrack \sigma(j-1)]}^{(k)} \right) \right)}, \\
	&\vert \Delta s \vert^{H \left(a+b\cdot \alpha_{[\sigma(\Delta)]} \right)} := \prod_{k=1}^d \vert \Delta s \vert^{H_k \left(a+b\cdot \alpha_{[\sigma(\Delta)]}^{(k)} \right)}.
\end{align*}
\end{small}

\begin{theorem}\label{thm:mainThmLocalTime} 
	Suppose the functions $\Psi_{\alpha }^{f}(\theta, t, z, H, d)$ and $\Psi_{\alpha }^{\varkappa}(\theta, t, H, d)$ defined in \eqref{eq:PsiF} and \eqref{eq:PsiKappa}, respectively, are finite. Then, 
	\begin{align}\label{eq:LambdaDef}
	\Lambda_{\alpha}^f (\theta, t, z) := (2\pi)^{-dn} \int_{(\Rbb^{d})^n} \int_{\Delta_{\theta,t}^n} \prod_{j=1}^n f_j (s_j, z_j)(-iu_j)^{\alpha_j} e^{-i\left\langle u_j, \Bhat_{s_j}^{d,H} - z_j \right\rangle }dsdu,
\end{align}
	where $\Bhat_{t}^{d,H} := \left( \frac{B_t^{H_1}}{\sqrt{\Kfrak_{H_1}}}, \dots, \frac{B_t^{H_d}}{\sqrt{\Kfrak_{H_d}}} \right)^\top$ and $\Kfrak_{H_k}$ is the constant in \Cref{lem:localNonDeterminism}, is a square integrable random variable in $L^{2}(\Omega)$ and
\begin{align}\label{eq:supestL}
	\EW{\left\vert \Lambda_{\alpha }^{f} (\theta, t, z)\right\vert^{2}} \leq \frac{T^{\frac{\vert \alpha \vert}{6}}}{(2\pi)^{dn}} \Psi_{\alpha}^{f} (\theta, t, z, H, d).  
\end{align}
	Furthermore,
\begin{align}\label{eq:intestL}
	&\Eabs{\int_{(\mathbb{R}^{d})^n}\Lambda _{\alpha }^{\varkappa f}(\theta, t, z)dz} \leq  \frac{T^{\frac{\vert \alpha \vert}{12}}}{\sqrt{2\pi}^{dn}} (\Psi _{\alpha }^{\varkappa }(\theta, t, H, d))^{\frac{1}{2}} \prod_{j=1}^n \left\Vert f_{j}\right\Vert_{L^{1}(\mathbb{R}^{d};L^{\infty}([0,T]))},
\end{align}
	and the integration by parts formula
	\begin{align}\label{eq:integrationByParts}
		\int_{\Delta_{\theta ,t}^n} D^{\alpha} f \left(s, \Bhat_s^{d,H} \right) ds = \int_{(\Rbb^{d})^n} \Lambda_{\alpha}^f(\theta, t, z) dz,
	\end{align}
	holds.
\end{theorem}
\begin{proof}
	For notational simplicity we consider merely the case $\theta = 0$ and write $\Lambda_{\alpha }^{f} (t, z) := \Lambda_{\alpha }^{f} (0, t, z)$. For any integrable function $g:(\mathbb{R}^{d})^{n}\longrightarrow \mathbb{C}$ we have that
\begin{align*}
	&\left\vert \int_{(\Rbb^d)^{n}}g(u_{1},...,u_{n})du_{1}...du_n \right\vert ^{2} \\
	&\quad = \int_{(\Rbb^d)^{n}}g(u_{1},...,u_n)du_{1}...du_n\int_{(\Rbb^d)^{n}}\overline{g(u_{n+1},...,u_{2n})}du_{n+1}...du_{2n} \\
	&\quad =\int_{(\Rbb^d)^{n}}g(u_{1},...,u_n)du_{1}...du_n(-1)^{dn}\int_{(\Rbb^d)^{n}}\overline{g(-u_{n+1},...,-u_{2n})} du_{n+1}...du_{2n},
\end{align*}
	where the change of variables $(u_{n+1},...,u_{2n})\longmapsto(-u_{n+1},...,-u_{2n})$ was applied in the last equality. Thus,
\begin{align*}
	\left\vert \Lambda _{\alpha }^{f}(t,z)\right\vert ^{2} &= (2\pi )^{-2dn}(-1)^{dn}\int_{(\Rbb^d)^{2n}}\int_{\Delta_{0,t}^n}\prod_{j=1}^{n}f_{j}(s_{j},z_{j})(-iu_{j})^{\alpha_{j}}e^{-i\left\langle u_{j}, \Bhat^{d,H}_{s_{j}}-z_{j}\right\rangle }ds \\
	&\quad \times \int_{\Delta_{0,t}^n}\prod_{j=n+1}^{2n}f_{[j]}(s_{j},z_{[j]})(-iu_{j})^{\alpha_{\lbrack j]}}e^{-i\left\langle u_{j}, \Bhat^{d,H}_{s_{j}}-z_{[j]}\right\rangle}ds du \\
	&= (2\pi )^{-2dn}(-1)^{dn} ~ i^{\vert \alpha \vert} \sum_{\sigma \in \Scal(n,n)}\int_{(\Rbb^d)^{2n}}\left( \prod_{j=1}^{n}e^{-i\left\langle z_{j},u_{j}+u_{j+n}\right\rangle }\right)  \\
	&\quad \times \int_{\Delta _{0,t}^{2n}}f_{\sigma }(s,z) \left( \prod_{j=1}^{2n}u_{\sigma(j)}^{\alpha _{\lbrack \sigma (j)]}} \right) \exp \left\{- i \sum_{j=1}^{2n}\left\langle u_{\sigma (j)}, \Bhat_{s_{j}}^{d,H} \right\rangle \right\} ds du,
\end{align*}
	where we applied shuffling in the sense of \eqref{shuffleIntegral}. Taking the expectation on both sides together with the independence of the fractional Brownian motions $B^{H_k}$, $k=1,...,d$, yields that
\begin{align}\label{Lambda}
	&\EW{\left\vert \Lambda _{\alpha }^{f}(t,z)\right\vert ^{2}} \notag \\
	&\quad =(2\pi )^{-2dn}(-1)^{dn}~ i^{\vert \alpha \vert} \sum_{\sigma \in \Scal(n,n)}\int_{(\Rbb^d)^{2n}}\left( \prod_{j=1}^{n} e^{-i\left\langle z_{j},u_{j}+u_{j+n}\right\rangle }\right) \notag \\
	&\qquad \times \int_{\Delta _{0,t}^{2n}}f_{\sigma }(s,z) \left( \prod_{j=1}^{2n} u_{\sigma(j)}^{\alpha _{\lbrack \sigma (j)]}} \right) \exp\left\{ -\frac{1}{2} \Var{\sum_{j=1}^{2n} \left\langle u_{\sigma (j)}, \Bhat_{s_{j}}^{d,H} \right\rangle}\right\} ds du \notag \\
	&\quad = (2\pi )^{-2dn}(-1)^{dn} ~ i^{\vert \alpha \vert} \sum_{\sigma \in \Scal(n,n)}\int_{(\Rbb^d)^{2n}} \left(\prod_{j=1}^n e^{-i\left\langle z_j ,u_j + u_{j+n}\right\rangle}\right)\notag  \\
	&\qquad \times \int_{\Delta _{0,t}^{2n}}f_{\sigma }(s,z) \left( \prod_{j=1}^{2n} u_{\sigma(j)}^{\alpha _{\lbrack \sigma (j)]}} \right) \exp\left\{ -\frac{1}{2} \sum_{k=1}^{d} \Var{\sum_{j=1}^{2n}u_{\sigma(j)}^{(k)} \frac{B_{s_{j}}^{H_{k}}}{\sqrt{\Kfrak_{H_k}}}} \right\}ds du \notag \\
	&\quad =(2\pi )^{-2dn}(-1)^{dn} ~ i^{\vert \alpha \vert} \sum_{\sigma \in \Scal(n,n)}\int_{(\mathbb{R}^{d})^{2n}}\left( \prod_{j=1}^{n} e^{-i\left\langle z_{j},u_{j}+u_{j+n}\right\rangle }\right) \notag \\
	&\qquad \times \int_{\Delta _{0,t}^{2n}} f_{\sigma }(s,z) \left( \prod_{j=1}^{2n} u_{\sigma(j)}^{\alpha _{\lbrack \sigma (j)]}} \right) \prod_{k=1}^{d}\exp \left\{ -\frac{1}{2 \Kfrak_{H_k}} ( u_\sigma^{(k)})^\top \Sigma_{k} u_{\sigma}^{(k)} \right\} ds  du,
\end{align}
	where $u_{\sigma}^{(k)} = \left( u_{\sigma(1)}^{(k)}, \dots, u_{\sigma(2n)}^{(k)} \right)^\top$ and
\begin{equation*}
	\Sigma_{k} = \Sigma_{k}(s):= \left(\EW{B_{s_{i}}^{H_k}B_{s_{j}}^{H_k}}\right)_{1\leq i,j\leq 2n}.
\end{equation*}
	Moreover, we obtain for every $\sigma \in \Scal(n,n)$ that
\begin{align}\label{Lambda2}
	&\int_{\Delta _{0,t}^{2n}}\left\vert f_{\sigma }(s,z)\right\vert \int_{(\Rbb^d)^{2n}} \prod_{k=1}^{d}  \left( \left( \prod_{j=1}^{2n}\left\vert u_{\sigma(j)}^{(k)}\right\vert^{\alpha_{[\sigma (j)]}^{(k)}} \right) e^{ -\frac{1}{2 \Kfrak_{H_k}} (u_\sigma^{(k)})^\top \Sigma_k u_{\sigma}^{(k)}} \right) du ds \notag \\
	&\quad =\int_{\Delta _{0,t}^{2n}}\left\vert f_{\sigma }(s,z)\right\vert \prod_{k=1}^{d} \left( \int_{\Rbb^{2n}} \left(\prod_{j=1}^{2n}\left\vert u_{j}^{(k)}\right\vert ^{\alpha_{[\sigma (j)]}^{(k)}}\right) e^{ -\frac{1}{2} \left\langle \frac{\Sigma_k}{\Kfrak_{H_k}} u^{(k)},u^{(k)}\right\rangle } du^{(k)} \right) ds,
\end{align}
	where $u^{(k)} := \left( u_1^{(k)}, \dots, u_{2n}^{(k)} \right)^\top.$ For every $1 \leq k \leq d$ we have by using substitution that
\begin{align}\label{eq:GaussianSub}
	&\int_{\Rbb^{2n}} \left(\prod_{j=1}^{2n}\left\vert u_{j}^{(k)}\right\vert^{\alpha_{[\sigma (j)]}^{(k)}}\right) e^{-\frac{1}{2}\left\langle \frac{\Sigma_k}{\Kfrak_{H_k}} u^{(k)}, u^{(k)}\right\rangle} du^{(k)} \\
	&\quad =\frac{\Kfrak_{H_k}^n}{(\det \Sigma_{k})^{1/2}}\int_{\Rbb^{2n}} \left(\prod_{j=1}^{2n} \left \vert \left\langle \sqrt{\Kfrak_{H_k}} \Sigma_{k}^{-1/2}u^{(k)},\etilde_{j}\right\rangle \right\vert^{\alpha _{[\sigma (j)]}^{(k)}}\right) e^{ -\frac{1}{2}\left\langle u^{(k)},u^{(k)}\right\rangle} du^{(k)}. \notag
\end{align}
	Considering a standard Gaussian random vector $Z\sim \Ncal(0,\Id_{2n})$, we get that
\begin{align}\label{eq:GaussianExp}
	&\int_{\Rbb^{2n}} \left(\prod_{j=1}^{2n}\left\vert \left\langle \Sigma_{k}^{-1/2}u^{(k)},\etilde_{j} \right\rangle \right\vert^{\alpha_{[\sigma(j)]}^{(k)}}\right) e^{ -\frac{1}{2}\left\langle u^{(k)},u^{(k)}\right\rangle} du^{(k)} \\
	&\quad = (2\pi )^{n}\EW{\prod_{j=1}^{2n}\left\vert \left\langle \Sigma_k^{-1/2}Z,\etilde_{j}\right\rangle \right\vert^{\alpha _{[\sigma(j)]}^{(k)}}}. \notag
\end{align}
	Using a Brascamp-Lieb type inequality which is due to \Cref{lem:LiWei}, we further get that
\begin{align*}
	&\EW{\prod_{j=1}^{2n}\left\vert \left\langle \Sigma_k^{-1/2}Z, \etilde_{j}\right\rangle \right\vert^{\alpha _{[\sigma (j)]}^{(k)}}} \leq \sqrt{\perm{A_k}}=\sqrt{\sum_{\pi \in S_{2\left\vert \alpha^{(k)}\right\vert }}\prod_{i=1}^{2\left\vert \alpha^{(k)}\right\vert}a_{i,\pi (i)}^{(k)}},
\end{align*}
	where $\left\vert \alpha^{(k)}\right\vert :=\sum_{j=1}^{n}\alpha _{j}^{(k)}$ and $\perm{A_k}$ is the permanent of the covariance matrix $A_k =(a_{i,j}^{(k)})_{1\leq i,j \leq 2\left\vert \alpha^{(k)}\right\vert}$ of the Gaussian random vector
\begin{align*}
	\underset{\alpha _{[\sigma (1)]}^{(k)}\text{ times}}{\left( \underbrace{\left\langle \Sigma_k^{-1/2}Z,\etilde_{1}\right\rangle ,...,\left\langle
\Sigma_k^{-1/2}Z, \etilde_{1}\right\rangle}\right.}, \dots ,\underset{\alpha_{[\sigma (2n)]}^{(k)}\text{ times}}{\left.\underbrace{\left\langle \Sigma_k^{-1/2}Z, \etilde_{2n}\right \rangle ,...,\left\langle \Sigma_k^{-1/2}Z, \etilde_{2n}\right\rangle}\right)},
\end{align*}
	 and $S_{m}$ denotes the permutation group of size $m$. Using an upper bound for the permanent of positive semidefinite matrices which is due to \cite{AG}, we find that 
\begin{align}\label{PSD}
	\perm{A_k}=\sum_{\pi \in S_{2\left\vert \alpha^{(k)}\right\vert}}\prod_{i=1}^{2\left\vert \alpha^{(k)}\right\vert }a_{i,\pi (i)}^{(k)}\leq (2\left\vert \alpha^{(k)}\right\vert )!\prod_{i=1}^{2\left\vert \alpha^{(k)}\right\vert }a_{i,i}^{(k)}.
\end{align}
	Now let $\sum_{l=1}^{j-1}\alpha_{[\sigma(l)]}^{(k)}+1 \leq i \leq \sum_{l=1}^{j}\alpha_{[\sigma (l)]}^{(k)}$ for some fixed $j\in \{1,...,2n\}$. Then
\begin{align*}
	a_{i,i}^{(k)}=\EW{\left\langle \Sigma_k^{-1/2}Z,\etilde_j \right\rangle \left\langle \Sigma_k^{-1/2}Z,\etilde_j \right\rangle}.
\end{align*}
	Substitution gives moreover that
	\begin{small}
\begin{align}\label{eq:Covariation}
	&\EW{\left\langle \Sigma_k^{-1/2}Z, \etilde_{j}\right\rangle \left\langle \Sigma_{k}^{-1/2}Z, \etilde_j \right\rangle} 
	= (\det \Sigma_{k})^{1/2}\frac{1}{(2\pi )^{n}}\int_{\Rbb^{2n}}u_{j}^{2}\exp \left\lbrace -\frac{1}{2}\left\langle \Sigma_k u, u\right\rangle \right\rbrace du.
\end{align}
\end{small}
	Applying \Cref{lem:CD} we get
\begin{align}\label{eq:CDLemma2app}
	\int_{\Rbb^{2n}}u_{j}^{2}\exp \left\lbrace -\frac{1}{2}\left\langle \Sigma_k u, u\right\rangle \right\rbrace du &= \frac{(2\pi )^{(2n-1)/2}}{(\det \Sigma_k)^{1/2}}\int_{\Rbb} v^{2}\exp \left\lbrace -\frac{1}{2}v^{2} \right\rbrace dv\frac{1}{\sigma _{j}^{2}} \notag \\
	&= \frac{(2\pi )^n}{(\det \Sigma_k)^{1/2}}\frac{1}{\sigma _{j}^{2}},
\end{align}
	where $\sigma_{j}^{2}:=\Var{B_{s_{j}}^{H_k}\left\vert B_{s_{1}}^{H_k},...,B_{s_{2n}}^{H_k}\text{ without }B_{s_{j}}^{H_k} \right.}.$ \par 
Subsequently, we aim at the application of the strong local non-determinism property of the fractional Brownian motions, cf. \Cref{lem:localNonDeterminism}, i.e. for all $0<r<t \leq T$ exists a constant $\Kfrak_{H_k}$ depending on $H_k$ and $T$ such that
\begin{equation*}
	\Var{B_{t}^{H_k}\left\vert B_{s}^{H_k},\left\vert t-s\right\vert \geq r \right. } \geq \Kfrak_{H_k} r^{2H_k}.
\end{equation*}
	Hence, we get due to \Cref{lem:determinant1} and \Cref{lem:determinant2} that
\begin{equation}\label{eq:determinant1}
	(\det \Sigma_k(s))^{1/2}\geq \Kfrak_{H_k}^{\frac{(2n-1)}{2}}\left\vert s_{1}\right\vert^{H_k}\left\vert s_{2}-s_{1}\right\vert ^{H_k}...\left\vert s_{2n}-s_{2n-1}\right\vert ^{H_k},
\end{equation}
	and
\begin{align*}
	\sigma_1^{2} &\geq \Kfrak_{H_k} \left\vert s_2 - s_1 \right\vert^{2H_k}, \\
	\sigma_j^{2} &\geq \Kfrak_{H_k} \min \left\lbrace \left\vert s_{j}-s_{j-1}\right\vert^{2H_k},\left\vert s_{j+1}-s_{j}\right\vert ^{2H_k}\right\rbrace, ~ 2\leq j \leq 2n-1, \\
	\sigma_{2n}^{2} &\geq \Kfrak_{H_k} \left\vert s_{2n}-s_{2n-1}\right\vert^{2H_k}.
\end{align*}
	Thus,
\begin{align}\label{eq:boundVariance}
	&\prod_{j=1}^{2n}\sigma _{j}^{-2\alpha _{[\sigma (j)]}^{(k)}} \leq \Kfrak_{H_k}^{-2 \vert \alpha^{(k)} \vert} T^{4 H_k \vert \alpha^{(k)} \vert} \vert \Delta s \vert^{-2H_k \alpha_{[\sigma(\Delta)]}^{(k)}}.
\end{align}
	Concluding from \eqref{PSD}, \eqref{eq:Covariation}, \eqref{eq:CDLemma2app}, and \eqref{eq:boundVariance} we have that
\begin{align*}
	\perm{A_k} &\leq \left(2\left\vert \alpha^{(k)}\right\vert\right)! \prod_{i=1}^{2\left\vert \alpha ^{(k)}\right\vert }a_{i,i}^{(k)} \\
	&\leq \left(2\left\vert \alpha^{(k)}\right\vert \right)! \prod_{j=1}^{2n} \left( (\det \Sigma_k)^{1/2} \frac{1}{(2\pi )^{n}}\frac{(2\pi )^{n}}{(\det \Sigma_k)^{1/2}}\frac{1}{\sigma _{j}^{2}} \right)^{\alpha _{[\sigma (j)]}^{(k)}} \\
	&\leq \left(2\left\vert \alpha^{(k)}\right\vert \right)! \Kfrak_{H_k}^{-2 \vert \alpha^{(k)} \vert} T^{4 H_k \vert \alpha^{(k)} \vert} \vert \Delta s \vert^{-2H_k \alpha_{[\sigma(\Delta)]}^{(k)}}.
\end{align*}
	Consequently,
\begin{align*}
	&\EW{\prod_{j=1}^{2n}\left\vert \left\langle \Sigma_k^{-1/2}Z, \etilde_{j} \right\rangle\right\vert^{\alpha _{\lbrack \sigma (j)]}^{(k)}}} \leq \sqrt{(2\left\vert \alpha^{(k)}\right\vert )!} \Kfrak_{H_k}^{-\vert \alpha^{(k)} \vert} T^{2 H_k \vert \alpha^{(k)}\vert} \vert \Delta s \vert^{-H_k \alpha_{[\sigma(\Delta)]}^{(k)}}.
\end{align*}
	Therefore we get from \eqref{Lambda}, \eqref{Lambda2}, \eqref{eq:GaussianSub}, \eqref{eq:GaussianExp}, and \eqref{eq:determinant1} that
\begin{align*}
	&\EW{\left\vert \Lambda_{\alpha }^{f}(t,z)\right\vert ^{2}} \\
	&\quad \leq (2\pi)^{-2dn} \sum_{\sigma \in \Scal(n,n)}\int_{\Delta _{0,t}^{2n}}\left\vert f_{\sigma }(s,z)\right\vert \prod_{k=1}^{d} \left( \int_{\Rbb^{2n}}\left\vert u^{(k)}\right\vert^{\alpha ^{(k)}} e^{ -\frac{1}{2\Kfrak_{H_k}}\left\langle \Sigma_k u^{(k)},u^{(k)}\right\rangle} du^{(k)} \right) ds \\
	&\quad \leq (2\pi)^{-dn} \sum_{\sigma \in \Scal(n,n)} \int_{\Delta _{0,t}^{2n}} \left\vert f_{\sigma }(s,z)\right\vert \prod_{k=1}^{d} \left( \frac{\Kfrak_{H_k}^{n+\vert \alpha^{(k)}\vert}}{\left( \det \Sigma_k(s) \right)^{\frac{1}{2}}} \EW{ \prod_{j=1}^{2n} \left\vert \left\langle \Sigma_k^{-\frac{1}{2}} Z, \etilde_j \right\rangle \right\vert^{\alpha_{\sigma(j)}^{(k)}}} \right) ds \\
	&\quad \leq (2\pi)^{-dn} \sum_{\sigma \in \Scal(n,n)}\int_{\Delta _{0,t}^{2n}}\left\vert f_{\sigma }(s,z)\right\vert  \left( \prod_{k=1}^d \vert \Delta s \vert^{-H_k} \Kfrak_{H_k}^{\vert \alpha^{(k)}\vert + \frac{1}{2}} \right)\\
	&\qquad \times \prod_{k=1}^{d} \left( \sqrt{(2\left\vert \alpha^{(k)}\right\vert )!} \Kfrak_{H_k}^{-\vert \alpha^{(k)} \vert} T^{2 H_k \vert \alpha^{(k)} \vert} \vert \Delta s \vert^{-H_k \alpha_{[\sigma(\Delta)]}^{(k)}} \right) ds \\
	&\quad \leq (2\pi)^{-dn} T^{\frac{\vert \alpha \vert}{6}} \left( \prod_{k=1}^{d} \sqrt{\Kfrak_{H_k}} \sqrt{(2\left\vert \alpha^{(k)}\right\vert )!} \right) \sum_{\sigma \in \Scal(n,n)} \int_{\Delta_{0,t}^{2n}} \left\vert f_{\sigma }(s,z)\right\vert \vert \Delta s \vert^{- H \left(1 +  \alpha_{[\sigma(\Delta)]} \right)} ds.
\end{align*}
	Since $\sup_{k \geq 1} \Kfrak_{H_k} \in (0,1)$, inequality \eqref{eq:supestL} holds. \par 
	Next we prove the estimate \eqref{eq:intestL}. With inequality \eqref{eq:supestL}, we get that 
\begin{align*}
	&\Eabs{\int_{(\Rbb^{d})^n}\Lambda_{\alpha}^{\varkappa f}(\theta ,t,z)dz} \leq \int_{(\Rbb^{d})^n} \Ep{\Lambda_{\alpha}^{\varkappa f}(\theta ,t,z)}{2} dz \\
	&\quad \leq \frac{T^{\frac{\vert \alpha \vert}{12}}}{\sqrt{2\pi}^{dn}} \int_{(\Rbb^{d})^n}(\Psi_{\alpha }^{\varkappa f}(\theta,t,z,H,d))^{\frac{1}{2}}dz.
\end{align*}
	Taking the supremum over $[0,T]$ with respect to each function $f_{j}$, i.e.
\begin{equation*}
	\left\vert f_{[\sigma (j)]}(s_{j},z_{[\sigma (j)]})\right\vert \leq \sup_{s_{j}\in [0,T]}\left\vert f_{[\sigma (j)]}(s_{j},z_{[\sigma(j)]})\right\vert , ~ j=1,...,2n,
\end{equation*}
	yields that
\begin{align*}
	&\Eabs{\int_{(\Rbb^{d})^n}\Lambda _{\alpha}^{\varkappa f}(\theta ,t,z)dz}  \\
	&\quad \leq \frac{T^{\frac{\vert \alpha \vert}{12}}}{\sqrt{2\pi}^{dn}} \max_{\sigma \in \Scal(n,n)}\int_{(\Rbb^{d})^n}\left( \prod_{j=1}^{2n}\left\Vert f_{[\sigma (j)]}(\cdot,z_{[\sigma (j)]})\right\Vert _{L^{\infty }([0,T])}\right)^{\frac{1}{2}}dz \\
	&\qquad \times \left(\prod_{k=1}^{d}\sqrt{(2\left\vert \alpha^{(k)}\right\vert )!}\sum_{\sigma \in \Scal(n,n)}\int_{\Delta _{\theta,t}^{2n}}\left\vert \varkappa_{\sigma }(s)\right\vert \vert \Delta s \vert^{- H \left( 1 + \alpha_{[\sigma (\Delta)]} \right)} ds \right)^{\frac{1}{2}} \\
	&\quad = \frac{T^{\frac{\vert \alpha \vert}{12}}}{\sqrt{2\pi}^{dn}} \max_{\sigma \in \Scal(n,n)}\int_{(\Rbb^{d})^n}\left( \prod_{j=1}^{2n}\left\Vert f_{[\sigma (j)]}(\cdot,z_{[\sigma (j)]})\right\Vert _{L^{\infty }([0,T])}\right)^{\frac{1}{2}} dz ~(\Psi _{\alpha }^{\varkappa }(\theta, t, H, d))^{\frac{1}{2}} \\
	&\quad = \frac{T^{\frac{\vert \alpha \vert}{12}}}{\sqrt{2\pi}^{dn}} \int_{(\Rbb^{d})^n}\prod_{j=1}^{n}\left\Vert f_{j}(\cdot ,z_{j})\right\Vert _{L^{\infty }([0,T])}dz ~ (\Psi_{\alpha }^{\varkappa }(\theta, t, H, d))^{\frac{1}{2}} \\
	&\quad = \frac{T^{\frac{\vert \alpha \vert}{12}}}{\sqrt{2\pi}^{dn}} \left( \prod_{j=1}^{n}\left\Vert f_{j} \right\Vert_{L^{1}(\Rbb^{d};L^{\infty }([0,T]))} \right)  (\Psi_{\alpha }^{\varkappa }(\theta, t, H, d))^{\frac{1}{2}}.
\end{align*}
	Finally, we show the integration by parts formula \eqref{eq:integrationByParts}. Note that \emph{a priori} one cannot interchange the order of integration in \eqref{eq:LambdaDef}, since e.g. for $m=1$, $f \equiv 1$ one gets an integral of the Donsker-Delta function which is not a random variable in the usual sense. Therefore, we define for $R>0$, 
\begin{align*}
	\Lambda^f_{\alpha,R} (\theta,t,z) := (2\pi)^{-dn} \int_{B(0,R)} \int_{\Delta^n_{\theta,t}} \prod_{j=1}^n f_j(s_j,z_j) (-iu_j)^{\alpha_j} e^{-i \langle u_j, \Bhat_{s_j}^{d,H}-z_j\rangle } ds dv,
\end{align*}
	where $B(0,R):=\{v\in (\Rbb^d)^n : \, |v|<R\}$. This yields
\begin{align*}
	|\Lambda^f_{\alpha,R} (\theta,t,z)| \leq C_{R} \int_{\Delta^n_{\theta,t}} \prod_{j=1}^n |f_j(s_j, z_j)| ds 
\end{align*}
	for a sufficient constant $C_R$. Under the assumption that the above right-hand side is integrable over $(\Rbb^d)^n$, similar computations as above show that $\Lambda^f_{\alpha,R}(\theta,t,z)\to \Lambda^f_{\alpha} (\theta,t,z)$ in $L^2(\Omega)$ as $R\to \infty$ for all $\theta,t$ and $z$. By Lebesgue's dominated convergence theorem and the fact that the Fourier transform is an automorphism on the Schwarz space, we obtain
\begin{align*}
	&\int_{(\Rbb^d)^n}  \Lambda^f_{\alpha}(\theta,t,z) dz = \lim_{R\to \infty} \int_{(\Rbb^d)^n}  \Lambda^f_{\alpha, R} (\theta,t,z)dx\\
	&\quad = \lim_{R\to \infty} (2\pi)^{-dn}\int_{(\Rbb^d)^n} \int_{B(0,R)} \int_{\Delta_{\theta,t}^n}  \prod_{j=1}^n f_j(s_j, z_j) (-iu_j)^{\alpha_j} e^{-i  \left\langle u_j, \Bhat_{s_j}^{d,H}-z_j \right\rangle} dzduds\\
	&\quad = \lim_{R\to \infty}   \int_{\Delta_{\theta,t}^n}   \int_{B(0,R)} (2\pi)^{-dn} \int_{(\Rbb^d)^n}   \prod_{j=1}^n f_j(s_j, z_j)e^{i  \langle u_j,z_j \rangle} dz (-iu_j)^{\alpha_j} e^{-i  \left\langle u_j, \Bhat_{s_j}^{d,H} \right\rangle} duds\\
	&\quad = \lim_{R\to \infty} \int_{\Delta_{\theta,t}^n}    \int_{B(0,R)}  \prod_{j=1}^n \widehat{f}_j(s,-u_j)  (-iu_j)^{\alpha_j} e^{-i \left\langle u_j, \Bhat_{s_j}^{d,H} \right\rangle} duds\\
	&\quad = \int_{\Delta_{\theta,t}^n} D^{\alpha} f\left(s,\Bhat_s^{d,H}\right)ds
\end{align*}
which is exactly the integration by parts formula \eqref{eq:integrationByParts}.
\end{proof}

Applying \Cref{thm:mainThmLocalTime} we obtain the following crucial estimate (compare \cite{ABP18}, \cite{ACHP}, \cite{BNP}, and \cite{BOPP.17}):

\begin{proposition}\label{prop:mainestimate1}
	Let the functions $f$ and $\varkappa $ be defined as in \eqref{f} and \eqref{kappa}, respectively. Further, let $ 0\leq \theta' <\theta <t \leq T$ and for some $m \geq 1$
\begin{equation*}
	\varkappa_{j}(s)=(K_{H_{m}}(s,\theta )-K_{H_{m}}(s,\theta'))^{\varepsilon_{j}},~\theta <s<t,
\end{equation*}
for every $j=1,...,n$ with $(\varepsilon_{1},...,\varepsilon_n)\in \{0,1\}^{n}$. Let $\alpha \in (\Nbb_{0}^{d})^{n}$ be a multi-index. Assume there exists $\delta $ such that 
\begin{align}\label{eq:assumptionSumAlpha}
-\sum_{k=1}^{d} H_k \left(1+2\alpha _{j}^{(k)}\right)+\left(H_{m} - \frac{1}{2} - \gamma_{m}\right)\geq \delta >-1
\end{align}
	for all $j= 1, \dots n$ and $d \geq 1$, where $\gamma_m \in (0,H_m)$ is sufficiently small. Then there exist constants $C_T$ (depending on $T$) and $K_{d,H}$ (depending on $d$ and $H$), such that for any $0\leq \theta <t \leq T$ we have
\begin{align*}
	&\Eabs{\int_{\Delta _{\theta ,t}^{n}}\left( \prod_{j=1}^{n} D^{\alpha_{j}}f_{j}(s_{j}, \Bhat_{s_{j}})\varkappa _{j}(s_{j})\right) ds}  \\
	&~ \leq \frac{K_{d,H}^{n} \cdot T^{\frac{\vert \alpha \vert}{12}}}{\sqrt{2\pi}^{dn}} \left( C_T \left( \frac{\theta -\theta'}{\theta \theta'}\right)^{\gamma_m} \theta^{(H_m-\frac{1}{2} - \gamma_m)}\right)^{\sum_{j=1}^{n} \varepsilon_j} \prod_{j=1}^n \left\Vert f_{j}(\cdot ,z_{j})\right\Vert _{L^{1}(\mathbb{R}^{d};L^{\infty }([0,T]))} \\
	&\quad \times \frac{\left(\prod_{k=1}^{d} \left(2\left\vert \alpha ^{(k)}\right\vert\right)!\right)^{\frac{1}{4}}(t-\theta )^{-\sum_{k=1}^{d} H_k \left(n+2\left\vert \alpha^{(k)}\right\vert \right)+\left(H_m-\frac{1}{2}-\gamma_m\right)\sum_{j=1}^n \varepsilon _{j}+n}}{\Gamma(2n - \sum_{k=1}^{d}H_k(2n+4\left\vert \alpha^{(k)}\right\vert )+2(H_m - \frac{1}{2}-\gamma_m)\sum_{j=1}^n \varepsilon _{j})^{\frac{1}{2}}}.
\end{align*}
\end{proposition}

In order to prove this result we need the following two auxiliary results.

\begin{lemma}\label{lem:doubleint} 
	Let $H \in \left(0,\frac{1}{2}\right)$ and $t\in [0,T]$ be fixed. Then, there exists $\beta \in \left(0,\frac{1}{2}\right)$ and a constant $C>0$ independent of $H$ such that 
\begin{align*}
	\int_0^t \int_0^t \frac{|K_H(t, \theta^{\prime}) - K_H(t,\theta)|^2}{|\theta^{\prime}-\theta|^{1+2\beta}}d\theta d\theta^{\prime} \leq C < \infty.
\end{align*}
\end{lemma}
\begin{proof}
	Let $0 \leq \theta^{\prime}<\theta \leq t$ be fixed. Write 
\begin{equation*}
	K_H (t,\theta) - K_H(t,\theta^{\prime}) = c_H\left[f_t(\theta) - f_t(\theta^{\prime}) + \left(\frac{1}{2}-H\right) \left(g_t(\theta) - g_t(\theta^{\prime})\right)\right],
\end{equation*}
	where $f_t (\theta):= \left(\frac{t}{\theta} \right)^{H-\frac{1}{2}} (t-\theta)^{H-\frac{1}{2}}$ and $g_t(\theta) := \int_{\theta}^t \frac{f_u (\theta)}{u}du$.\par 
We continue with the estimation of $K_H (t,\theta) - K_H(t,\theta^{\prime})$. First, observe that there exists a constant $0<C<1$ such that 
\begin{equation} \label{eq:inequalityAlpha}
	\frac{y^{-\alpha} -x^{-\alpha}}{(x-y)^{\gamma}} \leq C y^{-\alpha-\gamma},
\end{equation}
for every $0<y<x<\infty$ and $\alpha :=(\frac{1}{2}-H) \in \left(0,\frac{1}{2}\right)$ as well as $0 < \gamma < \frac{1}{2}-\alpha$. Indeed, rewriting \eqref{eq:inequalityAlpha} yields using the substitution $z:= \frac{x}{y}$, $z\in (1,\infty)$,
\begin{align*}
	\frac{y^{-\alpha} -x^{-\alpha}}{(x-y)^{\gamma}} y^{\alpha + \gamma} = \frac{1-z^{-\alpha}}{(z-1)^\gamma} =: g(z).
\end{align*}
	Furthermore, since $\alpha + \gamma < 1$ we get that
	\begin{align*}
		\lim_{z\to 1} g(z) = \lim_{z\to 1} \frac{1-z^{-\alpha}}{(z-1)^\gamma} = \lim_{z\to 1} \frac{1+\alpha z^{-\alpha-1}}{\gamma(z-1)^{\gamma-1}} = 0,
	\end{align*}
	and
	\begin{align*}
	\lim_{z\to \infty}  g(z) = 0.
	\end{align*}
	Moreover, for $2 \leq z \leq \infty$ we get the upper bound
	\begin{align*}
		0 \leq g(z) \leq \frac{1-z^{-\alpha}}{(z-1)^\gamma} < \frac{1}{1} = 1,
	\end{align*}
	and for $1< z < 2$ we have that
	\begin{align*}
		g(z) = \frac{z^{\alpha}-1}{(z-1)^\gamma z^\alpha} < \frac{z-1}{(z-1)^\gamma (z-1)^\alpha} = (z-1)^{1-\gamma-\alpha} \leq 1.
	\end{align*}
	This shows inequality \eqref{eq:inequalityAlpha} which then implies for $0 < \gamma < H$ that
\begin{align*}
	f_t(\theta) - f_t(\theta^{\prime}) &= \left(\frac{t}{\theta} (t-\theta)\right)^{H-\frac{1}{2}}-\left(\frac{t}{\theta^{\prime}} (t-\theta^{\prime})\right)^{H-\frac{1}{2}} \\
	&\lesssim \left(\frac{t}{\theta}(t-\theta)\right)^{H-\frac{1}{2} -\gamma}t^{2\gamma }\frac{(\theta-\theta^{\prime})^{\gamma }}{(\theta \theta^{\prime})^{\gamma }} \lesssim \left(t-\theta \right)^{H-\frac{1}{2}-\gamma} \frac{(\theta -\theta^{\prime})^{\gamma}}{(\theta \theta^\prime)^\gamma}.
\end{align*}
Further, 
\begin{align*}
	g_{t}(\theta )-g_{t}(\theta^{\prime}) &= \int_{\theta }^{t}\frac{f_{u}(\theta)-f_{u}(\theta^{\prime})}{u}du-\int_{\theta^{\prime}}^{\theta }\frac{f_{u}(\theta^{\prime})}{u}du \\
	&\leq \int_{\theta }^{t}\frac{f_{u}(\theta )-f_{u}(\theta^{\prime})}{u}du \\
	&\lesssim \frac{(\theta -\theta^{\prime})^{\gamma }}{(\theta \theta^{\prime})^{\gamma }}\int_{\theta }^{t}\frac{(u-\theta )^{H-\frac{1}{2}-\gamma }}{u}du \\
	&\leq \frac{(\theta -\theta^{\prime})^{\gamma }}{(\theta \theta^{\prime})^{\gamma }} \theta^{H-\frac{1}{2}-\gamma }\int_{1}^{\infty }\frac{(v-1)^{H-\frac{1}{2} -\gamma }}{v} dv \\
	&\lesssim \frac{(\theta-\theta^{\prime})^{\gamma }}{(\theta \theta^{\prime})^{\gamma }}\theta^{H-\frac{1}{2}-\gamma } \\
	&\lesssim \frac{(\theta -\theta^{\prime})^{\gamma }}{(\theta \theta^{\prime})^{\gamma }}\theta^{H-\frac{1}{2}-\gamma }(t-\theta )^{H-\frac{1}{2} - \gamma}.
\end{align*}

Consequently, we get for $\gamma\in (0,H)$, $0<\theta^{\prime}<\theta<t\leq T$, that
\begin{align*}
	K_{H}(t,\theta)-K_{H}(t,\theta^{\prime})\leq C \cdot c_H \frac{(\theta-\theta^{\prime})^{\gamma }}{(\theta \theta^{\prime})^{\gamma }}\theta^{H-\frac{1}{2} - \gamma }(t-\theta )^{H-\frac{1}{2}-\gamma },
\end{align*}
	where $C>0$ is a constant merely depending on $T$. Thus 
\begin{align*}
	\int_{0}^{t}\int_{0}^{\theta }&\frac{(K_{H}(t,\theta)-K_{H}(t,\theta^{\prime}))^{2}}{|\theta -\theta^{\prime}|^{1+2\beta }}d\theta^{\prime}d\theta \\
	&\lesssim \int_{0}^{t}\int_{0}^{\theta}\frac{|\theta -\theta^{\prime}|^{-1-2\beta +2\gamma }}{(\theta \theta^{\prime})^{2\gamma }}\theta^{2H-1-2\gamma }(t-\theta)^{2H-1-2\gamma }d\theta^{\prime}d\theta \\
	& = \int_{0}^{t}\theta^{2H-1-4\gamma }(t-\theta )^{2H-1-2\gamma }\int_{0}^{\theta}|\theta -\theta^{\prime}|^{-1-2\beta +2\gamma }(\theta^{\prime})^{-2\gamma}d\theta^{\prime}d\theta \\
	&= \int_{0}^{t}\theta^{2H-1-4\gamma -2\beta}(t-\theta )^{2H-1-2\gamma }\frac{\Gamma(-2\beta +2\gamma )\Gamma (-2\gamma +1)}{\Gamma (-2\beta +1)} d\theta \\
	&\lesssim \int_{0}^{t}\theta^{2H-1-4\gamma -2\beta }(t-\theta )^{2H-1-2\gamma }d\theta \\
	&= \frac{\Gamma (2H-2\gamma )\Gamma (2H-4\gamma -2\beta )}{\Gamma(4H-6\gamma -2\beta )}t^{4H-6\gamma -2\beta -1}<\infty,
\end{align*}
	for sufficiently small $\gamma $ and $\beta$. On the other hand, we have that 
\begin{align*}
	\int_{0}^{t}\int_{\theta}^{t}&\frac{(K_{H}(t,\theta)-K_{H}(t,\theta^{\prime}))^{2}}{|\theta -\theta^{\prime}|^{1+2\beta }}d\theta^{\prime}d\theta \\
	&\lesssim \int_{0}^{t}\theta^{2H-1-4\gamma }(t-\theta)^{2H-1-2\gamma }\int_{\theta}^{t}\frac{|\theta -\theta^{\prime}|^{-1-2\beta +2\gamma }}{(\theta^{\prime})^{2\gamma }}d\theta^{\prime}d\theta \\
	&\leq \int_{0}^{t}\theta^{2H-1-6\gamma }(t-\theta)^{2H-1-2\gamma } \int_{\theta}^t |\theta -\theta^{\prime}|^{-1-2\beta +2\gamma } d\theta^{\prime}d\theta \\
	&\lesssim \int_{0}^{t}\theta^{2H-1-6\gamma }(t-\theta )^{2H-1 -2\beta }d\theta \lesssim t^{4H-6\gamma -2\beta -1}.
\end{align*}
	Therefore,
\begin{align*}
	\int_{0}^{t}\int_{0}^{t}\frac{(K_{H}(t,\theta )-K_{H}(t,\theta^{\prime}))^{2}}{|\theta -\theta^{\prime}|^{1+2\beta }}d\theta^{\prime}d\theta <\infty .
\end{align*}
\end{proof}

\begin{lemma}\label{lem:iterativeInt} 
	Let $H \in \left( 0, \frac{1}{2} \right)$, $0 \leq \theta<t \leq T$ and $(\varepsilon_1,\dots, \varepsilon_{n})\in \{0,1\}^{n}$ be fixed. Assume $w_j+\left(H-\frac{1}{2}-\gamma\right) \varepsilon_j>-1$ for all $j=1,\dots,n$. Then there exists a finite constant $C_{H,T}>0$ depending only on $H$ and $T$ such that for $\gamma \in (0,H)$
\begin{align*}
	\int_{\Delta_{\theta,t}^{n}} &\prod_{j=1}^{n} (K_H(s_j,\theta) - K_H(s_j,\theta^{\prime}))^{\varepsilon_j} |s_j-s_{j-1}|^{w_j} ds \\
	&\leq \left( C_{H,T} \left(\frac{\theta-\theta^{\prime}}{\theta \theta^{\prime}}\right)^\gamma \theta^{\left( H-\frac{1}{2} - \gamma\right)} \right)^{\sum_{j=1}^n \varepsilon_j} \Pi_{\gamma}(n) \,(t-\theta)^{\sum_{j=1}^n \left( w_j + \left( H-\frac{1}{2}-\gamma\right) \varepsilon_j \right) +n},
\end{align*}
where 
\begin{align}\label{VI_Pi}
	\Pi_{\gamma}(m):= \frac{\prod_{j=1}^n\Gamma (w_j +1)}{\Gamma\left(\sum_{j=1}^{n} w_j + \left(H-\frac{1}{2}-\gamma
\right)\sum_{j=1}^{n} \varepsilon_j + n \right)}.
\end{align}
\end{lemma}
\begin{proof}
	Recall, that for given exponents $a,b>-1$ and some fixed $s_{j+1}>s_j$ we have 
\begin{equation*}
	\int_{\theta}^{s_{j+1}} (s_{j+1}-s_j)^{a} (s_j-\theta)^b ds_j =\frac{\Gamma\left( a+1\right)\Gamma \left( b+1\right)}{\Gamma \left( a+b+2\right)}
(s_{j+1}-\theta)^{a+b+1}.
\end{equation*}
	Due to \Cref{lem:doubleint} we have that for every $\gamma\in (0,H)$, $0<\theta^{\prime}<\theta<s_j\leq T$, 
\begin{align*}
	K_{H}(s_j,\theta )-K_{H}(s_j,\theta^{\prime})\leq C_{H,T} \frac{(\theta-\theta^{\prime})^{\gamma }}{(\theta \theta^{\prime})^{\gamma }}\theta^{H-
\frac{1}{2}-\gamma }(s_j-\theta )^{H-\frac{1}{2}-\gamma },
\end{align*}
	for $C_{H,T} := C \cdot c_H$, where $c_H$ is the constant in \eqref{eq:kernel} and $C>0$ is some constant merely depending on $T$. Consequently, we get that 
\begin{align*}
	\int_{\theta}^{s_2} &|K_H(s_1,\theta)-K_H(s_1,\theta^{\prime})|^{\varepsilon_1} |s_2-s_1|^{w_2}|s_1-\theta|^{w_1}ds_1 \\
	&\leq C_{H,T}^{\varepsilon_1} \frac{(\theta-\theta^{\prime})^{\gamma\varepsilon_1 }}{(\theta \theta^{\prime})^{\gamma\varepsilon_1}}\theta^{\left(H-\frac{1}{2}-\gamma\right)\varepsilon_1}\int_{\theta}^{s_2}|s_2-s_1|^{w_2}|s_1-\theta|^{w_1+\left(H-\frac{1}{2}-\gamma\right)\varepsilon_1}ds_1 \\
	&= C_{H,T}^{\varepsilon_1} \frac{(\theta-\theta^{\prime})^{\gamma\varepsilon_1 }}{(\theta \theta^{\prime})^{\gamma\varepsilon_1}}\theta^{\left(H-\frac{1}{2}-\gamma\right)\varepsilon_1 } \frac{\Gamma\left(\hat{w}_1\right)\Gamma\left(\hat{w}_2\right)}{\Gamma\left(\hat{w}_1+\hat{w}_2\right)}(s_2-\theta)^{w_1+w_2+\left(H-\frac{1}{2}-\gamma\right)\varepsilon_1+1},
\end{align*}
where 
\begin{equation*}
\hat{w}_1 := w_1+\left(H-\frac{1}{2}-\gamma\right)\varepsilon_1+1, \quad 
\hat{w}_2:=w_2+1.
\end{equation*}
Noting that
\begin{align*}
	\prod_{j=1}^{n-1} \frac{\Gamma \left(\sum_{l=1}^{j} w_l + \left(H-\frac{1}{2}-\gamma \right)\sum_{l=1}^{j} \varepsilon_l + j\right)\Gamma \left( w_{j+1}+1\right)}{\Gamma \left( \sum_{l=1}^{j+1} w_l + \left(H-\frac{1}{2}-\gamma \right)\sum_{l=1}^{j} \varepsilon_l + j + 1 \right)} \leq \Pi_\gamma(n).
\end{align*}
and iterative integration yields the desired formula.
\end{proof}

Finally, we are able to give the proof of \Cref{prop:mainestimate1}.

\begin{proof}[Proof of \Cref{prop:mainestimate1}]
	The integration by parts formula \eqref{eq:integrationByParts} yields that
\begin{equation*}
	\int_{\Delta _{\theta ,t}^n}\left( \prod_{j=1}^n D^{\alpha_{j}}f_{j}(s_{j},\Bhat_{s_{j}})\varkappa _{j}(s_{j})\right) ds=\int_{\Rbb^{dn}}\Lambda_{\alpha}^{\varkappa f}(\theta ,t,z)dz.
\end{equation*}
	Taking the expectation and applying \Cref{thm:mainThmLocalTime} we get that
\begin{align*}
	&\Eabs{\int_{\Delta _{\theta ,t}^n}\left( \prod_{j=1}^n D^{\alpha_{j}}f_{j}(s_{j}, \Bhat_{s_{j}})\varkappa _{j}(s_{j})\right) ds}  \\
	&\quad \leq \frac{T^{\frac{\vert \alpha \vert}{12}}}{\sqrt{2\pi}^{dn}} (\Psi _{\alpha }^{\varkappa }(\theta, t, H, d))^{\frac{1}{2}} \prod_{j=1}^n \left\Vert f_{j}\right\Vert_{L^{1}(\mathbb{R}^{d};L^{\infty}([0,T]))},
\end{align*}
	where
\begin{align*}
	&\Psi _{\alpha }^{\varkappa }(\theta, t, H, d) := \left( \prod_{k=1}^{d}\sqrt{(2\left\vert \alpha ^{(k)}\right\vert )!} \right) \\
	&\quad \times\sum_{\sigma \in \Scal(n,n)}\int_{\Delta_{0,t}^{2n}} \vert \Delta s \vert^{- H \left( 1 + \alpha_{[\sigma(\Delta)]} \right)} \prod_{j=1}^{2n} (K_{H_m}(s_{j},\theta)-K_{H_m}(s_{j},\theta' ))^{\varepsilon _{\lbrack \sigma (j)]}}ds.
\end{align*}
	Under the assumption $-\sum_{k=1}^{d}H_k (1+ \alpha _{\lbrack \sigma (j)]}^{(k)} + \alpha _{\lbrack \sigma (j-1)]}^{(k)})+(H_m - \frac{1}{2}-\gamma_m)\varepsilon _{\lbrack \sigma (j)]}>-1$ for all $j=1,...,2n$, we can apply \Cref{lem:iterativeInt} and thus get
\begin{align*}
	&\Psi _{\alpha }^{\varkappa }(\theta, t, H, d) \\
	&\quad \leq \sum_{\sigma \in \Scal(n,n)} \left( C_T \left( \frac{\theta -\theta'}{\theta \theta'}\right) ^{\gamma_m} \theta^{(H_m-\frac{1}{2}-\gamma_m)} \right)^{\sum_{j=1}^{2n} \varepsilon_{[\sigma(j)]}} \Pi_{\gamma }(2n) \\
 	&\qquad \times \left( \prod_{k=1}^{d}\sqrt{(2\left\vert \alpha ^{(k)}\right\vert )!} \right) (t-\theta )^{-\sum_{k=1}^{d}H_k \left(2n+4\left\vert \alpha^{(k)}\right\vert \right)+(H_m-\frac{1}{2}-\gamma_m)\sum_{j=1}^{2n}\varepsilon _{\lbrack \sigma (j)]}+2n},
\end{align*}
	where $\Pi_{\gamma }(2n)$ is defined as in \eqref{VI_Pi}. We define the constant $K_{d,H}$ by
	\begin{align}\label{eq:constantKdh}
		K_{d,H} := 2 \sup_{j= 1, \dots ,2n} \Gamma \left(1-\sum_{k=1}^{d} H_k \left(1+ \alpha_{\lbrack \sigma (j)]}^{(k)} + \alpha_{\lbrack \sigma (j-1)]}^{(k)}\right)\right)
	\end{align}
	and thus an upper bound of $\Pi_\gamma(2n)$ is given by
\begin{align*}
	\Pi_{\gamma }(2n) \leq \frac{K_{d,H}^{2n}}{2^{2n} \Gamma\left(-\sum_{k=1}^{d} H_k\left(2n+4\left\vert \alpha ^{(k)}\right\vert \right)+\left(H_m - \frac{1}{2}-\gamma_m \right)\sum_{j=1}^{2n}\varepsilon_{\lbrack \sigma(j)]}+2n \right)}.
\end{align*}
	Note that $\sum_{j=1}^{2n}\varepsilon_{\lbrack \sigma(j)]} = 2\sum_{j=1}^{n}\varepsilon_{j}$ and 
	\begin{align*}
		\# \Scal(n,n) = \dbinom{2n}{n} = \frac{2^{2n}}{\sqrt{\pi}} \frac{\Gamma\left(n+\frac{1}{2} \right)}{\Gamma(n+1)} \leq 2^{2n}.
	\end{align*}
	Hence, it follows that
\begin{align*}
	&(\Psi _{k}^{\varkappa }(\theta, t, H, d))^{\frac{1}{2}} \\
	&\quad \leq K_{d,H}^{n} \left( C_T \left( \frac{\theta -\theta'}{\theta \theta'}\right)^{\gamma_m }\theta^{(H_m-\frac{1}{2}-\gamma_m)}\right)^{\sum_{j=1}^n \varepsilon_{j}} \\
	&\qquad \times \frac{\left(\prod_{k=1}^{d}\left(2\left\vert \alpha^{(k)}\right\vert \right)!\right)^{\frac{1}{4}}(t-\theta )^{-\sum_{k=1}^{d}H_k \left(n+2\left\vert \alpha^{(k)}\right\vert \right)+\left(H_m-\frac{1}{2}-\gamma_m \right)\sum_{j=1}^n \varepsilon_{j}+n}}{\Gamma\left(2n - \sum_{k=1}^{d} H_k \left(2n+4\left\vert \alpha^{(k)}\right\vert \right)+2\left(H_m - \frac{1}{2}-\gamma_m \right)\sum_{j=1}^n \varepsilon _{j} \right)^{\frac{1}{2}}},
\end{align*}
\end{proof}

\begin{proposition}\label{prop:mainestimate2} 
	Let the functions $f$ and $\varkappa $ be defined as in \eqref{f} and \eqref{kappa}, respectively. Let $0 \leq \theta <t \leq T$ and
\begin{equation*}
	\varkappa _{j}(s)=(K_{H_m}(s,\theta ))^{\varepsilon _{j}},\theta <s<t,
\end{equation*}
for every $j=1, \dots, n$ with $(\varepsilon_1, \dots, \varepsilon_n)\in\{0,1\}^n$. Let $\alpha \in (\Nbb_{0}^{d})^n$ be a multi-index and suppose that there exists $\delta $ such that 
\begin{equation*}
	-\sum_{k=1}^{d} H_k \left(1+2\alpha _{j}^{(k)}\right)+\left(H_m-\frac{1}{2}\right) \geq \delta >-1
\end{equation*}
for all $j= 1, \dots, n$ and $d\geq 1$. Then there exist constants $C_T$ (depending on $T$) and $K_{d,H}$ (depending on $d$ and $H$) such that for any $0 \leq \theta <t \leq T$ we have 
\begin{align*}
	&\Eabs{\int_{\Delta _{\theta ,t}^{n}}\left( \prod_{j=1}^{n} D^{\alpha_{j}}f_{j}(s_{j}, \Bhat_{s_{j}})\varkappa _{j}(s_{j})\right) ds}  \\
	&~ \leq \frac{K_{d,H}^{n} \cdot T^{\frac{\vert \alpha \vert}{12}}}{\sqrt{2\pi}^{dn}} \left( C_T \theta^{(H_m-\frac{1}{2})}\right)^{\sum_{j=1}^{n} \varepsilon_j} \prod_{j=1}^n \left\Vert f_{j}(\cdot ,z_{j})\right\Vert _{L^{1}(\mathbb{R}^{d};L^{\infty }([0,T]))} \\
	&\quad \times \frac{\left(\prod_{k=1}^{d} \left(2\left\vert \alpha ^{(k)}\right\vert\right)!\right)^{\frac{1}{4}}(t-\theta )^{-\sum_{k=1}^{d} H_k \left(n+2\left\vert \alpha^{(k)}\right\vert \right)+ \left(H_m-\frac{1}{2}\right)\sum_{j=1}^n \varepsilon _{j}+n}}{\Gamma(2n-\sum_{k=1}^{d}H_k(2n+4\left\vert \alpha^{(k)}\right\vert )+2(H_m - \frac{1}{2})\sum_{j=1}^n \varepsilon _{j})^{\frac{1}{2}}}.
\end{align*}
\end{proposition}

	The proof of \Cref{prop:mainestimate2} is similar to the one of \Cref{prop:mainestimate1} by using the subsequent lemma instead of \Cref{lem:iterativeInt} and thus it is omitted in this manuscript.

\begin{lemma}\label{lem:iterativeInt2}
	Let $H \in \left(0,\frac{1}{2}\right)$, $0 \leq \theta<t \leq T$ and $(\varepsilon_1,\dots, \varepsilon_{n})\in \{0,1\}^{n}$ be fixed. Assume $w_j+\left(H-\frac{1}{2}\right) \varepsilon_j>-1$ for all $j=1,\dots,n$. Then there exists a finite constant $C_{H,T}>0$ depending only on $H$ and $T$ such that 
\begin{align*}
	\int_{\Delta_{\theta,t}^{n}} &\prod_{j=1}^{n} (K_H(s_j,\theta))^{\varepsilon_j} |s_j-s_{j-1}|^{w_j} ds \\
&\leq \left( C_{H,T} \theta^{\left( H-\frac{1}{2}\right)} \right)^{\sum_{j=1}^n \varepsilon_j} \Pi_0(n) \, (t-\theta)^{\sum_{j=1}^n \left( w_j + \left( H-\frac{1}{2}\right) \varepsilon_j \right) + n},
\end{align*}
where $\Pi_0$ is defined in \eqref{VI_Pi}.
\end{lemma}
\begin{proof}
	Using similar arguments as in the proof of \Cref{lem:doubleint} we get the following estimate 
\begin{equation*}
|K_H(s_j,\theta)| \leq C_{H,T} |s_j-\theta|^{H-\frac{1}{2}}\theta^{H-\frac{1}{2}}
\end{equation*}
for every $0<\theta<s_j<T$ and $C_{H,T}:= C \cdot c_H$, where $c_H$ is the constant in \eqref{eq:kernel} and $C>0$ is some constant merely depending on $T$. Thus,
\begin{align*}
	&\int_{\theta}^{s_2} (K_H(s_1,\theta))^{\varepsilon_1}|s_2-s_1|^{w_2}|s_1-\theta|^{w_1}ds_1 \\
	&\quad \leq C_{H,T}^{\varepsilon_1} \, \theta^{\left(H-\frac{1}{2}\right)\varepsilon_1} \int_{\theta}^{s_2} |s_2-s_1|^{w_2} |s_1-\theta|^{w_1+\left(H-\frac{1}{2}\right)\varepsilon_1}ds_1 \\
	&\quad = C_{H,T}^{\varepsilon_1} \, \theta^{\left(H-\frac{1}{2}\right)\varepsilon_1} \frac{\Gamma\left(w_1 +\left(H-\frac{1}{2}\right)\varepsilon_1+1\right)\Gamma\left(w_2+1\right)}{\Gamma\left(w_1 + w_2 + \left(H-\frac{1}{2}\right)\varepsilon_1+2\right)}(s_2-\theta)^{w_1+w_2+\left(H-\frac{1}{2}\right)\varepsilon_1+1}.
\end{align*}
Proceeding similar to the proof of \Cref{lem:iterativeInt} yields the desired estimate.
\end{proof}

\begin{remark}\label{rem:boundFaculty} 
	Note that
\begin{equation*}
	\prod_{k=1}^{d} \left(2\left\vert \alpha ^{(k)}\right\vert \right)! \leq \sqrt{2\pi}^d e^{\frac{\vert \alpha \vert}{2}} \frac{\Gamma\left( \frac{5}{2} \vert \alpha \vert +1\right)}{\sqrt{5 \pi \vert \alpha \vert}}.
\end{equation*}
Indeed, since for $n \geq 1$ sufficiently large we have by Stirling's formula that
	\begin{align*}
		\sqrt{2\pi n} \left( \frac{n}{e} \right)^n \leq n! \leq e^{\frac{1}{12n}} \sqrt{2\pi n} \left( \frac{n}{e} \right)^n,
	\end{align*}
	we get by assuming without loss of generality that $\vert \alpha^{(k)} \vert \geq 1$ for all $1 \leq k \leq d$, that
	\begin{align*}
		\prod_{k=1}^d \left( 2 \vert \alpha^{(k)} \vert \right)! &\leq \prod_{k=1}^d e^{\frac{1}{24 \vert \alpha^{(k)} \vert}} \sqrt{4 \pi \vert \alpha^{(k)} \vert} \left( \frac{2 \vert \alpha^{(k)} \vert}{e} \right)^{2 \vert \alpha^{(k)} \vert} \\
		&\leq e^{\frac{d}{24}} \sqrt{\frac{8}{5}\pi}^d \prod_{k=1}^d \left(\frac{5}{2} \vert \alpha^{(k)} \vert \right)^{\frac{\vert \alpha^{(k)} \vert}{2}} \left( \frac{\frac{5}{2} \vert \alpha^{(k)} \vert}{e} \right)^{2 \vert \alpha^{(k)} \vert} \\
		&\leq \sqrt{2\pi}^d \prod_{k=1}^d  \left( \frac{\frac{5}{2} \vert \alpha \vert}{e} \right)^{\frac{5}{2} \vert \alpha^{(k)} \vert} e^{\frac{\vert \alpha^{(k)} \vert}{2}} \\
		&\leq \sqrt{2\pi}^d e^{\frac{\vert \alpha \vert}{2}} \left( \frac{\frac{5}{2} \vert \alpha \vert}{e} \right)^{\frac{5}{2} \vert \alpha \vert} \leq \sqrt{2\pi}^d e^{\frac{\vert \alpha \vert}{2}} \frac{\Gamma\left( \frac{5}{2} \vert \alpha \vert +1\right)}{\sqrt{5 \pi \vert \alpha \vert}}.
	\end{align*}
\end{remark}

\section{Technical Results}

The following technical result can be found in \cite{LiWei}.

\begin{lemma}\label{lem:LiWei} 
	Assume that $X_{1},...,X_{n}$ are real centered jointly Gaussian random variables, and $\Sigma =(\Ebb[X_{j}X_{k}])_{1\leq j,k\leq n}$ is the covariance matrix, then
\begin{align*}
	\Eabs{ X_{1}\right\vert ...\left\vert X_{n}} \leq \sqrt{\perm{\Sigma}},
\end{align*}
	where $\perm{A}$ is the permanent of a matrix $A=(a_{ij})_{1\leq i,j\leq n}$ defined by
\begin{align*}
	\perm{A}=\sum_{\pi \in \Scal_{n}} \prod_{j=1}^{n}a _{j,\pi (j)}
\end{align*}
	for the symmetric group $\Scal_{n}$.
\end{lemma}

The next lemma corresponds to \cite[Lemma 2]{CD}:

\begin{lemma}\label{lem:CD} 
	Let $Z_{1},...,Z_{n}$ be mean zero Gaussian random variables which are linearly independent. Then for any measurable function $g:\mathbb{R}\longrightarrow \mathbb{R}_{+}$ we have that
\begin{align*}
	&\int_{\mathbb{R}^{n}}g(v_{1}) e^{-\frac{1}{2} \Var{\sum_{j=1}^{n}v_{j}Z_{j}}} dv_{1}...dv_{n} =\frac{(2\pi )^{\frac{n-1}{2}}}{(\det \Cov{Z_{1},...,Z_{n}})^{\frac{1}{2}}}\int_{\mathbb{R}} g\left(\frac{v}{\sigma _{1}}\right) e^{-\frac{v^2}{2}} dv,
\end{align*}
	where $\sigma _{1}^{2}:=\Var{Z_{1}\left\vert Z_{2},...,Z_{n}\right.}$.
\end{lemma}

\begin{remark}
	Note that here linearly independence is meant in the sense that $\det \Cov{Z_{1},...,Z_{n}} \neq 0$.
\end{remark}

\begin{lemma}\label{lem:prodSumInterchange}
	Let $a \in \ell^p$, $1 \leq p < \infty$. Then, for every $n\geq 1$ and $d \geq 1$
	\begin{align}\label{eq:prodSumInterchange}
		\sum_{k_1,\dots, k_n = 1}^d \prod_{j=1}^n a_{k_j} = \left( \sum_{k=1}^d a_k \right)^n,
	\end{align}
	and 
	\begin{align}\label{eq:prodSumLimit}
		\lim_{d \to \infty} \sum_{k_1,\dots, k_n}^d \prod_{j=1}^n \vert a_{k_j} \vert^p = \left( \Vert a \Vert_{\ell^p} \right)^n.
	\end{align}
\end{lemma}
\begin{proof}
	We proof equation \eqref{eq:prodSumInterchange} by induction. For $n=1$ the result holds. Therefore we assume that \eqref{eq:prodSumInterchange} holds for $n$ and we show that it also holds for $n+1$. Thus, we get by the induction hypothesis that
	\begin{align*}
		\sum_{k_1,\dots, k_{n+1} = 1}^d \prod_{j=1}^{n+1} a_{k_j} &= \sum_{k_{n+1}=1}^d a_{k_{n+1}} \left( \sum_{k_1,\dots, k_n = 1}^d \prod_{j=1}^n a_{k_j} \right) \\
		&= \sum_{k_{n+1}=1}^d a_{k_{n+1}} \left( \sum_{k = 1}^d a_{k} \right)^n = \left( \sum_{k = 1}^d a_{k} \right)^{n+1}.
	\end{align*}
	Equation \eqref{eq:prodSumLimit} is an immediate consequence of \eqref{eq:prodSumInterchange} and the continuity of the function $f(x) = x^n$ for fixed $n \geq 1$.
\end{proof}

The subsequent lemmas are due to \cite{Anderson}.

\begin{lemma}\label{lem:determinant1}
	Let $(X_1,\dots,X_n)$ be a mean-zero Gaussian random vector. Then,
	\begin{align*}
		\det \Cov{X_1,\dots,X_n} = \Var{X_1} \Var{X_2 \vert X_1} \cdots \Var{X_n \vert X_{n-1},\dots ,X_1}.
	\end{align*}
\end{lemma}

\begin{lemma}\label{lem:determinant2}
	For any square integrable random variable $X$ and $\sigma$-algebras $\Gcal_1 \subset \Gcal_2$
	\begin{align*}
		\Var{X\vert \Gcal_1} \geq \Var{X\vert \Gcal_2}.
	\end{align*}
\end{lemma}

\bibliography{literature}

\begin{thebibliography}{10}

\bibitem{ABP18}
O.~Amine, D.~Ba{\~n}os, and F.~Proske.
\newblock {Regularity Properties of the Stochastic Flow of a Skew Fractional
  Brownian Motion}.
\newblock {\em arXiv preprint arXiv:1805.04889}, 2018.

\bibitem{ACHP}
O.~Amine, E.~Coffie, F.~Harang, and F.~Proske.
\newblock {A Bismut-Elworthy-Li Formula for Singular SDE's Driven by a
  Fractional Brownian Motion and Applications to Rough Volatility Modeling}.
\newblock {\em arXiv preprint arXiv:1805.11435}, 2018.

\bibitem{AG}
N.~Anari, L.~Gurvits, S.~Gharan, and A.~Saberi.
\newblock {Simply Exponential Approximation of the Permanent of Positive
  Semidefinite Matrices}.
\newblock In {\em {2017 IEEE 58th Annual Symposium on Foundations of Computer
  Science (FOCS)}}, pages 914--925. IEEE, 2017.

\bibitem{Anderson}
T.~W. Anderson.
\newblock An introduction to multivariate statistical analysis.
\newblock Technical report, Wiley New York, 1962.

\bibitem{BHP.17}
D.~Ba{\~n}os, H.~Haferkorn, and F.~Proske.
\newblock Strong uniqueness of singular stochastic delay equations.
\newblock {\em arXiv preprint arXiv:1707.02271}, 2017.

\bibitem{BNP}
D.~Ba{\~n}os, T.~Nilssen, and F.~Proske.
\newblock {Strong existence and higher order Fr{\'e}chet differentiability of
  stochastic flows of fractional Brownian motion driven SDE's with singular
  drift}.
\newblock {\em arXiv preprint arXiv:1511.02717v2}, 2018.

\bibitem{BOPP.17}
D.~Ba{\~n}os, S.~Ortiz-Latorre, A.~Pilipenko, and F.~Proske.
\newblock {Strong solutions of SDE's with generalized drift and
  multidimensional fractional Brownian initial noise}.
\newblock {\em arXiv preprint arXiv:1705.01616}, 2017.

\bibitem{Bauer_StrongSolutionsOfMFSDEs}
M.~Bauer, T.~Meyer-Brandis, and F.~Proske.
\newblock Strong solutions of mean-field stochastic differential equations with
  irregular drift.
\newblock {\em Electronic Journal of Probability}, 23, 2018.

\bibitem{bogachev2010differentiable}
V.~Bogachev.
\newblock {\em {Differentiable measures and the Malliavin calculus}}.
\newblock Number 164 in {Mathematical Surveys and Monographs}. AMS, 2010.

\bibitem{borovkov2017bounds}
K.~Borovkov, Y.~Mishura, A.~Novikov, and M.~Zhitlukhin.
\newblock {Bounds for expected maxima of Gaussian processes and their discrete
  approximations}.
\newblock {\em Stochastics}, 89(1):21--37, 2017.

\bibitem{CG}
R.~Catellier and M.~Gubinelli.
\newblock {Averaging along irregular curves and regularisation of ODEs}.
\newblock {\em Stochastic Processes and their Applications}, 126(8):2323--2366,
  2016.

\bibitem{CD}
J.~Cuzick and J.~DuPreez.
\newblock {Joint continuity of Gaussian local times}.
\newblock {\em The Annals of Probability}, pages 810--817, 1982.

\bibitem{DPFPR13}
G.~Da~Prato, F.~Flandoli, E.~Priola, and M.~R{\"o}ckner.
\newblock Strong uniqueness for stochastic evolution equations in hilbert
  spaces perturbed by a bounded measurable drift.
\newblock {\em The Annals of Probability}, 41(5):3306--3344, 2013.

\bibitem{daprato1992compact}
G.~{Da Prato}, P.~Malliavin, and D.~Nualart.
\newblock {Compact families of Wiener functionals}.
\newblock {\em Comptes rendus de l'Acad{\'e}mie des sciences. S{\'e}rie 1,
  Math{\'e}matique}, 315(12):1287--1291, 1992.

\bibitem{Da07}
A.~Davie.
\newblock Uniqueness of solutions of stochastic differential equations.
\newblock {\em International Mathematics Research Notices}, 2007.

\bibitem{decreusefond1999stochastic}
L.~Decreusefond.
\newblock {Stochastic analysis of the fractional Brownian motion}.
\newblock {\em Potential analysis}, 10(2):177--214, 1999.

\bibitem{Engel}
H.~Engelbert.
\newblock {On the theorem of T. Yamada and S. Watanabe}.
\newblock {\em Stochastics: An International Journal of Probability and
  Stochastic Processes}, 36(3-4):205--216, 1991.

\bibitem{FZ}
S.~Fang and T.~Zhang.
\newblock {Stochastic differential equations with non-lipschitz coefficients:
  I. Pathwise uniqueness and large deviation}.
\newblock {\em C.R. Acad. Sci. Paris Ser. I. 337}, pages 737--740, 2003.

\bibitem{Flandoli}
F.~Flandoli.
\newblock {Random Perturbation of PDEs and Fluid Dynamic Models}.
\newblock In {\em {{\'E}cole d'{\'E}t{\'e} de Probabilit{\'e}s de Saint-Flour
  XL--2010}}, volume 2015. Springer, 2011.

\bibitem{FNP.13}
F.~Flandoli, T.~Nilssen, and F.~Proske.
\newblock {Malliavin differentiability and strong solutions for a class of SDE
  in Hilbert spaces}.
\newblock {\em Preprint series: Pure mathematics http://urn. nb. no/URN: NBN:
  no-8076}, 2013.

\bibitem{GyK96}
I.~Gy{\"o}ngy and N.~Krylov.
\newblock {Existence of strong solutions for It{\^o}'s stochastic equations via
  approximations}.
\newblock {\em Probability theory and related fields}, 105(2):143--158, 1996.

\bibitem{GyM01}
I.~Gy\"ongy and T.~Mart{\'\i}nez.
\newblock On stochastic differential equations with locally unbounded drift.
\newblock {\em Czechoslovak Mathematical Journal}, 51(4):763--783, 2001.

\bibitem{HaaPros.14}
S.~Haadem and F.~Proske.
\newblock {On the construction and Malliavin differentiability of solutions of
  Levy noise driven SDE's with singular coefficients}.
\newblock {\em Journal of functional analysis}, 266(8):5321--5359, 2014.

\bibitem{KR05}
N.~Krylov and M.~R\"ockner.
\newblock Strong solutions of stochastic equations with singular time dependent
  drift.
\newblock {\em Probability theory and related fields}, 131(2):154--196, 2005.

\bibitem{LR}
G.~Leha and G.~Ritter.
\newblock {On solutions to stochastic differential equations with discontinuous
  drift in Hilbert space}.
\newblock {\em Mathematische Annalen}, 270(1):109--123, 1985.

\bibitem{LiWei}
W.~Li and A.~Wei.
\newblock {A Gaussian inequality for expected absolute products}.
\newblock {\em Journal of Theoretical Probability}, 25(1):92--99, 2012.

\bibitem{MenoukeuMeyerBrandisNilssenProskeZhang_VariationalApproachToTheConstructionOfStrongSolutions}
O.~Menoukeu-Pamen, T.~Meyer-Brandis, T.~Nilssen, F.~Proske, and T.~Zhang.
\newblock {A variational approach to the construction and Malliavin
  differentiability of strong solutions of SDE's}.
\newblock {\em Mathematische Annalen}, 357(2):761--799, 2013.

\bibitem{MeyerBrandisProske_OnTheExistenceOfStrongSolutionsOfLevyNoiseSDEs}
T.~Meyer-Brandis and F.~Proske.
\newblock {On the existence and explicit representability of strong solutions
  of L{\'e}vy noise driven SDE's with irregular coefficients}.
\newblock {\em Commun. Math. Sci.}, 4(1):129--154, 03 2006.

\bibitem{MeyerBrandisProske_ConstructionOfStrongSolutionsOfSDEs}
T.~Meyer-Brandis and F.~Proske.
\newblock {Construction of strong solutions of SDE's via Malliavin calculus}.
\newblock {\em Journal of Functional Analysis}, 258(11):3922--3953, 2010.

\bibitem{MohammedNilssenProske_Sobolev}
S.~Mohammed, T.~Nilssen, and F.~Proske.
\newblock {Sobolev differentiable stochastic flows for SDEs with singular
  coefficients: Applications to the transport equation}.
\newblock {\em The Annals of Probability}, 43(3):1535--1576, 2015.

\bibitem{Moser}
J.~Moser.
\newblock A rapidly convergent iteration method and nonlinear differential
  equations.
\newblock {\em Uspekhi Matematicheskikh Nauk}, 23(4):179--238, 1968.

\bibitem{Nualart_MalliavinCalculus}
D.~Nualart.
\newblock {\em {The Malliavin calculus and related topics}}.
\newblock {Probability and its applications}. Springer, second edition, 2006.

\bibitem{nualart2002regularization}
D.~Nualart and Y.~Ouknine.
\newblock {Regularization of differential equations by fractional noise}.
\newblock {\em Stochastic Processes and their Applications}, 102(1):103--116,
  2002.

\bibitem{oldham1974fractional}
K.~Oldham and J.~Spanier.
\newblock {\em {The fractional calculus theory and applications of
  differentiation and integration to arbitrary order}}, volume 111.
\newblock Elsevier, 1974.

\bibitem{pitt1978local}
L.~D. Pitt.
\newblock {Local times for Gaussian vector fields}.
\newblock {\em Indiana University Mathematics Journal}, 27(2):309--330, 1978.

\bibitem{PT}
A.~Pylypenko and M.~Tantsyura.
\newblock Limit theorem for countable systems of stochastic differential
  equations.
\newblock {\em Ukrainian Mathematical Journal}, 68(10):1591--1619, 2017.

\bibitem{samko1993fractional}
S.~Samko, A.~Kilbas, and O.~Marichev.
\newblock {\em {Fractional integrals and derivatives: theory and
  applications}}.
\newblock Gordon and Breach Science Publishers, 1993.

\bibitem{Snitzman_TopicsInPropagationOfChaos}
A.-S. Sznitman.
\newblock {Topics in propagation of chaos}.
\newblock In {\em {Ecole d'{\'E}t{\'e} de Probabilit{\'e}s de Saint-Flour XIX
  --- 1989}}, volume 1464 of {\em {Lecture Notes in Mathematics}}, chapter~3,
  pages 165--251. Springer, 1991.

\bibitem{Ver79}
A.~Veretennikov.
\newblock On the strong solutions of stochastic differential equations.
\newblock {\em Theory of Probability and its Applications}, 24(2):354, 1979.

\bibitem{YW71}
T.~Yamada and S.~Watanabe.
\newblock On the uniqueness of solutions of stochastic differential equations.
\newblock {\em Journal of Mathematics of Kyoto University}, 11(1):155--167,
  1971.

\bibitem{Zvon74}
A.~Zvonkin.
\newblock A transformation of the phase space of a diffusion process that
  removes the drift.
\newblock {\em Mathematics of the USSR-Sbornik}, 22(1):129, 1974.

\end{thebibliography}
\bibliographystyle{abbrv}

\end{document}